\documentclass[11pt,leqno]{amsart}
\usepackage{amsmath,amsthm,amsfonts,graphicx,amssymb,psfrag}
\usepackage[dvips, colorlinks]{hyperref}

\newtheorem{theorem}{Theorem}
\newtheorem{proposition}[theorem]{Proposition}
\newtheorem{corollary}[theorem]{Corollary}
\newtheorem{lemma}[theorem]{Lemma}
\newtheorem{sublemma}[theorem]{Sublemma}

\theoremstyle{definition}

\newtheorem{question}[theorem]{Question}

\newtheorem{definition}[theorem]{Definition}
\newtheorem{remark}[theorem]{Remark}

\DeclareMathOperator{\diam}{diam}
\DeclareMathOperator{\dist}{dist}
\DeclareMathOperator{\id}{id}
\DeclareMathOperator{\Area}{Area}
\DeclareMathOperator{\length}{length}

\DeclareMathOperator{\spt}{spt}
\renewcommand{\mod}{\operatorname{mod}}

\newcommand{\del}{\delta}
\newcommand{\gam}{\gamma}
\newcommand{\eps}{\epsilon}

\newcommand{\lam}{\lambda}

\newcommand{\bdry}{\partial_\infty}

\newcommand{\nats}{\mathbb{N}}

\newcommand{\cT}{\mathcal{T}}

\newcommand{\cH}{\mathcal{H}}
\newcommand{\cC}{\mathcal{C}}
\newcommand{\cD}{\mathcal{D}}
\newcommand{\cU}{\mathcal{U}}

\newcommand{\cR}{\mathcal{R}}

\newcommand{\ra}{\rightarrow}
\newcommand{\R}{\mathbb{R}}
\newcommand{\N}{\mathbb{N}}
\newcommand{\C}{\mathbb{C}}

\newcommand{\ba}{\mathbf{a}}
\newcommand{\Sph}{\mathbb{S}}

\newcommand{\restrict}{\begin{picture}(12,12)
                        \put(2,0){\line(1,0){8}}
                        \put(2,0){\line(0,1){8}}
                       \end{picture}}

\def\Xint#1{\mathchoice
{\XXint\displaystyle\textstyle{#1}}%
{\XXint\textstyle\scriptstyle{#1}}%
{\XXint\scriptstyle\scriptscriptstyle{#1}}%
{\XXint\scriptscriptstyle\scriptscriptstyle{#1}}%
\!\int}
\def\XXint#1#2#3{{\setbox0=\hbox{$#1{#2#3}{\int}$}
\vcenter{\hbox{$#2#3$}}\kern-.5\wd0}}

\def\dashint{\Xint-}

\numberwithin{theorem}{section}
\numberwithin{equation}{section}
\begin{document}
\title[Modulus and Poincar\'e inequalities on carpets]{Modulus and
  Poincar\'e inequalities on non-self-similar Sierpi\'nski carpets}
\date{\today}
\author{John M. Mackay}
\address{Mathematical Institute, University of Oxford, 24-29 St
  Giles, Oxford, OX1 3LB, UK}
\email{John.Mackay@maths.ox.ac.uk}
\author{Jeremy T. Tyson}
\address{Department of Mathematics, University of Illinois, 1409 West
  Green St., Urbana, IL 61801 USA}
\email{tyson@math.uiuc.edu}
\author{Kevin Wildrick}
\address{Mathematisches Institut \\ Universit\"at Bern \\ Sidlerstrasse 5, 3012 Bern, Switzerland}
\email{kevin.wildrick@math.unibe.ch}
\thanks{JMM and JTT were supported by US National Science Foundation Grant DMS-0901620, 
JMM was supported by EPSRC grant ``Geometric and analytic aspects of infinite groups'', JTT was supported by US National Science Foundation Grant DMS-1201875,
KW supported by Academy of Finland Grants 120972 and 128144, the Swiss National Science Foundation, ERC Project CG-DICE, and European Science Council Project HCAA}
\begin{abstract}
A carpet is a metric space homeomorphic to the Sierpi\'nski
carpet. We characterize, within a certain class of examples,
non-self-similar carpets supporting curve families of nontrivial
modulus and supporting Poin\-car\'e inequalities. Our results yield
new examples of compact doubling metric measure spaces supporting
Poincar\'e inequalities: these examples have no manifold points, yet
embed isometrically as subsets of Euclidean space.
\end{abstract}

\maketitle


\section{Introduction}

Metric spaces equipped with doubling measures that support Poin\-car\'e
inequalities (also known as {\it PI spaces}) are ideal environments
for first-order analysis and differential geometry \cite{hk:quasi},
\cite{ch:lipschitz}, \cite{hkst:banach}, \cite{kei:differentiable},
\cite{kz:poincare}. Extending the scope of this theory by verifying
Poincar\'e inequalities on new classes of spaces is a problem of high
interest and relevance. Previously, several classes of spaces have
been shown to support Poincar\'e inequalities:
\begin{itemize}
\item compact Riemannian manifolds or noncompact Riemannian manifolds
  satisfying suitable curvature bounds \cite{bus:isoperimetric},
\item Carnot groups and more general sub-Riemannian manifolds equipped
  with Carnot-Carath\'e\-odory (CC) metric \cite{hk:quasi},
  \cite{hei:capacity}, \cite{jer:poincare},
\item boundaries of certain hyperbolic Fuchsian buildings, see Bourdon
  and Pajot \cite{bp:buildings},
\item Laakso's spaces \cite{laa:regular},
\item linearly locally contractible manifolds with good volume growth
  \cite{sem:curves}.
\end{itemize}
These examples fall into two (overlapping) classes: examples for which
the underlying topological space is a manifold, and abstract metric
examples which admit no bi-Lipschitz embedding into any
finite-dimensional Euclidean space. Such bi-Lipschitz nonembeddability
follows from Cheeger's celebrated Rademacher-style differentiation
theorem in PI spaces, as explained in \cite[\S 14]{ch:lipschitz}.
Euclidean bi-Lipschitz nonembeddability is known, for instance, for
all nonabelian Carnot groups and other regular sub-Riemannian
manifolds, as well as for the examples of Bourdon and Pajot
\cite{bp:buildings} and Laakso \cite{laa:regular}.

The preceding dichotomy should not be taken too seriously.
Nonabelian Carnot groups equipped with the CC metric, for instance,
have underlying space which is a topological manifold, yet do not
admit any Euclidean bi-Lipschitz embedding. On the other hand, it is
certainly possible to construct Euclidean subsets with some
nonmanifold points which are PI spaces. This can be done, for
instance, by appealing to various gluing theorems for PI spaces, see
\cite[Theorem 6.15]{hk:quasi} (reproduced as Theorem
\ref{gluing-theorem} in this paper) for a general result along these
lines. However, the following question appears to have been
unaddressed in the literature until now.

\begin{question}\label{Q1}
Do there exist sets $X \subset \R^N$ (for some $N$) with {\bf no}
manifold points which are PI spaces when equipped with the Euclidean
metric and some suitable measure?
\end{question}

In connection with Question \ref{Q1} we recall the examples
constructed by Heinonen and Hanson \cite{hh:manifold}. For each $n\ge
2$, these authors construct a compact, geodesic, Ahlfors $n$-regular
PI space of topological dimension $n$ with no manifold points. They
suggest \cite[p.\ 3380]{hh:manifold}, but do not check, that their
nonmanifold example admits a bi-Lipschitz embedding into some
Euclidean space. Note that the question about embeddability of the
Heinonen--Hansen example is not resolved by Cheeger's work, since this
example admits almost everywhere unique tangent cones coinciding
with~$\R^n$.

The examples of PI spaces due to Bourdon and Pajot \cite{bp:buildings}
comprise a class of compact metric spaces arising as the Gromov
boundaries of certain hyperbolic groups acting geometrically on
Fuchsian buildings. Topologically, all of the Bourdon--Pajot examples
are homeomorphic to the Menger sponge. It is well-known that
`typical' Gromov hyperbolic groups have Menger sponge boundaries.
While examples of Gromov hyperbolic groups with Sierpi\'nski carpet
boundary do exist, it is not presently known whether any such boundary
can verify a Poincar\'e inequality in the sense of Heinonen and
Koskela.

\begin{question}\label{Q2}
Do there exist PI spaces that are homeomorphic to the
Sierpi\'nski carpet?
\end{question}

In this paper we answer Questions \ref{Q1} and \ref{Q2} affirmatively.
We identify a new class of doubling metric measure spaces supporting
Poincar\'e inequalities. Our main results are Theorem \ref{thmA} and
\ref{thmB}. Our spaces have no manifold points, indeed, they are all
homeomorphic to the Sierpi\'nski carpet. On the other hand, all of our
examples arise as explicit subsets of the plane equipped with the
Euclidean metric and the Lebesgue measure. These are the first
examples of compact subsets of Euclidean space without interior that
support Poincar\'e inequalities for the usual Lebesgue measure.

To fix notation and terminology we recall the notion of Poincar\'e
inequality on a metric measure space as introduced by Heinonen and
Koskela \cite{hk:quasi}. Let $(X,d,\mu)$ be a metric measure space,
i.e., $(X,d)$ is a metric space and $\mu$ is a Borel measure which
assigns positive and finite measure to all open balls in $X$. A Borel
function $\rho:X \ra [0,\infty]$ is an {\em upper gradient} of a
function $u:X\ra\R$ if $|u(x)-u(y)| \le \int_\gamma \rho \, ds$
whenever $\gamma$ is a rectifiable curve joining $x$ to $y$.

\begin{definition}[Heinonen--Koskela]\label{Poincare-inequality-definition}
Fix $p \ge 1$. The space $(X,d,\mu)$ is said to support a {\em
  $p$-Poincar\'e inequality} if there exist constants $C,\lam \ge 1$
so that for any continuous function $u:X \ra \R$ with upper gradient
$\rho:X \ra [0,\infty]$, the inequality
\begin{equation}\label{eq-pi}
\dashint_B \left| u - \dashint_{B} u\, d\mu \right| \, d\mu \le C
\diam(B) \left( \dashint_{\lam B} \rho^p \, d\mu \right)^{1/p}
\end{equation}
holds for every ball $B = B(x,r) \subset X$. Here we denote, for a
subset $E\subset X$ of positive measure, the mean value of a function
$u:E\to\R$ by $\dashint_{E} u\, d\mu = \frac{1}{\mu(E)} \int_E u\,d\mu$.
\end{definition}

The validity of a Poincar\'e inequality in the sense of Definition \ref{Poincare-inequality-definition} reflects strong connectivity properties of the underlying space. Roughly speaking, metric measure spaces $(X,d,\mu)$ supporting a Poincar\'e inequality have the property that any two regions are connected by a rich family of relatively short curves which are evenly distributed with respect to the background measure $\mu$. (For a more precise version of this statement, see Theorem \ref{keith-theorem-1}.) The main results of this paper are a reflection and substantiation of this general principle in the setting of a highly concrete collection of planar examples.

We now turn to a description of those examples. To each sequence $\ba = (a_1, a_2, \ldots)$ consisting of reciprocals
of odd integers strictly greater than one we associate a modified
Sierpi\'nski carpet $S_\ba$ by the following procedure. Let
$T_0=[0,1]^2$ be the unit square and let $S_{\ba,0}=T_0$. Consider the
standard tiling of $T_0$ by essentially disjoint closed congruent
subsquares of side length $a_1$. Let $\cT_1$ denote the family of such
subsquares obtained by deleting the central (concentric) subsquare,
and let $S_{\ba,1} = \cup \{ T : T \in \cT_1\}$. Again, let $\cT_2$
denote the family of essentially disjoint closed congruent subsquares
of each of the elements of $\cT_1$ with side length $a_1a_2$ obtained
by deleting the central (concentric) subsquare from each square in
$\cT_1$, and let $S_{\ba,2} = \cup \{ T : T \in \cT_2\}$. Continuing
this process, we construct a decreasing sequence of compact sets
$\{S_{\ba,m}\}_{m\ge 0}$ and an associated carpet
\begin{equation*}
S_\ba := \bigcap_{m\ge 0} S_{\ba,m}.
\end{equation*}
For example, when $\ba = (\tfrac13,\tfrac13,\tfrac13,\ldots)$, the set
$S_\ba$ is the classical Sierpi\'nski carpet $S_{1/3}$ (Figure \ref{fig:sc}).
For any $\ba$, $S_\ba$ is a compact, connected, locally connected
subset of the plane without interior and with no local cut points. By
a standard fact from topology, $S_\ba$ is homeomorphic to the
Sierpi\'nski carpet $S_{1/3}$.

\begin{figure}[b]
\begin{minipage}{225pt}
\centering\includegraphics[width=0.45\textwidth]{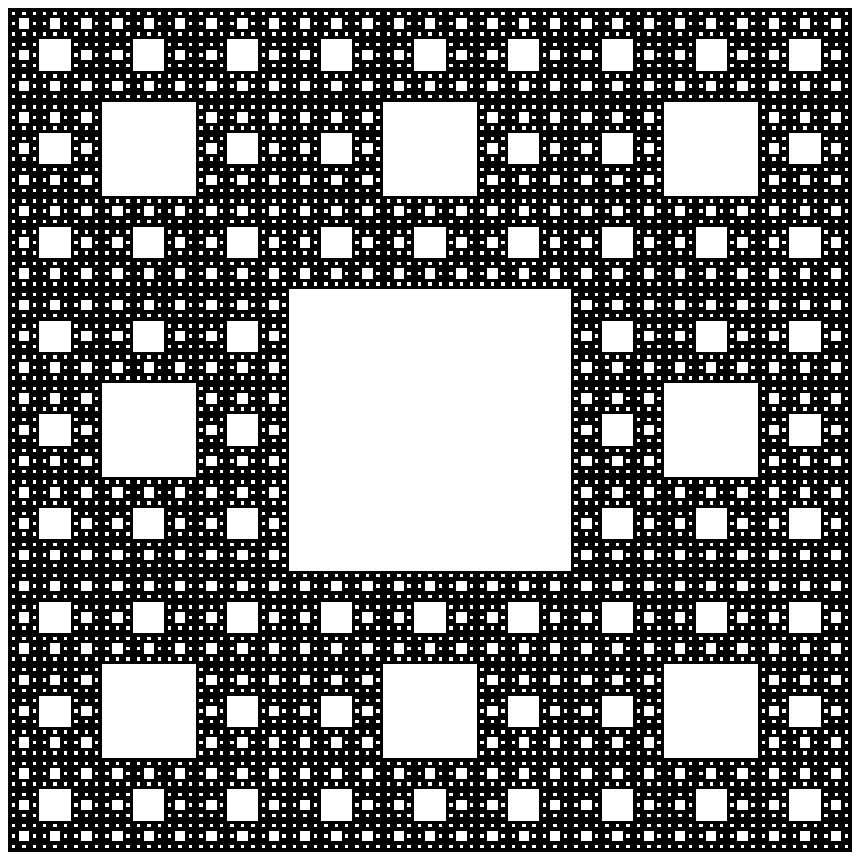}
\caption{$S_{1/3}$}\label{fig:sc}
\end{minipage}
\begin{minipage}{225pt}
\centering\includegraphics[width=0.45\textwidth]{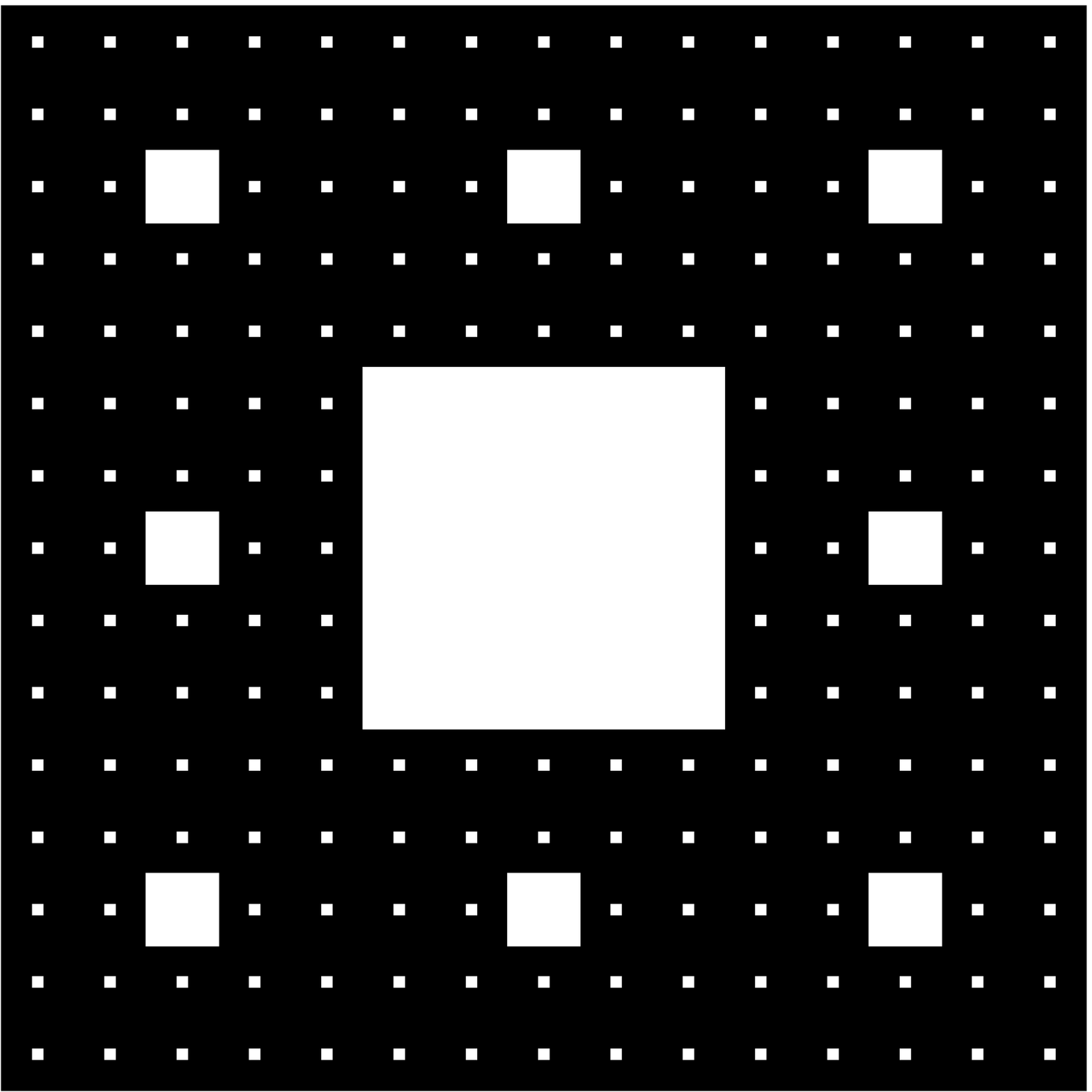}
\caption{$S_{(1/3,1/5,1/7,\ldots)}$}\label{fig:357}
\end{minipage}
\end{figure}

For each $k\in\N$, we will denote by $S_{1/(2k+1)}$ the self-similar
carpet $S_\ba$ associated to the constant sequence $\ba=(\tfrac1{2k+1},\tfrac1{2k+1},\tfrac1{2k+1},\ldots)$.
For each $k$, the carpet $S_{1/(2k+1)}$ has Hausdorff dimension equal to
\begin{equation}\label{Qk}
Q_k = \frac{\log((2k+1)^2-1)}{\log(2k+1)} =
\frac{\log(4k^2+4k)}{\log(2k+1)} < 2
\end{equation}
and is Ahlfors regular in that dimension.

The starting point for our investigations was the following well-known fact.

\begin{proposition}\label{factZ}
For each $k$, the carpet $S_{1/(2k+1)}$, equipped with Euclidean
metric and Hausdorff measure in its dimension $Q_k$, does {\bf not}
support any Poincar\'e inequality.
\end{proposition}

Several proofs for Proposition \ref{factZ} can be found in the
literature. Bourdon and Pajot \cite{bp:survey} provide an elegant
argument involving the mutual singularity of one-dimensional Lebesgue
measure and the push forward of the $Q_k$-dimensional Hausdorff measure
on $S_{1/(2k+1)}$ under projection to a coordinate axis. A different
argument involving modulus computations can be found in the monograph
by the first two authors \cite{mt:confdimsurvey}.

In this paper, we study non-self-similar carpets $S_\ba$ for which $\ba$ is not a constant sequence. We are primarily
interested in the case when $S_\ba$ has Hausdorff dimension two. It is easy to see that this holds, for instance, if the sequence $(a_m)$ of scaling ratios tends to zero, i.e., $\ba \in c_0$.
Figure \ref{fig:357} illustrates the set $S_{(1/3,1/5,1/7,\ldots)}$.

Note that the left and right hand edges of $S_\ba$ are separated by the generalized Cantor set
$$
C_\ba := S_\ba \cap \left( \big\lbrace\tfrac12\big\rbrace \times [0,1] \right).
$$
This Cantor set will have positive length if and only if the length at each stage, $\prod_{j=1}^{m} (1-a_j)$, remains bounded away from zero.  After taking logarithms, this is seen to be equivalent to $\ba \in \ell^1$.

In a similar fashion, we see that $\Area(S_\ba) = \cH^2(S_\ba)$ is positive if and only if $\Area(S_{\ba,m}) = \prod_{j=1}^{m} (1-a_j^2)$ is bounded away from zero, i.e., $\ba \in \ell^2$.

We equip $S_\ba$ with the Euclidean metric $d$ and the canonically defined measure $\mu$ arising as the weak limit of normalized Lebesgue measures on the precarpets $S_{\ba,m}$. For all $\ba$, the measure $\mu$ is doubling. Under the assumption $\ba \in \ell^2$, $\mu$ is Ahlfors $2$-regular and is comparable (with constant depending only on $||\ba||_2$) to the restriction of Lebesgue measure to $S_\ba$. For these and other facts, see Proposition \ref{properties-of-mu}.

We now state our main theorems.

\begin{theorem}\label{thmA}
The carpet $(S_\ba,d,\mu)$ supports a $1$-Poincar\'e inequality if and only if $\ba \in \ell^1$.
\end{theorem}

Under the assumption of Theorem \ref{thmA}, the $1$-modulus of all
horizontal paths in $S_\ba$ is easily seen to be positive. This fact
follows from the usual Fubini argument, since the cut set $C_\ba$ has
positive length. The difficult part of the proof of Theorem \ref{thmA}
is the verification of the $1$-Poincar\'e inequality. This is done
using a theorem of Keith (Theorem \ref{keith-theorem-1}) and a
combinatorial procedure involving concatenation of curve families of
positive $1$-modulus.

\begin{theorem}\label{thmB}
The following are equivalent:
\begin{enumerate}
\item[(a)] $(S_\ba,d,\mu)$ supports a $p$-Poincar\'e inequality for
  each $p>1$,
\item[(b)] $(S_\ba,d,\mu)$ supports a $p$-Poincar\'e inequality for
  some $p>1$,
\item[(c)] $\ba \in \ell^2$.
\end{enumerate}
\end{theorem}

For $\ba \in \ell^2 \setminus \ell^1$, the $p$-modulus of all
horizontal paths in $S_\ba$ is equal to zero for any $p$. However, the
$p$-modulus ($p>1$) of all rectifiable paths is positive. In section
\ref{sec:validity-pi-2} we exhibit explicit path families with
positive modulus. This provides a first step towards our eventual
verification of the Poincar\'e inequality. Such verification in this
context relies on the same theorem of Keith and a similar
concatenation argument, starting from curve families of positive
$p$-modulus as constructed above.

It is not unexpected, and seems to have been informally recognized, that
a generalized Sierpi\'nski carpet $S_\ba$ admits some Poincar\'e inequalities,
provided the sequence $\ba$ tends to zero sufficiently rapidly.
Indeed, if $\ba$ tends rapidly to zero then the omitted squares at each stage of the construction occupy a vanishingly small proportion of their parent square; this leaves plenty of room in the complementary region to construct well distributed curve families. The essential novelty of Theorems \ref{thmA} and \ref{thmB} lies in their sharp character; we identify the precise summability conditions necessary and sufficient for the validity of the $p$-Poincar\'e inequality for each choice of $p \in [1,\infty)$. Note that, by Theorems \ref{thmA} and \ref{thmB}, if $\ba \in \ell^2 \setminus \ell^1$, then $(S_\ba,d,\mu)$ supports a $p$-Poincar\'e inequality for each $p>1$, but does not support a $1$-Poincar\'e inequality. A significant recent result of Keith and Zhong \cite{kz:poincare} asserts that the set of values of $p$ for which a given complete PI space supports a $p$-Poincar\'e inequality, is necessarily a relatively open subset of $[1,+\infty)$.

Remarkably, the $\ell^2$ summability condition on the defining
sequence $\ba$ has recently arisen in a rather different (although related)
context. To wit, Dor\'e and Maleva \cite{dm:universal-differentiability-set}
show that when $\ba \in c_0 \setminus \ell^2$, the compact set $S_\ba$
is a {\it universal differentiability set}, i.e., it contains a
differentiability point for every real-valued Lipschitz function on $\R^2$.

Theorems \ref{thmA} and \ref{thmB} have a number of interesting consequences which we now enunciate.

\begin{corollary}\label{corA}
There exist compact planar sets of topological dimension one that are Ahlfors $2$-regular and $2$-Loewner when equipped with the Euclidean metric and the Lebesgue measure.
\end{corollary}

For each $\ba\in\ell^2$, the carpet $S_\ba$ verifies the conditions in Corollary \ref{corA}. This follows from Theorem \ref{thmB} and the equivalence of the $Q$-Loewner condition with the $Q$-Poincar\'e inequality in quasiconvex Ahlfors $Q$-regular spaces \cite{hk:quasi}. We remark that the examples of Bourdon--Pajot \cite{bp:buildings} and Laakso \cite{laa:regular} are $Q$-regular $Q$-Loewner metric spaces of topological dimension one, however, these examples admit no bi-Lipschitz embedding into any finite-dimensional Euclidean space.

\begin{corollary}\label{corB}
There exists a compact set $S \subset \R^2$, equipped with the Euclidean metric and a doubling measure, with the following properties: $S$ supports no $p$-Poincar\'e inequality for any finite $p$, yet every strict weak tangent of $S$ supports a $1$-Poincar\'e inequality with universal constants. Moreover, $S$ can be chosen to be quasiconvex and uniformly locally Gromov--Hausdorff close to planar domains.
\end{corollary}

It is a general principle of analysis in metric spaces that quantitative geometric or analytic conditions often persist under Gromov--Hausdorff convergence. In particular, quantitative and scale-invariant conditions pass to weak tangent spaces. For instance, every weak tangent of a given doubling metric measure space satisfying a $p$-Poincar\'e inequality is again doubling and satisfies the same $p$-Poincar\'e inequality (see Theorem \ref{keith-theorem-2} for a version of this result used in this paper). Corollary \ref{corB} shows that weak tangent spaces can be significantly better behaved than the spaces from which they are derived, even in the presence of other good geometric properties.

The indicated example can be obtained by choosing $S=S_\ba$ for any
$\ba \in c_0 \setminus \ell^2$. This follows from Theorem \ref{thmB}
and Proposition \ref{propD} discussed in section
\ref{subsec:weak-tangents}, where further details of the proof of
Corollary \ref{corB} can be found.

A {\it carpet} is a metric measure space homeomorphic to $S_{1/3}$.
There has been considerable interest of late in the problem of
quasisymmetric uniformization of carpets by either round carpets or
slit carpets \cite{bonk:uniformization}, \cite{bkm:schottky},
\cite{bm:carpet}, \cite{mer:schottky}, \cite{mer:cohopf}. The
following results are additional consequences of Theorem~\ref{thmB}.

\begin{corollary}\label{corC}
There exist round carpets in $\R^2$ which are Ahlfors $2$-regular and
support a $p$-Poincar\'e inequality for some $p<2$.
\end{corollary}

\begin{corollary}\label{corD}
There exist parallel slit carpets which are Ahlfors $2$-regular and
support a $p$-Poincar\'e inequality for some $p<2$.
\end{corollary}

Recall that a planar carpet is said to be a {\it round carpet} if all
of its peripheral circles are round geometric circles. A {\it slit
carpet} is a carpet which is a Gromov--Hausdorff limit of a sequence
of planar slit domains equipped with the internal metric. Recall that
a domain $D\subset\C$ is a {\it slit domain} if $D =
D'\setminus\bigcup_{i \in I} \gamma_i$, where $D'$ is a simply
connected domain and $\{\gamma_i\}_{i\in I}$ is a collection (of
arbitrary cardinality) of disjoint closed arcs contained in $D'$. We
admit the possibility that some of these arcs are degenerate, i.e.,
reduce to a point. A slit domain, resp.\ a slit carpet, is {\it
parallel} if the nondegenerate arcs are parallel line segments, resp.\
if it is a limit of parallel slit domains.

Corollaries \ref{corC} and \ref{corD} are proved in section
\ref{sec:uniformization}. Corollary \ref{corC} follows from Theorem
\ref{thmB} and results of Bonk and Koskela--MacManus on quasisymmetric
uniformization of carpets and quasisymmetric invariance of Poincar\'e
inequalities on Ahlfors regular spaces. Corollary \ref{corD} follows
from Theorem \ref{thmB}, Koebe's uniformization theorem and the same
work of Koskela--MacManus. Indeed, every carpet $S_\ba$ with $\ba \in
\ell^2$ is quasisymmetrically equivalent to both a round carpet and
also to a slit carpet with the stated properties.

\subsection{Outline of the paper}

In section \ref{sec:preliminaries} we recall general facts about
analysis in metric spaces, particularly, facts about Poincar\'e
inequalities in the sense of Definition
\ref{Poincare-inequality-definition}. In section \ref{sec:sba} we
prove basic metric and measure-theoretic properties of the carpets
$S_\ba$. In particular, we show that the canonical measure on $S_\ba$
is always a doubling measure, and we indicate in which situations it
verifies upper or lower mass bounds.

Section \ref{sec:failure} is devoted to the necessity of the $\ell^2$
summability condition for the validity of Poincar\'e inequalities on
the carpets $S_\ba$. The main result of this section, Proposition \ref{noPIl2}, shows
that $\ell^2$ summability of $\ba$ is best possible for such conclusions.
We also describe in more detail the weak tangents of the carpets $S_\ba$ and substantiate Corollary \ref{corB}.

Our proofs of the sufficiency of the summability criteria in Theorems
\ref{thmA} and \ref{thmB} are contained in sections
\ref{sec:validity-pi-1} and \ref{sec:validity-pi-2}, respectively. In the setting of Theorem \ref{thmA}, where $\ba \in \ell^1$, the Cantor set corresponding to the thinnest part of the carpet has positive length. This enables us to give a combinatorial construction of parameterized curve
families that joins arbitrary pairs of points in $S_\ba$ and verifies a modulus lower estimate due
to Keith (Theorem \ref{keith-theorem-1}) known to be equivalent to the Poincar\'e inequality in a wide setting.

In the setting of Theorem \ref{thmB}, where $\ba$ is only assumed to be in $\ell^2$, a different technique is required. 
The key step is to perform, in the
special case of the carpets $S_\ba$, the following abstract procedure:
in a metric space $(X,d)$ endowed with a wide supply of rectifiable
curves (in our case, $\R^2$), deform a given curve family so as to
avoid a prespecified obstacle, at a small quantitative multiplicative
cost to the $p$-modulus. Iterating this procedure produces curve
families of positive $p$-modulus that avoid a countable family of
obstacles of prespecified geometric sizes. Our implementation, while
not completely general, covers a wider class of residual sets than
just carpets: see Theorem \ref{modulus-extension-theorem} for a
precise statement.

In both cases, our proof of the suitable Poincar\'e inequalities makes substantial use of the
precise rectilinear structure of carpets. Hence,
the validity of a Poincar\'e inequality on the more general class of
residual sets indicated in the preceding paragraph is less clear.

In the final section (Section \ref{sec:uniformization}) we discuss uniformization of the
carpets $S_\ba$ by either round carpets or slit carpets. In
particular, we establish Corollaries \ref{corC} and \ref{corD}.

\subsection{Acknowledgements} We are grateful to Mario Bonk for
numerous discussions and especially for suggestions concerning
uniformization of metric carpets. We also thank Jasun Gong and Hrant Hakobyan
for very helpful remarks. We wish to extend particular thanks to the referee for an extremely
careful reading of the paper and for his or her detailed and constructive
input.

Research for this paper was performed while the first and third authors were at
the University of Illinois and during visits of all three authors to the
Mathematics Institute at the University of Bern. The hospitality of
both institutions is gratefully appreciated.

\section{Preliminaries}\label{sec:preliminaries}

\subsection{Basic definitions and notation}

If $B = B(x,r)$ denotes a ball in a metric space $X=(X,d)$, we write
$\lam B$ for the dilated ball $B(x, \lam r)$.


A {\it metric measure space} is a metric space $(X,d)$ equipped with a
Borel measure $\mu$ that  is finite and positive on balls. The measure
$\mu$ is {\it doubling} if there exists a constant $C>0$ so that
$\mu(B(x,2r)) \le C \mu(B(x,r))$ for all metric balls $B(x,r)$ in $X$.
It is {\it Ahlfors $Q$-regular} for some $Q>0$ if there exists a
constant $C>0$ so that $r^Q/C \le \mu(B(x,r)) \le C r^Q$ for all
metric balls $B(x,r)$ in $X$ with $0<r<\diam X$. We say that $\mu$ is
{\it Ahlfors regular} if it is Ahlfors $Q$-regular for some $Q>0$. It
is well known that any Ahlfors $Q$-regular measure on a metric space
is comparable to the Hausdorff $Q$-measure $\cH^Q$, and hence that
$\cH^Q$ is also Ahlfors $Q$-regular in that case. Ahlfors regular
measures are always doubling. Let us remark that we always denote by
$\cH^s$ the $s$-dimensional Hausdorff measure in any metric space; we
normalize these measures so that $\cH^n$ coincides with Lebesgue
measure in $\R^n$.

A metric space $(X,d)$ is said to be {\it quasiconvex} if there exists
a constant $C$ so that any pair of points $x,y\in X$ can be joined by
a rectifiable path $\gamma$ whose length is no more than $C d(x,y)$. A
metric space is quasiconvex if and only if it is bi-Lipschitz
equivalent to a length metric space.

Every doubling metric measure space admitting a Poincar\'e
inequality is quasiconvex, see for instance
\cite{haj:sobolev} or \cite{ch:lipschitz}. By making use of
quasiconvexity, we may assume that $\lam=1$ in \eqref{eq-pi}, at the
cost of increasing the value of $C$ \cite[Corollary 9.8]{haj:sobolev}.

\subsection{Poincar\'e inequalities and moduli of curve families}

The following result of Keith \cite[Theorem 2]{kei:modulus} will be of great importance in this paper.

\begin{theorem}[Keith]\label{keith-theorem-1}
Fix $p \ge 1$.  Let $(X,d,\mu)$ be a complete, doubling metric measure
space. Then $X$ admits a $p$-Poincar\'e inequality if and only if
there exist constants $C_1>0$ and $C_2\ge 1$ so that
\begin{equation}\label{keith-equation-1}
d(x,y)^{1-p} \le C_1 \mod_p(\Gamma_{xy}; \mu_{xy}^{C_2})
\end{equation}
for every pair of distinct points $x,y\in X$.
\end{theorem}

Here $\mod_p(\Gamma_{xy}; \mu_{xy}^C)$ denotes the $p$-modulus of the curve
family $\Gamma_{xy}$ joining $x$ to $y$, where the measure $\mu_{xy}^C$ is the
symmetric Riesz kernel
\[
	\mu_{xy}^C(A) = \int_{A \cap B_{xy}^C} \frac{d(x,z)}{\mu(B(x,d(x,z)))} +
		\frac{d(y,z)}{\mu(B(y,d(y,z)))} d\mu(z),
\]
where $B_{xy}^C = B(x,Cd(x,y)) \cup B(y,Cd(x,y))$. We recall that
$$
\mod_p(\Gamma;\nu) := \inf \int \rho^p \, d\nu
$$
for a Borel measure $\nu$ on $(X,d)$. Here the infimum is taken over
all nonnegative Borel functions $\rho$ which are {\it admissible} for
$\Gamma$, i.e., for which $\int_\gamma \rho \, ds \ge 1$ for all
locally rectifiable curves $\gamma \in \Gamma$. When $(X,d)$ is
endowed with a fixed ambient measure $\mu$, we abbreviate
$\mod_p\Gamma = \mod_p(\Gamma;\mu)$.

\subsection{Poincar\'e inequalities and metric gluings}\label{subsec:gluings}

The Poincar\'e inequality \eqref{eq-pi} is maintained under metric gluings. The following is a special case of a more general theorem of Heinonen and Koskela \cite[Theorem 6.15]{hk:quasi}, see also \cite[Theorem 3.3]{hh:manifold}.

\begin{theorem}[Heinonen--Koskela]\label{gluing-theorem}
Let $X$ and $Y$ be locally compact Ahlfors $Q$-regular metric measure spaces, $Q>1$, let $A\subset X$ be a closed subset, and let $\iota:A\to Y$ be an isometric embedding. Let $p>1$. Assume that both $X$ and $Y$ support a $p$-Poincar\'e inequality and that the inequality
$$
\min \left\{ \cH^{Q-1}_\infty(A \cap B_X(x,r)), \cH^{Q-1}_\infty(\iota(A) \cap B_Y(y,r)) \right\} \ge c r^{Q-1}
$$
holds for all $x \in A$, $y\in \iota(A)$ and $0<r<\min\{\diam X,\diam Y\}$, where the constant $c>0$ is independent of $x$, $y$ and $r$. Then $X \cup_A Y$ supports a $p$-Poincar\'e inequality. The data for the $p$-Poincar\'e inequality on $X \cup_A Y$ depends quantitatively on the Ahlfors regularity and Poincar\'e inequality data of $X$ and $Y$, on $p$, and on the above constant $c$.
\end{theorem}

We recall that the metric gluing $X\cup_A Y$ is the quotient space
obtained by imposing on the disjoint union $X\coprod Y$ the
equivalence relation which identifies each $a\in A$ with its image
$\iota(a)$. We equip this space with a natural metric which extends
the metrics on $X$ and $Y$ as follows: for points $x \in X$ and $y \in
Y$, let $d(x,y) = \inf \{ d(x,a) + d(\iota(a),y) : a \in A \}$.
Observe that the $Q$-regular measures on $X$ and $Y$, respectively,
combine to give a measure on $X\cup_A Y$ which is also $Q$-regular.

\subsection{Gromov--Hausdorff convergence and weak tangents}\label{subsec:gh}

A metric space $(X,d)$ is {\it proper} if closed and bounded sets are
compact.

\begin{definition}
A sequence of pointed proper metric measure spaces
$$
\{(X_n,x_n,d_n,\mu_n)\}
$$
converges to a pointed metric measure space
$(X,x,d,\mu)$ if there exists a pointed proper metric space
$(Z,z,\rho)$ and isometric embeddings $f_n:X_n \ra Z$, $f:X \ra Z$ so
that $f_n(x_n) = f(x) = z$ for all $n$, $(f_n)_\#\mu_n \ra f_\#\mu$ weakly,
and $f_n(X_n) \ra f(X)$ in the following sense:
for all $R>0, \eps>0$ there exists $N$ so that for all $n \geq N$,
$f_n(X_n) \cap B(z,R)$ is contained in the $\eps$-neighborhood of $f(X)$,
and $f(X) \cap B(z,R)$ is contained in the $\eps$-neighborhood of $f_n(X_n)$.
\end{definition}

We emphasize that the spaces $X_n$, $X$ are not assumed to be compact.
For the notion of pointed Gromov--Hausdorff convergence, see \cite[\S
2.2]{kei:modulus} or \cite[Chapter 7]{bbi:metric}.

\begin{definition}
Let $(X,d,\mu)$ be a proper metric measure space. A pointed proper
metric measure space $(Y,y,\rho,\nu)$ is called a {\it weak tangent}
of $(X,d,\mu)$ if there exists a sequence of points $\{x_n\}\subset X$
and constants $\del_n>0$, $\lam_n >0$, so that the pointed proper
metric measure spaces
$\{(X,x_n,\frac{1}{\del_n}d,\frac{1}{\lam_n}\mu_n)\}$ converge to
$(Y,y,\rho,\nu)$.
	
We do not require that $\del_n \ra 0$.  In the event that this occurs,
we call the limit space a {\it strict weak tangent} of $(X,d,\mu)$. If
$x_n = x \in X$ for all $n$, we call $(Y,y,\rho,\nu)$ a {\it tangent}
to $X$ at $x$. The notion of {\it strict tangent} is defined similarly.
\end{definition}

Poincar\'e inequalities persist under Gromov--Hausdorff convergence; see Cheeger \cite[\S9]{ch:lipschitz}.
We state here a version of this result due to Keith \cite{kei:modulus}, in a form which is suitable for our setting. For
another version, see Koskela \cite{kos:ug-pi}.

\begin{theorem}[Cheeger, Koskela, Keith]\label{keith-theorem-2}
Suppose $X_1 \supset X_2 \supset \cdots$ are subsets of $\R^2$, and for each
$n\in\N$, $\mu_n$ is a doubling measure supported on $X_n$, with uniform
doubling constant. Let $X = \bigcap_{n\in\N} X_n$, and suppose that the
measures $\{\mu_n\}$ converge weakly to a measure $\mu$ supported on $X$. If each $(X_n,d,\mu_n)$ supports a $p$-Poincar\'e inequality with uniform constants, then $(X,d,\mu)$ also supports a $p$-Poincar\'e inequality.
\end{theorem}

\section{Definition and basic properties of the carpets $S_\ba$}\label{sec:sba}

We review the construction of the carpets $S_\ba$. Fix a sequence
$$
\ba = (a_1, a_2, \ldots)
$$
where each $a_m$ is an element of the set $\{\frac{1}{3}, \frac{1}{5},
\frac{1}{7}, \ldots\}$. Starting from the unit square $T_0=[0,1]^2$ we
set the {\it level parameter} $m=1$ and iteratively apply the
following two steps:
\begin{itemize}
\item Divide each current square into $a_m^{-2}$ essentially disjoint
  closed congruent subsquares, where $m$ denotes the current level parameter,
  and remove the central (concentric) subsquare from each square,
\item Increase the level parameter $m$ by $1$.
\end{itemize}
We let $\cT_m$ denote the collection of all remaining level $m$ squares. For
each $m$, $\cT_m$ consists of
\begin{equation*}
\prod_{j=1}^m \left( a_j^{-2}-1 \right)
\end{equation*}
essentially disjoint closed squares, each of side length
\begin{equation*}
s_m := \prod_{j=1}^m a_j.
\end{equation*}
The union of all squares in $\cT_m$ is the {\it level $m$ precarpet},
denoted $S_{\ba,m}$. A {\it peripheral square} is a connected
component of the boundary of a precarpet. Finally,
\begin{equation*}
S_\ba := \bigcap_{m\ge 0} S_{\ba,m} = \bigcap_{m\ge 0} \bigcup \cT_m.
\end{equation*}

Each carpet $S_{\ba}$ is quasiconvex; this can be demonstrated using
curves built by countable concatenations of horizontal and vertical
segments. It is well-known that the usual Sierpi\'nski carpet
$S_{1/3}$ contains other nontrivial line segments, neither horizontal
or vertical. Indeed, $S_{1/3}$ contains nontrivial line segments of
each of the following slopes: $0$, $1/2$, $1$, $2$ and $\infty$.
For an explicit description of the set of slopes of nontrivial line
segments in all carpets $S_\ba$ in terms of Farey fractions, see~\cite{dt:lines}.

\subsection{The natural measure on $S_\ba$}\label{sec:measure}

There is a natural probability measure on $S_\ba$.  Since each
precarpet $S_{\ba,m}$ has positive area, we define a measure
$\mu_m$ on $[0,1]^2$ which is the Lebesgue measure restricted to the set
$S_{\ba,m}$, renormalized to have total measure one.  The sequence of
measures $(\mu_m)$ converges weakly to a probability measure $\mu$
with support $S_\ba$.
To see this, note that on each (closed) square $T$ of scale $s_m$ that is not
discarded, we have $\mu_n(T) = \mu_m(T)$ for all $n \ge m$, since later
renormalizations merely redistribute mass within $T$. Therefore,
$$
\mu(T) = \mu_m(T) = \prod_{j=1}^m (a_j^{-2}-1)^{-1} =: v_m.
$$
Moreover, for fixed $Q>0$,
$$
\frac{\mu(T)}{s_m^Q} = \prod_{j=1}^m a_j^{2-Q}(1-a_j^2)^{-1}.
$$
Note that if all $a_m = 1/(2k+1)$, then $\mu(T) = s_m^{Q_k}$ for all
$T\in\cT_m$ and $\dim S_{1/(2k+1)} = Q_k$. Here $Q_k$ denotes the
value in \eqref{Qk}.


The following proposition describes the basic properties of $\mu$.
We write $a \lesssim b$ to mean that there exists a constant $C>0$ so that
$a \leq C b$, where $C$ depends only on the relevant data.
Also, the notation $a \asymp b$ signifies that $a \lesssim b$ and $b \lesssim a$.

\begin{proposition}\label{properties-of-mu} The metric measure space $(S_\ba, d, \mu)$ has the following properties:
\begin{enumerate}
\item[(i)] For any $\ba$, $\mu$ is a doubling measure.
\item[(ii)] For any $\ba$, we have the lower mass bound $\mu(B(x,r))
  \gtrsim r^2$ for all $x$ and $r\le 1$.
\item[(iii)] If $\ba \in c_0$, then for any $Q<2$ we have $\mu(B(x,r))
  \lesssim r^Q$ for all $x$ and $r>0$, hence $\dim S_\ba = 2$.
\item[(iv)] If $\ba \in \ell^2$, then $\mu$ is comparable to Lebesgue
  measure with constant depending only on $||\ba||_2$. Moreover, in
  this case, $\mu$ is an Ahlfors $2$-regular measure on $S_{\bf a}$.
\item[(v)] If $\ba=(a_m)$ is eventually constant (and equal to
  $\frac1{2k+1}$), then $\mu$ is comparable to the Hausdorff
  measure $\cH^{Q_k}$ and is an Ahlfors $Q_k$-regular measure on
  $S_\ba$.
\end{enumerate}
\end{proposition}

For $x\in S_\ba$ and $r>0$ define two integers $m(x,r)$ and $m(r)$ as
follows:
\begin{enumerate}
\item $m(x,r)$ is the smallest integer $m$ so that there exists $T \in
  \cT_m$ with $x\in T \subset B(x,r)$,
\item $m(r)$ is the smallest integer $m$ so that $s_m \le r$.
\end{enumerate}


First, an easy lemma:

\begin{lemma}\label{easy-lemma}
For any $x$ and $r$, $m(\sqrt2 r) \le m(x,r) \le
m(\frac{r}{\sqrt2}) + 1$.
\end{lemma}

\begin{proof}
If $T\in \cT_{m(x,r)}$ satisfies $x \in T \subset B(x,r)$, then
$\sqrt2 s_{m(x,r)} = \diam T \le \diam B(x,r) \le 2r$ which implies
that $s_{m(x,r)} \le \sqrt2 r$ and $m(\sqrt2 r) \le m(x,r)$. Since $x
\in T$ for some $T \in \cT_{m(r / \sqrt2)+1}$, and $\diam{T} \leq
\frac{r}{3}$, we have $m(x,r) \leq m(\frac{r}{\sqrt2})+1$.
\end{proof}

We will derive the various parts of Proposition \ref{properties-of-mu} from
the following

\begin{proposition}\label{prop-ball-measure}
For each $x\in S_\ba$ and $0<r\le 1$,
\begin{equation*}\label{P2E}
\mu(B(x,r)) \asymp h(r) := r^2
\prod_{j=1}^{m(r)}\left(\frac1{1-a_j^2}\right).
\end{equation*}
\end{proposition}

\begin{proof}[Proof of Proposition \ref{properties-of-mu}]
Note that $m(r)$ is a decreasing function of $r$. Part (i)
follows easily:
$$
\mu(B(x,2r)) \lesssim (2r)^2 \prod_{j=1}^{m(2r)}
\left(\frac1{1-a_j^2}\right) \le 4 r^2 \prod_{j=1}^{m(r)}
\left(\frac1{1-a_j^2}\right) \lesssim \mu(B(x,r)).
$$
Part (ii) is also clear, since the finite product term in the
definition of $h(r)$ is always greater than or equal to one.

Next, we assert that $m(r) \le m(2r) + 1$ for all $r>0$. If not, we
have $m(r) \geq m(2r)+2$, so $m(r) -1 \geq m(2r) + 1$, thus
\[
r < s_{m(r)-1} \leq s_{m(2r)+1} \leq \frac{1}{3}s_{m(2r)} \leq \frac{1}{3} \cdot 2r,
\]
a contradiction.

We now turn to part (iii). Assume that $\ba \in c_0$, i.e., $a_m \to
0$. We will show that $\limsup_{r\to 0} \frac{\mu(B(x,r))}{r^Q}$ is
finite for each $Q<2$, uniformly in $x$. It suffices to show that
$$
\limsup_{r\to 0} \frac{h(r)}{r^Q} < \infty.
$$

First we verify that $m(r) \leq - \log_2(r) +1$. Suppose that $n$ is
the largest integer so that $2^{n} r \leq 1$. Since $m(1)=0$,
$$
m(r) \leq m(2r)+1 \le \cdots \le m(2^{n+1} r) + n+1 \le m(1) + n + 1 =
n+1 \le -\log_2(r) + 1.
$$
Now, since $a_j \to 0$, for any $\eps > 0$ there exists some
$C=C(\eps)$ so that
\[
	\prod_{j=1}^{m(r)} \left(\frac1{1-a_j^2}\right)
	\leq C (1+\eps)^{m(r)}
	\leq C (1+\eps)^{-\log_2{r}+1}
	\lesssim r^{-\log_2(1+\eps)}.
\]
If we choose $\eps$ so that $2-Q > \log_2(1+\eps)$, then we are done.
From here part (iii) follows easily. Parts (iv) and (v) were discussed
in the introduction.
\end{proof}

\begin{proof}[Proof of Proposition \ref{prop-ball-measure}]
It is straightforward to bound $\mu(B(x,r))$ from above: cover
$B(x,r)$ by squares from $\cT_{m(r)}$.  Then, as $s_{m(r)} \le r$,
$$
\mu(B(x,r)) \le \frac{(3r)^2}{s_{m(r)}^2} \cdot v_{m(r)}
\le \frac{9r^2}{s_{m(r)}^2} \cdot s_{m(r)}^2
\prod_{j=1}^{m(r)} \left(\frac1{1-a_j^2}\right) \lesssim h(r).
$$
To bound $\mu(B(x,r))$ from below, we split the proof into two cases.

\

\paragraph{\bf Case 1.} $r \leq 100 s_{m(x,r)}.$
	
Since $B(x,r)$ contains a square of side $s_{m(x,r)}$, we use the
obvious bound $\mu(B(x,r)) \geq v_{m(x,r)}$. Note that $m(r) - 1 \le
m(\frac{r}{\sqrt2}) - 1 \leq m(\sqrt2 r) \leq m(x,r)$. Now,
$$
v_{m(x,r)} = s_{m(x,r)}^2 \prod_{j=1}^{m(x,r)} \left(
  \frac{1}{1-a_j^2} \right) \ge \left( \frac{1}{100} \right)^2 r^2
\cdot (1-a_{m(r)}^2) \prod_{j=1}^{m(r)} \left( \frac{1}{1-a_j^2}
\right)  \gtrsim h(r).
$$
	
\paragraph{\bf Case 2.} $r > 100 s_{m(x,r)}.$
	
Choose $T \in \cT_{m(x,r)-1}$ so that $x \in T$. Since $T \nsubseteq
B(x,r)$, the side length of $T$ is at least $\frac{r}{\sqrt2}$. Since $T$ is a square, $T \cap B(x,r)$ contains a (Euclidean) square
$V'$ of side $\frac{r}{4}$. Finally, since $s_{m(x,r)} \leq \frac{r}{100}$ and at most one square of
generation $m(x,r)$ is deleted in $T$, $V'$ contains a square $V$ of
side $s_v \in [\frac{r}{32}, \frac{r}{16}]$ consisting entirely of squares from $\cT_{m(x,r)}$.

From the preceding facts we conclude that
\begin{align*}
\mu(B(x,r))
& \ge \mu(V) = \left( \frac{s_v}{s_{m(x,r)}} \right)^2 v_{m(x,r)}
	= s_v^2 \cdot \prod_{j=1}^{m(x,r)} \left(  \frac{1}{1-a_j^2} \right)\\
&\ge \frac{r^2}{32^2} \cdot (1-a_{m(r)}^2) \prod_{j=1}^{m(r)}
\left( \frac{1}{1-a_j^2} \right) \gtrsim h(r).
\end{align*}
The proof is finished.
\end{proof}

We make a final observation regarding the conformal dimension of $S_\ba$.
Recall that a metric space $(X,d)$ is \emph{minimal for conformal dimension} if its Hausdorff dimension is less than or equal to the Hausdorff dimension of any quasisymmetrically equivalent metric space. The self-similar carpets $S_{1/(2k+1)}$ are not minimal for conformal dimension. This result is a consequence of a theorem of Keith and Laakso \cite{kl:confdim},
see also \cite{mt:confdimsurvey} for a brief recapitulation of the proof.

\begin{corollary}\label{prop-conf-dim}
If $\ba \in c_0$, then $S_\ba$ is minimal for conformal dimension.
\end{corollary}

\begin{proof}
	By Proposition~\ref{properties-of-mu}(iii), $S_\ba$ has Hausdorff dimension 2.
	In a similar way, one shows that the Cantor set $C_\ba$ has Hausdorff dimension 1.
	Since $S_\ba$ contains the product of $C_\ba$ and an interval, which has Hausdorff dimension 2,
	by \cite[Section 5, Remark 1]{bt:locmin} the space $S_\ba$ is minimal for conformal dimension.
\end{proof}

The conformal dimensions of the carpets $S_\ba$ when $\ba \not \in c_0$ remain unknown. Determining the conformal dimension of $S_{1/3}$ is a longstanding open problem.

\section{Failure of the Poincar\'e inequality}\label{sec:failure}

In this section we provide conditions under which the $p$-Poincar\'e
inequality fails to be satisfied on $S_\ba$ for various choices of $p$
and $\ba$. In doing so we verify the necessity of the summability
criteria in Theorems \ref{thmA} and \ref{thmB}.

\begin{proposition}\label{noPIl1}
If $\ba \notin \ell^1$, then $S_\ba$ does not support a $1$-Poincar\'e inequality.
\end{proposition}

\begin{proof}
For each $m \in \N$, let $T_m \subset S_\ba$ be the vertical middle strip of width $s_m$.
Define $f_m: S_\ba \ra [0, 1]$ to be the function which is $0$ to the left of $T_m$,
$1$ to the right of $T_m$ and extend it linearly across $T_m$.
This function has upper gradient $\rho_m : S_\ba \ra [0,\infty]$
which is identically $1/{s_m}$ on $T_m$ and $\rho_m \equiv 0$ elsewhere.
We compute
$$
\int \rho_m \, d\mu = \prod_{i=1}^m \left( a_i^{-1} \cdot \frac{a_i -
    a_i^2}{1-a_i^2} \right) = \prod_1^m \frac{1}{1+a_i}.
$$
Since $\ba \notin \ell^1$, the right hand side goes to zero
as $m \ra \infty$.
Observe that $\dashint_{S_\ba} f_m d\mu = 1/2$, and $f_m$ takes values of $0$ and $1$ on a set of
measure bounded away from zero independently of $m$.
Therefore, \eqref{eq-pi} cannot be satisfied for $p=1$ and fixed constant $C$.
\end{proof}

We now consider $p$-Poincar\'e inequalities with $p>1$.
If $\ba \not\in \ell^3$, a careful adaptation of the proof of the previous proposition shows that $S_\ba$
does not support any Poincar\'e inequality.
However, the carpets considered in this paper have a very specific geometry that leads to the following sharp result.

\begin{proposition}\label{noPIl2}
If $\ba \notin \ell^2$, then $S_\ba$ does not support a $p$-Poincar\'e
inequality for any $p \geq 1$.
\end{proposition}

\begin{proof}
Our goal is to build a set $X \subset S_\ba$ with $\mu(X) = 0$
so that for every rectifiable curve $\gamma$ joining the left and right hand edges
of $S_\ba$ we have $\int_\gamma \rho\, ds \geq 1$, where $\rho$ is the characteristic function of $X$.

This suffices to show the failure of the Poincar\'e inequality, for we can then
define a function $f$ on $S_\ba$ by letting $f(x)$ be the infimum of $\int_\gamma \rho\, ds$, where
$\gamma$ ranges over all rectifiable curves joining the left edge of $S_\ba$ to $x$.
As $\rho \leq 1$ and $S_\ba$ is quasi-convex, $f$ is a Lipschitz function which is zero on the left edge of $S_\ba$.
The property described above shows that $f \geq 1$ on the right edge of $S_\ba$.
Since $f$ has an upper gradient with essential supremum zero, we have a contradiction to \eqref{eq-pi}.

In the remainder of the proof we build the set $X$ and show it has the desired properties
for some fixed, arbitrary rectifiable curve $\gam$ in $S_\ba$ that joins the left and right hand
edges of $S_\ba$. By passing to a subcurve if necessary we may assume
that $\gam$ is an arc, i.e., that it is injective.

As part of our proof we shall build cut sets which disconnect $S_\ba$.
To simplify our discussion later, we define the set $H_\ba$ to be the union
of sets $[a, b] \times (c, d)$, for every deleted open square $(a, b) \times (c, d)$ in the construction
of $S_\ba$.

\

\paragraph{\bf Initial step.} Let $A_0 = S_\ba$, and $\Gamma_0 = \{\gamma\}$.
	
We divide $S_\ba$ into $m_1 = s_1^{-1}$ vertical strips of width $a_1$.
These strips are bounded by vertical cut sets
$V_0,V_1,\ldots,V_{m_1}$, where
$V_0 = (\{0\}\times[0,1]) \setminus H_\ba, V_1 = (\{a_1^{-1}\}\times[0,1]) \setminus H_\ba$, and so on.
In other words, each $V_j$ is a vertical line, with the exception that the interiors of vertical sides
of the deleted square of side $a_1$ are not contained in the appropriate $V_j$.
	
We now split $\Gamma_0 = \{\gamma\}$ into a disjoint family of curves.
We parametrize $\gamma$ by the interval $[0,1]$, with $\gamma(0) \in V_0$, and
$\gamma(1) \in V_{m_1}$.  Let $t_0^+ \geq 0$ be the last time $\gamma$ meets $V_0$.
Let $t_1^- > t_0^+$ be the next time after that $\gamma$ meets $V_1$.
Let $\gamma_1$ be the subpath of $\gamma$ given by restricting to $[t_0^+,t_1^-]$.
	
Continue inductively, letting $t_{j-1}^+ \geq t_{j-1}^-$ be the last time
$\gamma$ meets $V_{j-1}$, and $t_j^- > t_{j-1}^+$ be the next time $\gamma$
meets $V_j$. Let $\gamma_j$ be the subpath of $\gamma$ given by $[t_{j-1}^+,
t_j^-]$.
	
By construction, $\Gamma_1 = \{\gamma_1, \ldots, \gamma_{m_1}\}$ is a family
of curves, where each $\gamma_j$ joins $V_{j-1}$ to $V_j$, and is
contained between them.
(See Figure~\ref{fig-no-pi-chop}, where the deleted subpaths are
indicated by dotted lines.)

\begin{figure}[b]
	\centering
	\includegraphics[width=0.25\textwidth]{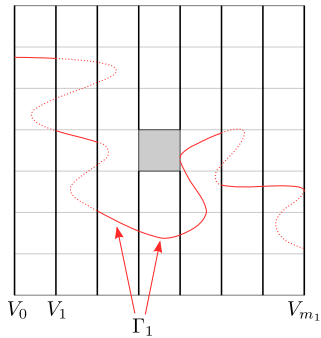}
	\caption{Curve splitting}
	\label{fig-no-pi-chop}
\end{figure}
	
Note that the length of $\Gamma_1$ (i.e., the sum of the lengths of $\gamma_1,\ldots,\gamma_{m_1}$), is at least one and at most the length of $\Gamma_0$, that is $\length(\gamma)$.

\

\paragraph{\bf Inductive step (fold in).}
Fix $i \ge 1$. We are given as input a collection $\Gamma_i = \{\gamma_j\}$ of $m_i =
s_i^{-1}$ curves and vertical slices $V_0, \ldots, V_{m_i}$, where
$\gamma_j$ joins $V_{j-1}$ to $V_j$ and is contained between them, for
each $j = 1, \ldots m_i$.
	
Choose the largest $l_i \in \N$ so that $l_i a_{i+1} \leq \frac{1}{3}$.  Note
that since $a_{i+1} \in \{ \frac{1}{3},\frac{1}{5}, \ldots \}$, we have $l_i
a_{i+1} > \frac{1}{3}-\frac{1}{5}$.
	
Let $\cD_i$ be the collection of open rectangles of width $l_i a_{i+1} s_i$ and
height $(1 - 2 l_i a_{i+1})s_i$ centered on and adjacent to either the left or right sides of
the deleted squares of side length $s_i$. Consider the squares of side length
$l_i a_{i+1} s_i$ above and below each of these rectangles. Each such square $S$ has a diagonal which meets the corner of a deleted square of side length $s_i$. This diagonal divides $S$ into two triangles; let $\cR_i$ be the collection of closures of those triangles
that share a side with a deleted square of side $s_i$.
See Figure~\ref{fig-no-pi-DRF} for part of an example where these regions are labeled.
We use the convention that the rectangles and triangles referred to are
actually the intersection of the corresponding planar set and the carpet $S_{\ba}$.

\begin{figure}[t]
	\centering
	\includegraphics[width=0.25\textwidth]{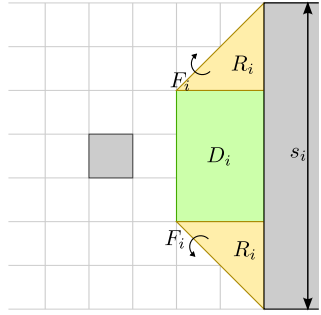}
	\caption{Unfolding}
	\label{fig-no-pi-DRF}
\end{figure}

We now define a folding map $F_i:S_\ba \ra S_\ba$ by declaring $F_i$
to be the identity except on $\cup\cR_i$, where the map folds each triangle $R \in \cR_i$
across the diagonal. Note that the horizontal edges of every rectangle in $\cD_i$ are mapped to vertical edges, and that $F_i$ is discontinuous along such edges.

Notice that in Figure~\ref{fig-no-pi-DRF}, if $a_{i+1} < \frac{1}{3}$ then the region $\cup \cD_i$ will not
overlap with $\cup \cD_{i+1}$. However, when $a_{i+1} = \frac{1}{3}$, $\cD_i$
may contain a square $Q$ of side length $s_{i+1}$ adjacent to the
left or right of a particular deleted square of side length $s_{i+1}$, but
it cannot contain both such squares. When this happens, we leave the
square $Q$ untouched at step $i+1$, and do not include any part of it in $\cD_{i+1}$.

We apply $F_i$ to the collection of curves $\Gamma_i$. Consider the resulting collection of curves.
(Note that some curves have been broken into smaller pieces along the discontinuities of $F_i$.)
We now build $\Gamma_{i+1}$ using the same inductive construction as we used in the initial step to
build $\Gamma_1$ from $\Gamma_0$.
As before, let $V_j = (\{j s_{i+1} \} \times [0,1]) \setminus H_\ba$, for $j=0,\ldots,m_{i+1}$, separate the unit square
into essentially disjoint vertical strips of width $s_{i+1}$. Here $m_{i+1} := s_{i+1}^{-1}$. We define times
$$
0 \le t_0^+ < t_1^- \le t_1^+ < \cdots
$$
as before, for the broken curve $F_i(\Gamma_i)$.

Let $\gamma_j$ be the broken curve in $F_i(\Gamma_i)$ given by the restriction to $[t_{j-1}^+,t_j^-]$. In fact, $\gamma_j$ is connected.

By construction, $\Gamma_{i+1} = \{\gamma_1,\ldots,\gamma_{m_{i+1}}\}$ is a family of curves, where each $\gamma_j$ joins $V_{j-1}$ to $V_j$, and is contained between them.

Moreover, $\Gamma_{i+1}$ will lie inside $A_i = \overline{A_{i-1} \setminus \cup (\cD_i \cup \cR_i)}$.
To see why this is so, consider (for example) $j$ so that $V_j$ lies on the left edge of a rectangle $D$ in $\cD_i$, where $D$ lies on the left of an omitted square of side length $s_i$. By definition, $t_j^+$ is the last time the broken curve $F_i(\Gamma_i)$ meets $V_j$. This corresponds to the last time that $\Gamma_i$ meets either $V_j$ or the horizontal edges of any element of $\cD_i$ above or below $D$. Consequently, the family $\Gamma_{i+1}$ is disjoint from $\cup\cD_i$ and (by the definition of $F_i$) it is also disjoint from the interior of $\cup\cR_i$. As before, the length of $\Gamma_{i+1}$ is at least one and at most the length of $\Gamma_i$.

\

\paragraph{\bf Conclusion.} We continue this construction until $i=n$, when we have a collection of curves $\Gamma_{n+1}$ that lies in
$A_n$, and have deleted the rectangles in the collections $\cD_1, \ldots, \cD_n$ from the sides of the deleted squares.
We let $X_n^{(n)}=A_n$, and now proceed to unfold this set and the curves $\Gamma_{n+1}$ back out into
the regions $\cup\cR_1,\ldots,\cup\cR_n$.

Define inductively
\[
X_n^{(i-1)} = F_{i}^{-1}(X_n^{(i)}),
\]
for $i=n,n-1, \ldots, 1$.   Observe that $X_n^{(0)}$ is all of $S_{\ba}$,
but with certain rectangles removed that are adjacent to sides of removed squares of $S_{\ba,n}$.

It is clear that $S_\ba = X_0^{(0)} \supset X_1^{(0)} \supset
X_2^{(0)} \supset \cdots$. Let $X = \bigcap_{i=0}^{\infty} X_i^{(0)}$.
For each $n$, and for any $\gamma$ in $\Gamma$, our construction implies that
\begin{align*}
	 \length(\gamma \cap X_n^{(0)}) & = \length(\Gamma_0 \cap X_n^{(0)})
		\geq \length(\Gamma_1 \cap X_n^{(0)}) \\
		& \geq \length(\Gamma_2 \cap X_n^{(1)})
		\geq \cdots
		\geq \length(\Gamma_{n+1} \cap X_n^{(n)}) \\
		& = \length(\Gamma_{n+1} \cap A_n)
		\geq 1,
\end{align*}
since $\Gamma_{n+1}$ is a chain of paths crossing each vertical strip of width
$s_{n+1}$ from the left to the right, and $\Gamma_{n+1}$ lives in $A_n$.

Therefore, $\length(\gamma \cap X) \geq 1$, since $\length = \cH^1$ is
a measure on arcs.

Let $\rho$ be the characteristic function of $X$.
Since $\gamma$ was an arbitrary rectifiable curve joining the left and right sides of $S_\ba$,
and $X$ was constructed independently of $\gamma$,
we have shown that
$\int_\gamma \rho\, ds \geq 1$ for every such curve $\gamma$.

It remains to prove that $\mu(X)=0$. Consider each deleted square of
side $s_i$. Out of the neighboring $(a_i^{-2}-1)$ boxes of side $s_i$,
from at least one (two if $a_i \neq \frac{1}{3}$) of these we will
delete a rectangle in $\cD_i$ whose $\mu$-measure, as a proportion of a
square of side $s_i$, is at least $(1-2 l_i a_{i+1}) l_i a_{i+1} \geq
\frac{2}{45}$. Since all the rectangles in $\cD_i$ are pairwise disjoint, we have
\[
\mu(X) \leq \prod_{i=1}^\infty \left( 1- \frac{1}{a_i^{-2}-1}
  \cdot \frac{2}{45} \right) = \prod_{i=1}^\infty \left( 1 - \frac{2}{45}
  a_i^2 + \cdots \right),
\]
which converges to zero since $\ba \notin \ell^2$.
This completes the proof.
\end{proof}

\begin{remark}
The argument also shows that $S_\ba$ does not support an
$\infty$-Poincar\'e inequality. See \cite{djs:infinity-poincare} for
the definition, which is weaker
than the $p$-Poincar\'e inequality for any finite $p$.
\end{remark}

\subsection{Weak tangents of Sierpi\'nski carpets}\label{subsec:weak-tangents}

Weak tangents of metric spaces describe infinitesimal behavior at a point or
along a sequence of points. In this section we characterize the strict
weak tangents of non-self-similar carpets. More precisely, we prove
the following proposition.

\begin{proposition}\label{propD}
Let $\ba \in c_0$. Then every strict weak tangent of $S_\ba$ is of the
form $(\R^2 \setminus T,d,\nu)$ where $T$ is a generalized square and
$\nu$ is proportional to Lebesgue measure restricted to $\R^2
\setminus T$.
\end{proposition}

By a {\it generalized square} we mean a set $T \subsetneq \R^2$ of the
type
$$
T = (a,b) \times (c,d)
$$
where $-\infty \le a \le b \le \infty$, $-\infty \le c \le d \le
\infty$, and $b-a=d-c$ if one (hence both) of these values is finite.
(We interpret the degenerate interval $(a,b)$, $a=b$, as the empty
set.) Thus $T$ is either the empty set, an open square, a quadrant
or a half-space.

Suppose $W$ is a strict weak tangent arising as the limit of the
sequence of metric spaces $\{X_n = (S_\ba, x_n,\frac{1}{\del_n}d)\}$,
where $x_n \in S_\ba$, $\del_n \in (0, \infty)$, and $\del_n \ra 0$.

The following lemma indicates why $W$ can omit at most one large
square.
	
\begin{lemma}
There exist $R_n \ra \infty$ and $r_n \ra 0$ so that in the ball $B(x_n, R_n)
\subset X_n$ there is at most one square of side greater than $1$ removed, and
all other squares removed have size at most $r_n$.
\end{lemma}

\begin{proof}
Fix $n$, and let $m = m(\del_n)$, i.e., $s_m \leq \del_n < s_{m-1}$.

Either $\del_n \in [s_{m},s_{m-1} \sqrt{a_{m}})$, or $\del_n \in [s_{m-1} \sqrt{a_{m}}, s_{m-1})$. In the first case, removed squares of size at least $\frac{s_{m}}{\del_n}$ are $\frac{s_{m-1}}{2\del_n} \geq \frac{1}{2\sqrt{a_{m}}}$ separated, while all others have size at most $\frac{s_{m+1}}{\del_n} \leq a_{m+1}$. In the second case, removed squares of size at least $\frac{s_{m-1}}{\del_n} \geq 1$	are $\frac{s_{m-2}}{2\del_n} \geq \frac{1}{2a_{m-1}}$ separated, and all others have size at most $\frac{s_{m}}{\del_n} \leq \sqrt{a_{m}}$.

Setting $R_n = \tfrac14\min\{\frac{1}{a_{m-1}},\frac{1}{\sqrt{a_{m}}}\}$ and $r_n = \max\{\sqrt{a_{m}},a_{m+1}\}$, we have proved the lemma.
\end{proof}

\begin{proof}[Proof of Proposition \ref{propD}]	
Using the preceding lemma, we can reduce the proof of Proposition \ref{propD} to
consideration of limits of $\R^2 \setminus T_n$ where $T_n$ is either
a square of side at least one, or the empty set.  It is easy to see
that any strict weak tangent $W$ as above will be isometric to $\R^2 \setminus T$, where $T$ is either
a square (of side at least one), a quarter-plane, half-plane or the
empty set.
	
Since the measure $\mu$ on $S_\ba$ agrees with the weak limit of
renormalized Lebesgue measure on the domains $S_{\ba,m}$, by the
lemma, if we look at measures of balls in $X_n$ of size much larger
than $r_n$, they will agree with a constant multiple of Lebesgue
measure up to small error. Consequently, the possible measures on $W$
will arise as limits of rescaled Lebesgue measure on $\R^2 \setminus
T_n$. The only possible non-trivial Radon measure of this type is
a constant multiple of Lebesgue measure restricted to $\R^2 \setminus T$. This finishes the
proof of Proposition \ref{propD}.
\end{proof}

Note that all of the weak tangent spaces identified in the conclusion of Proposition \ref{propD}
support a $1$-Poincar\'e inequality, with uniform constants (i.e.,
independent of the choice of such a weak tangent space). This is
because we have only a finite number of similarity types of spaces
(full space, half space, quarter space, or the complement of a
square), and the Poincar\'e inequality data is invariant under
similarities. The quasiconvexity of the original carpets $S_\ba$ is a
standard fact. Indeed, arbitrary pairs of points can be joined by
quasiconvex curves which are comprised of countable unions of
horizontal and vertical line segments.

\begin{definition}
A metric space $(X,d)$ is {\it locally Gromov--Hausdorff close to planar domains} if for each $x\in X$ and each $\eps>0$, there exists $r>0$ and a domain $\Omega\subset\R^2$ so that the Gromov--Hausdorff distance between the metric ball $B(x,r)\subset X$ and $\Omega$ is at most $\eps$. Furthermore, $(X,d)$ is {\it uniformly locally Gromov--Hausdorff close to planar domains} if $\eps$ can be chosen independently of $x$.
\end{definition}

The fact that $S_\ba$ is uniformly locally Gromov--Hausdorff close to planar domains follows easily from the construction and the condition $\ba \in c_0$.

The preceding discussion and Theorem~\ref{thmB} understood, the proof of Corollary \ref{corB} is complete
by the choice of $\ba \in c_0 \setminus \ell^2$.


\section{Validity of the Poincar\'e inequality: the case $\ba \in \ell^1$}\label{sec:validity-pi-1}

In this section, we make the standing assumption that $\ba \in \ell^1$, and show that $S_\ba$, equipped with the Euclidean metric and the canonical measure
described in subsection \ref{sec:measure}, admits a $1$-Poincar\'e inequality.
Recall that whenever $\ba$ is in $\ell^2 \subset \ell^1$, Proposition~\ref{properties-of-mu}(iv)
states that $\mu$ is comparable to Hausdorff $2$-measure $\cH^2$ restricted to $S_\ba$.
For simplicity we will work with $\cH^2$ in this section and the next.

According to Theorem \ref{keith-theorem-1}, the validity of a
Poincar\'e inequality is equivalent to the existence of curve families
of uniformly and quantitatively large weighted modulus joining
arbitrary pairs of points. The desired curve family must spread out as it escapes from the endpoints. This divergence is
measured via transversal measures on the edges of squares in the precarpets.

We explicitly construct this family using the structure of the carpet
$S_{\ba}$. In subsection \ref{subsec:modular} we state and prove four
lemmas providing the `building blocks' of the construction. Each of
these building blocks consists of families of disjoint curves joining
edges of certain squares in the carpet. These families are each
equipped with a natural transversal measure, and concatenated to
produce the desired family connecting the given endpoints.

By Theorem \ref{keith-theorem-2}, to demonstrate that $S_\ba$ admits a
$1$-Poincar\'e inequality, it suffices to prove
that the precarpets $\{(S_{\ba,m},d,\mu_m)\}_{m \in \nats}$ support a
$1$-Poincar\'e inequality with constants independent of $m$.
In order to simplify the discussion, we work in a fixed precarpet
$S_{\ba, M}$. In subsection \ref{subsec:pi} we use the `building
block' lemmas of subsection \ref{subsec:modular} to build the desired
path family in $S_{\ba, M}$, and complete the proof of our main
theorem.

To simplify the argument we will impose the requirement
\begin{equation}\label{aia}
a_i \le a_* := \frac1{20} \qquad \mbox{for all $i$.}
\end{equation}
This requirement entails no loss of generality, as the carpet $S_\ba$  is the finite union of similar copies of some carpet $S_{\ba'}$, where $\ba' \in \ell^1$ and
all entries of $\ba'$ are less than $a_*$, glued along their boundaries. By Theorem \ref{gluing-theorem}, if each of these smaller
carpets supports a $1$-Poincar\'e inequality, then the original carpet
will also. We note that the gluing procedure in Theorem \ref{gluing-theorem}
differs from the union considered here, however, the resulting metrics
are bi-Lipschitz equivalent and the validity of Poincar\'e
inequalities is unaltered by this change of metric. The constants for the overall Poincar\'e
inequality will depend on the number of copies which are glued
together, which in turn depends on how far out in the sequence $\ba$
we must go to ensure condition \eqref{aia}. If $\ba$ is monotone
decreasing, this data depends only on $\|\ba\|_1$.

If \eqref{aia} is not satisfied, the algorithmic construction in the proof of Lemma \ref{lem-block-chain}
becomes slightly more complicated, however, the rest of the argument is unchanged. 
We leave such modifications to the industrious reader.

\subsection{Building block lemmas}\label{subsec:modular}

Recall that $\cT_m$ denotes the collection of all level $m$ squares
in the construction of the carpet $S_\ba$.

\begin{definition}
Fix $m \in \N$ and a square $T' \in \cT_{m-1}$.
For non-negative integers $a, b, k$ and $l$, a set $\cC=[a s_m,(a+k)s_m]\times [b s_m,(b+l)s_m]$
is called a \emph{$k$ by $l$ block in $T'$}
if it is contained in $T'$ and does not contain the removed central subsquare of $T'$.
We will often choose a preferred edge $L$ of a block $\cC$ that does not contain the boundary
of the removed central subsquare of $T'$, 
and declare it to be the \emph{leading edge} of $\cC$.
The pair $(\cC,L)$ is called a \emph{directed block in $T'$}.
A directed $1$ by $1$ block in $T'$ is called a \emph{directed square in $T'$}.
We will suppress reference to $T'$ if it the dependence is clear or unimportant.

Note that the choice of a leading edge of a block gives rise to an outward-pointing unit
normal vector $(1,0), (0, 1), (-1,0)$ or $(0,-1)$.

Directed blocks (possibly in different squares or even of different generations) are
\emph{coherent} if the corresponding outward-pointing unit normal vectors coincide.

We say that the directed block $(\cC_2,L_2)$ \emph{follows} the directed block $(\cC_1,L_1)$
if $\cC_1 \cap \cC_2 = L_1$ and $L_1 \nsubseteq L_2$.
\end{definition}

We introduce a distinguished set $\pi_M$ which will parameterize certain curve
families. Let $\pi_M$ be the set of all $x \in [0,1]$ with the property that the line
$\{x\} \times \R$ does not meet the interior or left hand side of a peripheral
square removed in the construction of $S_{\ba, M}$.
Let $(\cC,L)$ be a directed block in a square $T' \in \cT_{m-1}$. There is a unique
orientation preserving isometry $i:\R^2\ra\R^2$ so that $(0,0) \in i(\cC) \subset [0,1]^2$,
and so that $i(L)$ is contained in the $x$-axis. We define $\pi_M(L)$ to be the
union of $L \cap i^{-1}(\pi_M)$ with the endpoints of $L$. It follows from the assumption $\mathbf{a} \in \ell^1$ that
\begin{equation}\label{H1-of-pi-M-L}
\cH^1(\pi_M(L)) \asymp \diam(L).
\end{equation}
Given two such sets $\pi_M(L_1)$ and $\pi_M(L_2)$ arising from isometries $i_1$
and $i_2$, there is a unique bijection $h\colon \pi_M(L_1) \to \pi_M(L_2)$ so
that $i_2 \circ h \circ i_1^{-1}$  is an order-preserving, piecewise linear bijection from
$i_1(\pi_M(L_1))\subset \R$ to $i_2(\pi_M(L_2)) \subset \R$ with a.e.\ constant derivative.
We call $h$ the natural \emph{ordered bijection}.
	
\begin{definition}
Let $E$ be a Borel subset of a side of a block $\cC$ such that $0<\cH^1(E)<\infty$.
A \emph{path family on $E$ (in $\cC$)} is
a collection of disjoint curves $\Gamma= \{\gamma_z\}_{z \in E}$ in $\cC \cap S_{\ba,M}$
with the property that $\gamma_z(0)=z$ for all $z \in E$. We also require that
the measure
$$
\nu_\Gamma(A) := \frac{1}{\cH^1(E)} \int_{E} \cH^1(A \cap \gamma_z) \ d\cH^1(z),
$$
is Borel. 
\end{definition}


As previously discussed, we will construct curve families of uniformly and quantitatively large weighted modulus joining
arbitrary pairs of points in $S_\ba$. The following notion of $\infty$-connection quantifies the degree to which
these curve families must spread out as they escape from the endpoints, measured with respect to the $L^\infty$ norm. The $L^\infty$ norm arises here by H\"older duality, as we are proving the $1$-Poincar\'e inequality. In the following section, we will introduce the analogous notion of $q$-connection for finite $q$ in order to address the case of the $p$-Poincar\'e inequality for $p>1$.

\begin{definition}
Suppose that the directed block $(\cC_2,L_2)$ follows the directed block $(\cC_1,L_1)$, and let $h \colon \pi_M(L_1) \to \pi_M(L_2)$ be the natural ordered bijection.
A path family $\Gamma$ on $\pi_M(L_1)$ in $\cC_2$ is called an
\emph{$\infty$-connection} (with constant $C$)
if $\gamma(1) = h(\gamma(0))$ for each $\gamma \in \Gamma$,
and if $\nu_\Gamma \ll \cH^2{\restrict\spt(\Gamma)}$ with
\begin{equation}\label{good h def 1}
\left\| \frac{d\nu_\Gamma}{d\cH^2} \right\|_{L^\infty(\cC_2; \cH^2)} \leq \frac{C}{\cH^1(\pi_M(L_1))} \,.
\end{equation}
\end{definition}

For the remainder of this subsection, we fix $0<m \leq M$, and
directed blocks $(\cC_2,L_2)$ following $(\cC_1,L_1)$ in a directed square $(T',L') \in \cT_{m-1}$.
We declare the \emph{central column} of $T'$ to be the central row or column of $T'$ that intersects $L'$.

We now state our building block lemmas; see figures \ref{fig:doubling-1} and \ref{fig:turning-straight}.

\begin{figure}
\begin{minipage}{\textwidth}
\centering
\psfrag{C1}{$\mathcal{C}_1$}
\psfrag{C2}{$\mathcal{C}_2$}
\psfrag{L1}{$L_1$}
\psfrag{L2}{$L_2$}
\raisebox{-\height}{\includegraphics[width=0.3\textwidth]{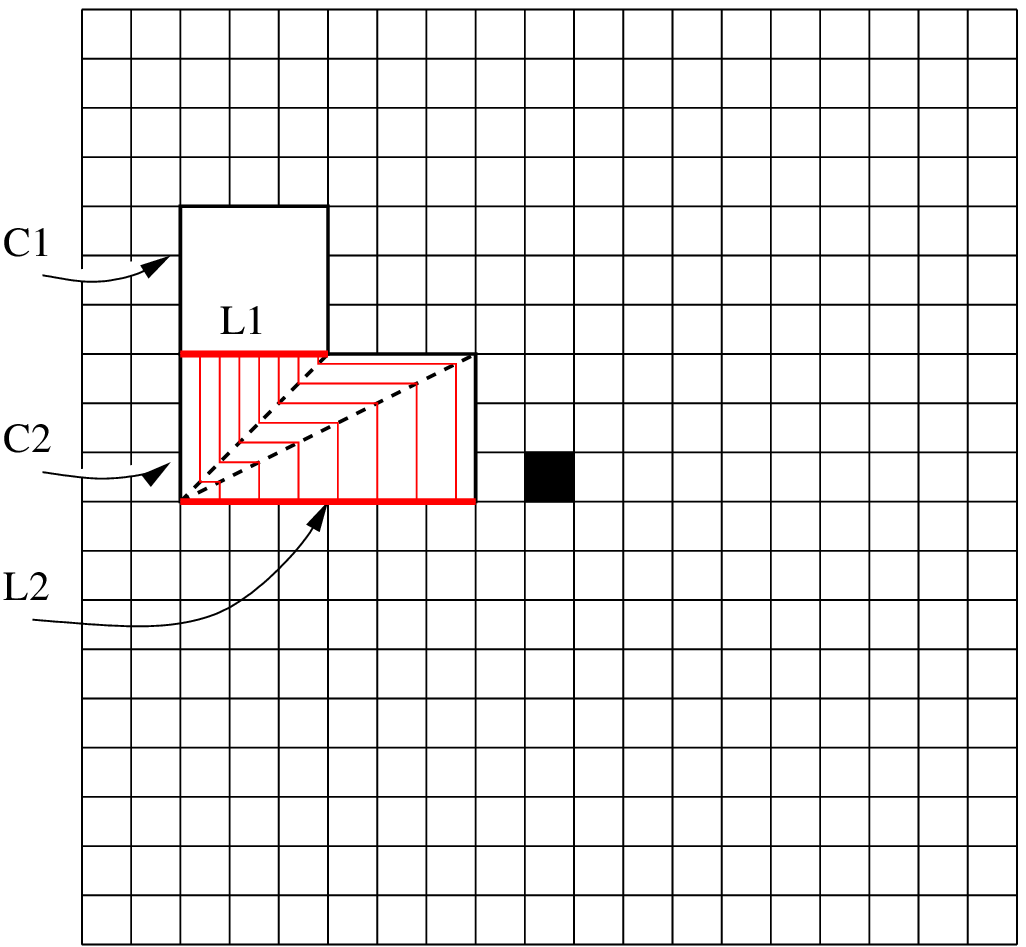}}
\hspace{50pt}
\raisebox{-\height}{\includegraphics[width=0.32\textwidth]{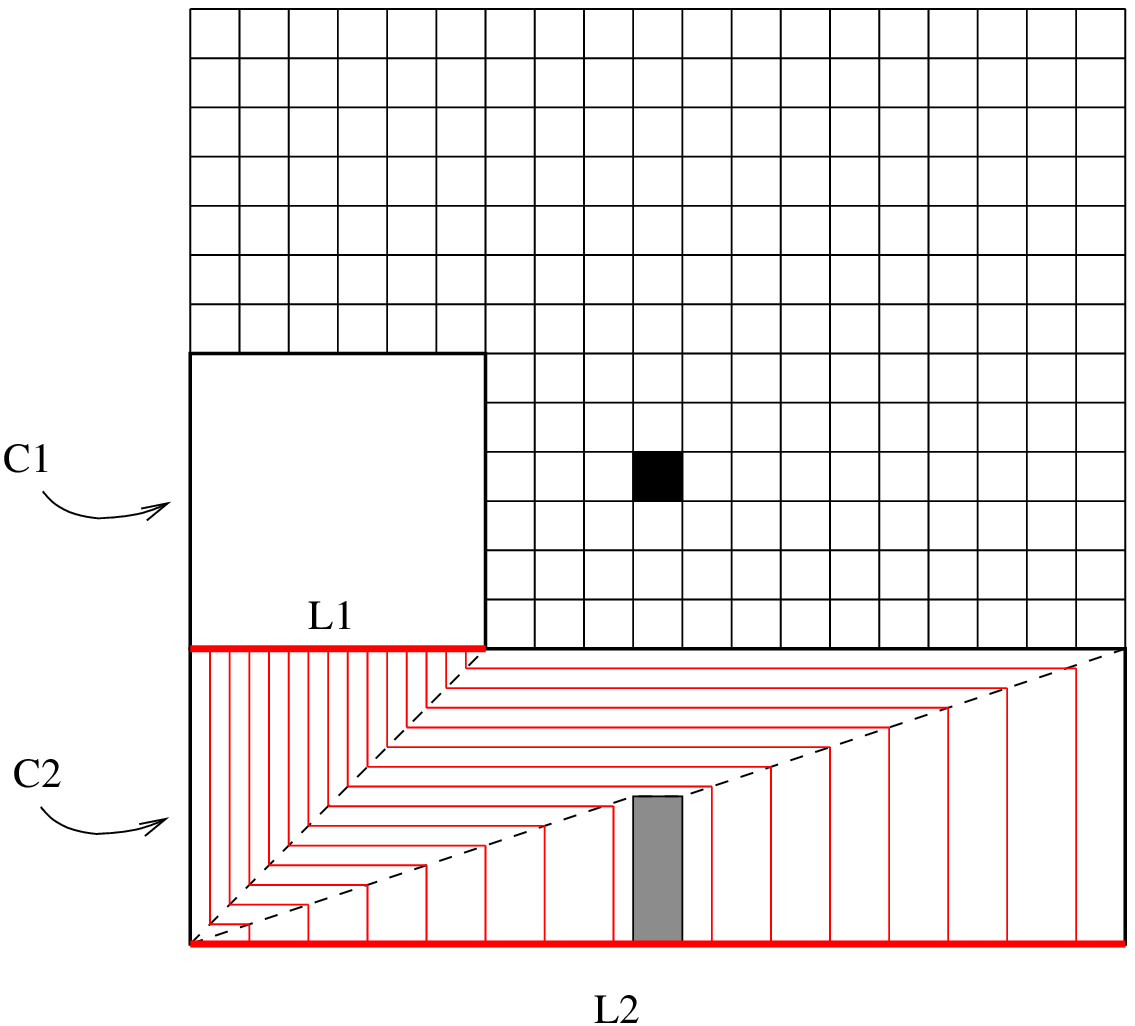}}
\end{minipage}
\caption{(a) Expanding to a larger interval; (b) Expanding to the parent generation} 
\label{fig:doubling-1}
\end{figure}

\begin{lemma}[Expanding]\label{doubling}  Suppose that
\begin{itemize}
\item $(\cC_1,L_1)$ and $(\cC_2,L_2)$ are coherent with $(T',L')$,
\item the sides of $\cC_2$ perpendicular to $L_2$ have length equal to
	that of $L_1$, and
\item it holds that $\cH^1(L_2)/\cH^1(L_1) \leq 10$.
\end{itemize}
Then there is an $\infty$-connection $\Gamma$ in $\cC_2$ with constant $C=C(\| \ba \|_{1})$.
\end{lemma}


\begin{lemma}[Expanding to the parent generation]\label{doubling2}  Suppose that
\begin{itemize}
\item $(\cC_1,L_1)$ is coherent with $(T',L_2)$, where $L_2=L'$,
\item the length of $L_1$ is equal to the length of an edge of $\cC_2$ perpendicular to $L_2$,
\item $L_1$ intersects an edge of $T'$ perpendicular to $L'$, and
\item it holds that $\cH^1(L_2)/\cH^1(L_1) \leq 10$.
\end{itemize}
Then there is an $\infty$-connection $\Gamma$ in $\cC_2$ with constant $C=C(\| \ba \|_{1} )$.
\end{lemma}

%
%

\begin{lemma}[Turning]\label{turning}
Suppose that
\begin{itemize}
\item $L_1$ and $L_2$ are perpendicular, and
\item all edges of $\cC_2$ have length equal to the length of $L_1$.
\end{itemize}
Then there is an $\infty$-connection $\Gamma$ in $\cC_2$ with constant $C=C(\| \ba \|_{1})$.
\end{lemma}

\begin{lemma}[Going straight]\label{doubling4}
Suppose that
\begin{itemize}
\item $(\cC_1,L_1)$ and $(\cC_2,L_2)$ are coherent with $(T',L')$,
\item $L_1$ and $L_2$ are of equal length, and
\item the sides of $\cC_2$ perpendicular to $L_2$ have length in $[\cH^1(L_1)/2, 10 \cH^1(L_1)]$.
\end{itemize}
Then there is an $\infty$-connection $\Gamma$ in $\cC_2$ with constant $C=C(\| \ba \|_{1})$.
\end{lemma}


\begin{figure}
\begin{minipage}{\textwidth}
\centering
\psfrag{C1}{$\mathcal{C}_1$}
\psfrag{C2}{$\mathcal{C}_2$}
\psfrag{L1}{$L_1$}
\psfrag{L2}{$L_2$}
\raisebox{-\height}{\includegraphics[height=0.3\textwidth]{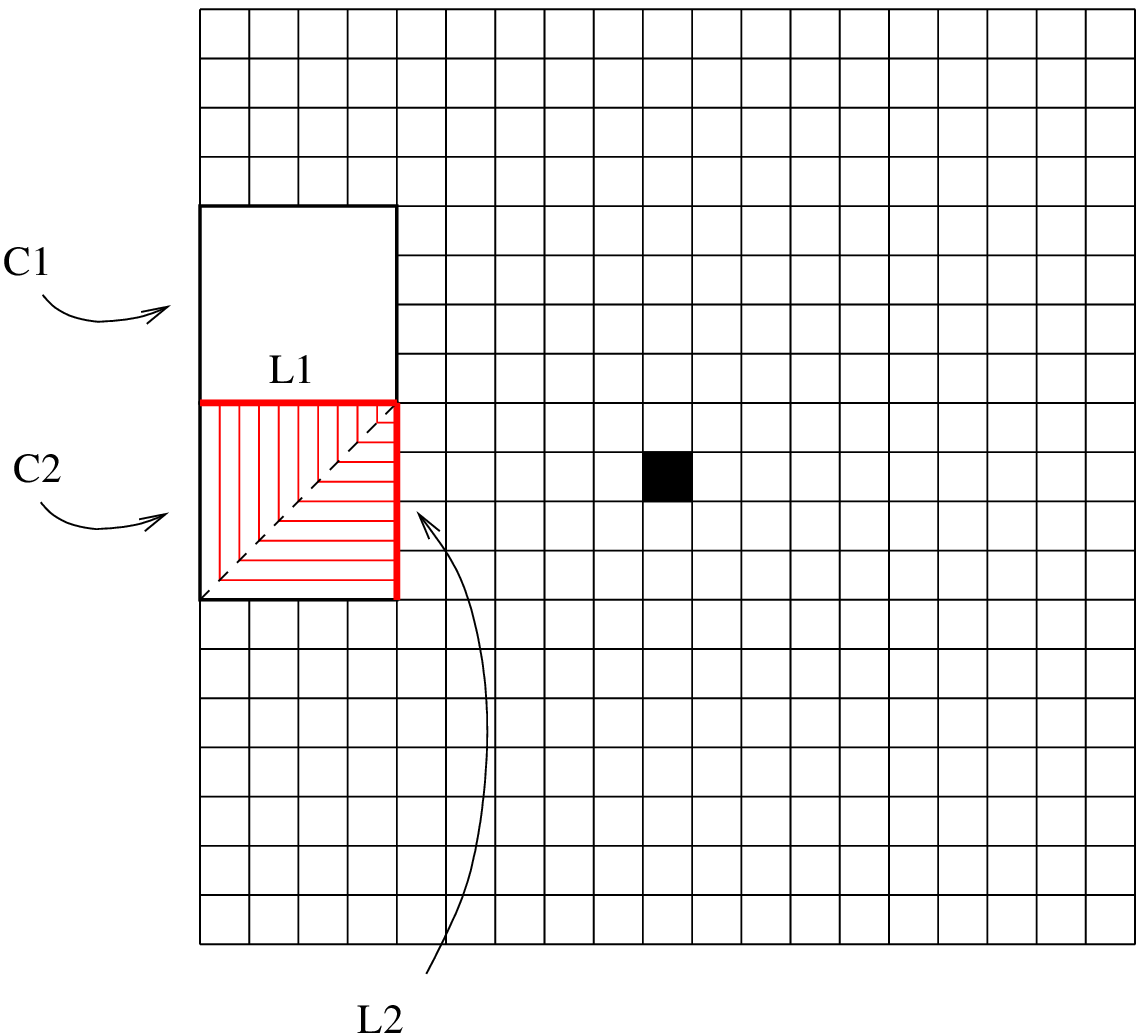}}
\hspace{50pt}
\raisebox{-\height}{\includegraphics[height=0.3\textwidth]{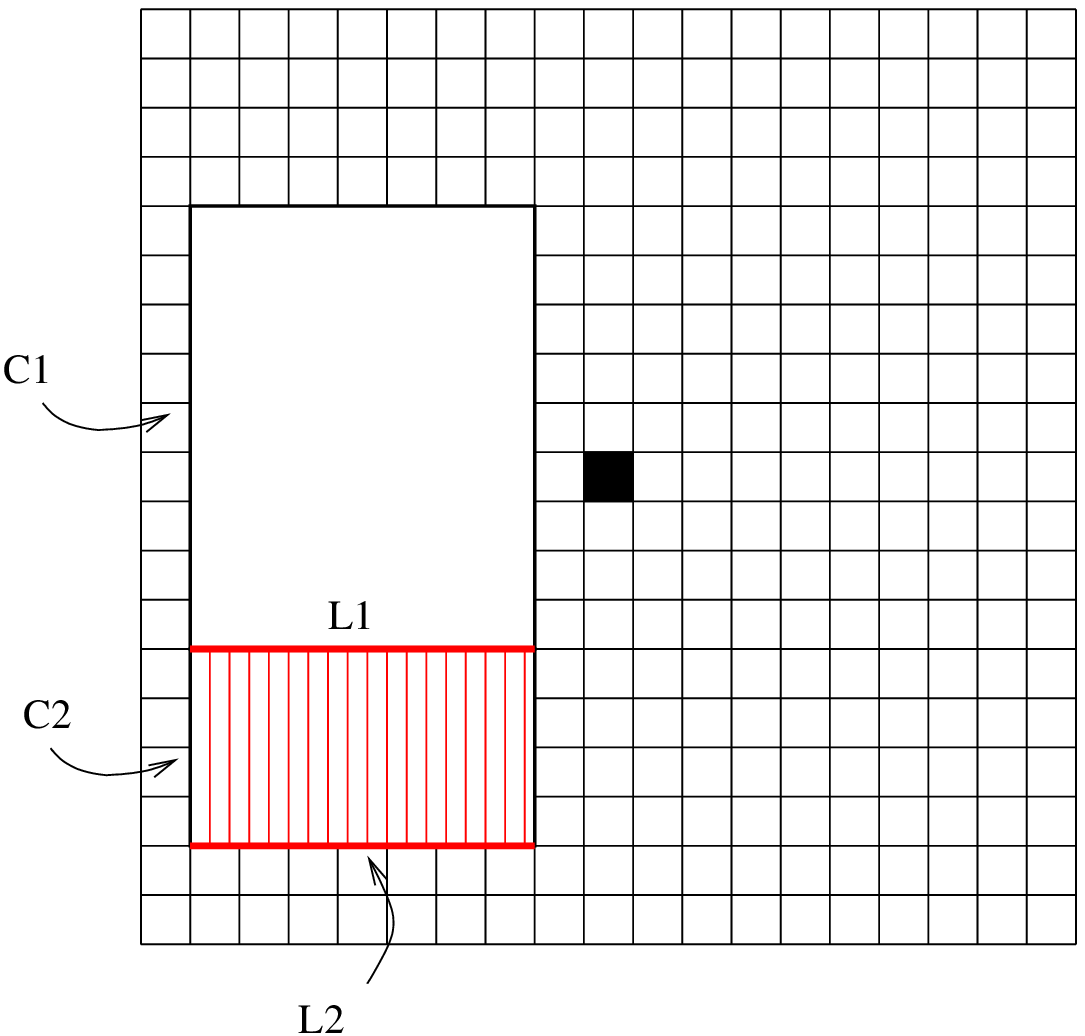}}
\end{minipage}
\caption{(a) Turning; (b) Going straight}\label{fig:turning-straight}
\end{figure}

\begin{proof}[Proof of Lemma~\ref{doubling}]
We assume that $T'=[0,s_{m-1}]^2$ and  $L'=[0,s_{m-1}]\times\{0\}$.
We further assume that $\cC_2 = [0,as_m]\times[0,bs_m]$ and $\cC_1=[0,bs_m]\times[bs_m,cs_m]$, as well as that $L_2= [0,as_m] \times \{0\}$ and $L_1 = [0,as_m]\times \{bs_m\}$.
Here $a$, $b$, and $c$ are positive integers with $a \geq b$ and $c > b$.
For ease of notation, set $E=\pi_M(L_1)$ and $F=\pi_M(L_2)$.
Note that the natural ordered bijection
$h \colon E \to F$ satisfies
$$h'(z)=\frac{\cH^1(F)}{\cH^1(E)} \geq 1 $$
for every interior point $z$ of $E$, i.e., for all but finitely many points.
In the case that $m=M$, the function $h$ is affine.

We now define a path family $\Gamma$ on $E$. Given $z=(u,bs_m) \in E$, let
$\gamma^1_{z}$ be the vertical line segment connecting $(u,bs_m)$ to
$(u,u)$, let $\gamma^2_z$ be the horizontal line segment connecting
$(u,u)$ to $h(z)+(0,u)$, and let $\gamma^3_z$ be the vertical line
segment connecting $h(z)+(0,u)$ to $h(z)$. Let $\gamma_{z}$ be the
concatenation of $\gamma_{z}^1$, $\gamma_{z}^2$ and $\gamma_{z}^3$;
then $\gamma_{z} \subseteq S_{\ba,M}\cap \cC_2$.  Let
$\Gamma=\{\gamma_{z}:z \in E\}$.

We may write the support of $\Gamma$ as the union of the supports of
the curve families
$$
\Gamma^i= \{\gamma^i_z\}_{z \in E}, \qquad i=1,2,3.
$$
Given a set $A$ contained in the support of $\Gamma$, we write $A^i=A
\cap \spt \Gamma^i$.

For $i=1$ or $2$, Fubini's theorem yields $\cH^2(A^i) = \cH^1(E)\nu_\Gamma(A^i)$.
For $i=3$, a simple change of variables shows that
\[
	\cH^2(A^3) = \cH^1(F) \nu_\Gamma(A^3) \geq \cH^1(E) \nu_\Gamma(A^3).
\]
Together, this shows that $\Gamma$ is an $\infty$-connection with constant $1$.
\end{proof}

We omit the proofs of Lemmas \ref{doubling2}--\ref{doubling4} as they are nearly identical to that of Lemma~\ref{doubling}.

\subsection{Verification of the $1$-Poincar\'e inequality}\label{subsec:pi}

We are now ready to prove the following proposition.

\begin{proposition}\label{yesPIl1}
Suppose that $\ba \in \ell^1$. Then $S_\ba$ admits a $1$-Poincar\'e inequality.
\end{proposition}

\begin{proof}
As mentioned in the introduction to this section, it is enough to prove that
for fixed $M$, the precarpet $(S_{\ba,M},d,\mu_M)$ supports a $1$-Poincar\'e inequality with
constant independent of $M$. Towards this end, we take advantage of Theorem
\ref{keith-theorem-1}. The constant $C_2$ in Keith's condition \eqref{keith-equation-1}
will be an absolute quantity which could in principle be computed explicitly as
a fixed multiple of the implicit multiplicative constant in conditions (5) and (6) of
Lemma~\ref{lem-block-chain} below.
On the other hand, the constant $C_1$ in \eqref{keith-equation-1} depends on
$C_0$ in Lemma~\ref{lem-block-chain}, which in turn depends on the constants
in Lemmas~\ref{doubling}--\ref{doubling4} above.
In particular, $C_1$ will depend heavily on $||\ba||_{1}$.

In order to verify the condition in Theorem \ref{keith-theorem-1}, let us fix $x,y \in S_{\ba,M}$ with $x \neq y$.
If $|x-y| < 10s_M$, then we are in the Euclidean situation with possibly a square removed nearby, so \eqref{keith-equation-1} holds with uniform constants. Let us assume that for some $m \leq M$ we have $10s_m \leq |x-y| < 10s_{m-1}$.

The implicit multiplicative constants in conditions (5) and (6) below
are fixed, universal quantities which could be explicitly computed; to
simplify the story we have spared the reader any explicit calculation.
For instance, both of these multiplicative constants can be chosen to be $100$.

\begin{lemma}\label{lem-block-chain}
There exist integers $K_- < 0 < K_+$,
a sequence of directed blocks $\{(\cC_i,L_i)\}_{i=K_-}^{K_+}$,
and path families $\Gamma_i$ each supported on $\cC_i$,
with the following properties, for some uniform constant $C_0$.
\begin{enumerate}
\item $\cC_{K_-}$ and $\cC_{K_+}$ are $2$ by $1$ blocks (or $1$ by $2$ blocks)
	on scale $s_M$ containing $x$ and $y$ respectively.
\item $\dist(x, L_{K_-}) \geq s_M/2$ and $\dist(y, L_{K_+}) \geq s_M/2$.
\item $\Gamma_{K_-}$ and $\Gamma_{K_+}$ consist of the collection of straight lines joining $x$ to $L_{K_-}$ and $y$ to $L_{K_+}$ respectively.
\item\label{good con exist} For each $i=(K_-+1), \ldots, -1$, $(\cC_i,L_i)$ follows $(\cC_{i-1}, L_{i-1})$;
	for each $i=1, \ldots, K_+-1$, $(\cC_i,L_i)$ follows $(\cC_{i+1}, L_{i+1})$;
	$(\cC_0,L_1)$ follows $(\cC_{-1},L_{-1})$, and
	$(\cC_0,L_{-1})$ follows $(\cC_1, L_1)$.
	In each case, $\Gamma_i$ is an $\infty$-connection with constant $C_0$.
\item\label{not far} For each $i=(K_-+1), \ldots, (K_+-1)$,
\begin{equation*}
\min\{\dist(x,\cC_i), \dist(y,\cC_i)\} \asymp \diam(\cC_i) \asymp \diam(L_i).
\end{equation*}
\item $\sum_{i=K_-}^{K_+} \diam(\cC_i) \asymp |x-y|$.
\item\label{no overlap} The blocks $\cC_{K_-},\hdots,\cC_{K_+}$ are essentially disjoint.
\end{enumerate}
\end{lemma}

We postpone the proof of this lemma.

The path families $\Gamma_{K_-}, \ldots, \Gamma_{K_+}$ concatenate together by gluing paths using the natural
ordered bijection on each block.  This gives a path family $\Gamma$
consisting of pairwise disjoint, rectifiable curves joining $x$ to $y$, carrying a
probability measure $\sigma = \sigma_\Gamma$ on $\Gamma$ which agrees with $\sigma_{\Gamma_i}$ for each $i$ on $\cC_i$.
The measure $\nu = \nu_\Gamma$ on the support of $\Gamma$ defined by
\begin{equation}\label{nu-Gamma}
\nu(A) = \int_\Gamma \cH^1(A \cap \gamma) d\sigma
\end{equation}
restricts to $\nu_{\Gamma_i}$ on each $\cC_i$.

This measure $\nu$ is absolutely continuous with respect to $\mu_{xy}$.
For $i = {K_-}, \ldots, {K_+}$, we have the following bound on the
Radon--Nikodym derivative $\frac{d\nu}{d\mu_{xy}}$.
We have $\diam(\cC_i) \asymp \min\{\dist(x,\cC_i), \dist(y,\cC_i)\}$, and so
$\mu_{xy} \asymp \frac1{\diam(\cC_i)} \cH^2$ on $\cC_i$.  Therefore
\begin{align*}
	\left\| \frac{d\nu}{d\mu_{xy}} \right\|_{L^\infty(\cC_i; \mu_{xy})}
		& \asymp \left\| \frac{d\nu}{d\cH^2} \right\|_{L^\infty(\cC_i; \cH^2)} \diam(\cC_i) \\
		& \lesssim \left( \cH^1(\pi_M(L_i)) \right)^{-1} \diam(\cC_i) \lesssim 1.
\end{align*}
This bound also holds on $\cC_{K_-}$ and $\cC_{K_+}$: note that on $\spt(\Gamma_{K_-}) \subset \cC_{K_-}$, both
$\nu$ and $\mu_{xy}$ are comparable to the measure $A \mapsto \int_A  1/|x-z| \ d\cH^2(z)$.
An elementary calculation gives
\[
	\left\| \frac{d\nu}{d\mu_{xy}} \right\|_{L^\infty(\cC_{K_-}; \mu_{xy})} \lesssim 1.
\]
An analogous argument proves the bound for $\cC_{K_+}$.

Let $\rho$ be admissible for $\Gamma$. Then
\begin{align*}
1 &\le \int_{\Gamma} \int_\gamma \rho \, ds \, d\sigma(\gamma)
 = \int_{\spt\Gamma} \rho \ d\nu \\
 &= \int_{S_{\ba,M}} \rho \, \frac{d\nu}{d\mu_{xy}} \, d\mu_{xy}
 \le \| \rho \|_{L^1(\mu_{xy})}    \    \left\| \frac{d\nu}{d\mu_{xy}} \right\|_{L^\infty(\mu_{xy})}
 \lesssim \| \rho \|_{L^1(S_{\ba,M},\mu_{xy})} \, .
\end{align*}
Thus \eqref{keith-equation-1} holds.  This completes the proof that $S_{\ba}$ admits a $1$-Poincar\'e inequality.
\end{proof}

It remains to construct the block family described in the statement of Lemma~\ref{lem-block-chain}.

\begin{proof}[Proof of Lemma~\ref{lem-block-chain}]
Recall that $m \leq M$ is chosen so that $10s_m \leq |x-y| < 10s_{m-1}$.

We construct the sequence of blocks and path families by induction.
To make the proof more readable, we outline the basic steps,
and leave the details to the reader.
The basic idea is as follows: we use the expanding and turning lemmas
(Lemmas~\ref{doubling}--\ref{doubling4})
to build a sequence of blocks which grow in size at a linear rate as they
to travel away from $x$ until reaching size $\sim |x-y|/100$.
We do the same for $y$, and then join up the two sequences using the same lemmas.

We now describe the construction in more detail, assuming \eqref{aia} in order to simplify
the argument.

First, $x \in T \subset T'$ for some $T \in \cT_M$ and $T' \in \cT_{M-1}$,
and we can find a $1$ by $2$ (or $2$ by $1$) directed block $(\cC_0, L_0)$ so that $T \subset \cC_0 \subset T'$,
and $L_0$ is the short edge of $\cC_0$ furthest from $x$, and $L_0$ does not meet the boundary of $T'$ or the
square of side $s_M$ removed from $T'$ in more than one point.
This gives us our first directed block $(\cC_0, L_0)$, which satisfies conditions (1) and (2), and
we define $\Gamma_0$ according to condition (3).

The induction step is as follows.  We assume that we have a sequence of blocks contained in a
$1$ by $2$ (or $2$ by $1$) directed block $(\cC_-,L_-)$ on scale $s_n$,
which is contained in some $T' \in \cT_{n-1}$ in such a way that the short edge $L_-$ does not meet the
the boundary of $T'$, or of the central removed square of $T'$, in more than one point.

We choose a $1$ by $2$ (or $2$ by $1$) directed block $(\cC_+,L_+)$ on scale $s_{n-1}$ so
that $T' \subset \cC_+$, and $L_+$ is the short edge of $\cC_+$ furthest from $\cC_-$,
and $L_+$ does not meet the boundary or centrally removed square of the square $T'' \in \cT_{s_{n-2}}$
with $T' \subset T''$.

We now build a sequence of directed blocks
$$
(\cC_-,L_-)=(\cC_0',L_0'), (\cC_1',L_1'), \ldots, (\cC_t', L_t')
$$
inside $\cC_+$, where $L_t' = L_+$, and where for $i=1,\ldots,t$,
$(\cC_i,L_i)$ follows $(\cC_{i-1},L_{i-1})$. Moreover, these blocks
satisfy conditions (4),(5) and (7), and their diameters sum to $\asymp
\dist(\cC_-, L_+)$.

The sequence of directed blocks is constructed using the following
algorithm. See figures \ref{fig:blockchain1} and \ref{fig:blockchain2}
for an illustration.

\begin{enumerate}
	\item Use turning Lemma \ref{turning} and straight Lemma
          \ref{doubling4} between zero and six times to build a chain
          of blocks so that the last block is coherent with $L_+$,
          closer to $L_+$ than $\cC_-$, and is not contained in the
          central column of $(\cC_+,L_+)$.
	\item Use expanding Lemma \ref{doubling} repeatedly to double
          away from the central column until, with an expansion by a
          factor between two and four, the long edge of~$\cC_+$ is
          reached.
	\item Use expanding Lemma \ref{doubling} repeatedly to double
          towards the central column, until of size $\geq
          \diam(L_+)/5$.
	\item Go straight (Lemma \ref{doubling4}) until past the last
          removed square of the central column.
	\item Expand by a factor less than five (Lemma
          \ref{doubling2}), with the last edge $L_+$.
\end{enumerate}

\begin{figure}
\begin{center}
\includegraphics[width=.25\textwidth]{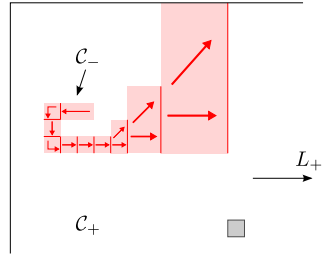}
\caption{Steps (1) and (2) of the algorithm}
\label{fig:blockchain1}
\end{center}
\end{figure}

\begin{figure}
\begin{center}
\includegraphics[width=.5\textwidth]{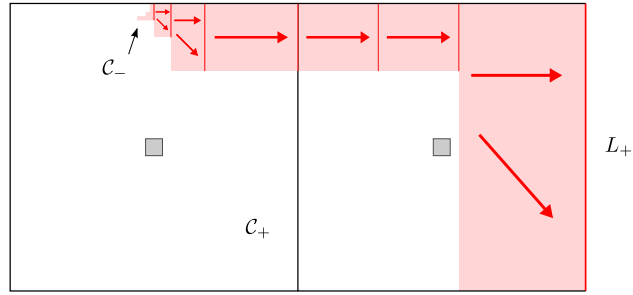}
\caption{Steps (3)--(5) of the algorithm}
\label{fig:blockchain2}
\end{center}
\end{figure}

We repeat this construction, growing in scale each time, until we are
on scale $s_m$, then again until we have a block of size $|x-y|/100$,
at a distance less than $|x-y|/10$ from $x$.  This gives us most of
the sequence of blocks with negative index.

We do the same for $y$, getting most of the (reverse) sequence of
blocks with positive index, then join the two chains together using
Lemmas \ref{turning} and \ref{doubling4}.
\end{proof}


\section{Validity of the Poincar\'e inequality: the case $\ba \in \ell^2 \setminus \ell^1$}\label{sec:validity-pi-2}

In this section we address the case $\ba \in \ell^2 \setminus \ell^1$. Our goal (Proposition \ref{yesPIl2}) is to prove that in this case the carpet $S_\ba$ verifies the $p$-Poincar\'e inequality for each $p>1$.

The overall structure of the proof is similar to that in the previous section. According to Theorems \ref{keith-theorem-1} and \ref{keith-theorem-2}, it suffices to verify the weighted modulus lower bound \eqref{keith-equation-1} on the precarpets $S_{\ba,M}$ with constants independent of $M$. To this end, we construct suitable curve families joining pairs of points in $S_{\ba,M}$ by concatenating curve families in suitable blocks. The construction of these curve families relies on `building block' lemmas analogous to Lemmas \ref{doubling}--\ref{doubling4} from the previous section. However, the proofs of the `building block' lemmas in the cases $\ba \in \ell^1$ and $\ba \in \ell^2\setminus\ell^1$ are quite different, and require two completely different methods.

Recall that in the case $\ba \in \ell^1$, the thinnest part of the carpet has positive length, so we were able to use combinatorial constructions to build modular curve families avoiding all omitted squares in the precarpet $S_{\ba,M}$.

In the case $\ba \in \ell^2$, considering only curves which pass through a thin part of the carpet cannot suffice. We instead
construct a larger curve family using ``bending'' machinery which we will develop in subsection~\ref{subsec:Bending}. We use this bending machinery to verify the $p$-Poincar\'e inequality for $p>1$ via Keith's theorem \ref{keith-theorem-1}. (This machinery does not apply to the previous case, as the H\"older conjugate exponent $q$ must be finite in order to use Proposition \ref{prop-weight-bound}.)

The bending machinery, and hence our building block lemmas for $p>1$, require the following condition:
\begin{equation}\label{aia2}
a_i < a_{**} \qquad \mbox{for all $i$,}
\end{equation}
where $a_{**} \in (0,a_*)$ is a fixed, universal constant whose value is determined in subsection~\ref{subsec:Bending}.
Here as always we have $\ba = (a_1, a_2, \ldots) \in \ell^2$.

As before, Theorem \ref{gluing-theorem} permits us to reduce to the case that \eqref{aia2} holds,
because any $S_\ba$ with $\ba \in \ell^2$ is the finite union of
similar copies of some carpet $S_{\ba'}$, where $\ba' \in \ell^2$ and
all entries of $\ba'$ are less than $a_{**}$, glued along their
boundaries. For the remainder of this section, we impose assumption \eqref{aia2}.

Fix $p>1$ with H\"older conjugate $q<\infty$. We first define a notion of $q$-connection which generalizes the previous notion of $\infty$-connection. As we will employ bending machinery in this section rather than restricting to paths in the thin part of the carpet, the choice of the distinguished set $\pi_M$ is simpler in this setting. For a directed block $(\cC,L)$ in a square $T'$, we let $\pi_M(L)=L$ and we let $h$ be the corresponding bijection (which now works out to be the restriction of an affine map).

\begin{definition}
Suppose that the directed block $(\cC_2,L_2)$ follows the directed block $(\cC_1,L_1)$, and let $h \colon \pi_M(L_1) \to \pi_M(L_2)$ be the natural ordered bijection.
A path family $\Gamma$ on $\pi_M(L_1)$ in $\cC_2$ is called an
\emph{$q$-connection} (with constant $C$)
if $\gamma(1) = h(\gamma(0))$ for each $\gamma \in \Gamma$,
and if $\nu_\Gamma \ll \cH^2{\restrict\spt(\Gamma)}$ with
\begin{equation}\label{good h def 2}
\left\| \frac{d\nu_\Gamma}{d\cH^2} \right\|_{L^q(\cC_2; \cH^2)} \leq C \left( \cH^1(\pi_M(L_1)) \right)^{-1+2/q}.
\end{equation}
\end{definition}

The mysterious exponent on the right hand side of \eqref{good h def 2} can be justified by a dimensional analysis. Note that the measure $\nu_\Gamma$ is homogeneous of degree $1$ relative to the scalings $x\mapsto\lambda x$, $\lambda>0$, of $\R^2$. Hence the Radon--Nikodym derivative $\tfrac{d\nu_\Gamma}{d\cH^2}$ is homogeneous of degree $-1$. Since the $L^q$ norm is computed with respect to Lebesgue measure, it follows that the left hand side of \eqref{good h def 2} is homogeneous of degree $-1+\tfrac2q$.

In this setting of $\ba \in \ell^2$, our building block lemmas take the following form.  Recall that we fix $0<m \leq M$, and
are given directed blocks $(\cC_2,L_2)$ following $(\cC_1,L_1)$ in a directed square $(T',L') \in \cT_{m-1}$.

\begin{lemma}[Expanding]\label{doubling-q}  Suppose that 
\begin{itemize}
\item $(\cC_1,L_1)$ and $(\cC_2,L_2)$ are coherent with $(T',L')$,
\item the sides of $\cC_2$ perpendicular to $L_2$ have length equal to
	that of $L_1$, and
\item it holds that $\cH^1(L_2)/\cH^1(L_1) \leq 10$.
\end{itemize}
Then there is a $q$-connection $\Gamma$ in $\cC_2$ with constant $C=C(\| \ba \|_{2})$.
\end{lemma}

\begin{lemma}[Expanding to the parent generation]\label{doubling2-q}  Suppose that
\begin{itemize}
\item $(\cC_1,L_1)$ is coherent with $(T',L_2)$, where $L_2=L'$,
\item the length of $L_1$ is equal to the length of an edge of $\cC_2$ perpendicular to $L_2$,
\item $L_1$ intersects an edge of $T'$ perpendicular to $L'$, and
\item it holds that $\cH^1(L_2)/\cH^1(L_1) \leq 10$.
\end{itemize}
Then there is a $q$-connection $\Gamma$ in $\cC_2$ with constant $C=C(\| \ba \|_{2})$.
\end{lemma}

\begin{lemma}[Turning]\label{turning-q}
Suppose that
\begin{itemize}
\item $L_1$ and $L_2$ are perpendicular, and
\item all edges of $\cC_2$ have length equal to the length of $L_1$.
\end{itemize}
Then there is a $q$-connection $\Gamma$ in $\cC_2$ with constant $C=C(\| \ba \|_{2})$.
\end{lemma}

\begin{lemma}[Going straight]\label{doubling4-q}
Suppose that
\begin{itemize}
\item $(\cC_1,L_1)$ and $(\cC_2,L_2)$ are coherent with $(T',L')$,
\item $L_1$ and $L_2$ are of equal length, and
\item the sides of $\cC_2$ perpendicular to $L_2$ have length in $[\cH^1(L_1)/2, 10 \cH^1(L_1)]$.
\end{itemize}
Then there is a $q$-connection $\Gamma$ in $\cC_2$ with constant $C=C(\| \ba \|_{2})$.
\end{lemma}

To produce the desired $q$-connection in these lemmas, we will use the bending machinery which we develop in
subsection~\ref{subsec:Bending}. We therefore postpone the proof of these lemmas until that time.
Assuming for the moment their validity, we complete the proof of the $p$-Poincar\'e inequality.

\subsection{Verification of the $p$-Poincar\'e inequality for $p>1$}\label{subsec:pi2}

We prove the following proposition.

\begin{proposition}\label{yesPIl2}
Suppose that $\ba \in \ell^2 $ and $p>1$. Then $S_\ba$ admits a $p$-Poincar\'e inequality.
\end{proposition}

\begin{proof}

This proof is virtually identical to that of Proposition \ref{yesPIl1}, 
using the new building block lemmas \ref{doubling-q}--\ref{doubling4-q}. 
Lemma \ref{lem-block-chain} remains the same, except that condition (\ref{good con exist}) is replaced by
\begin{enumerate}
\item[(4')]\label{good con exist 2} For each $i=(K_-+1), \ldots, -1$, $(\cC_i,L_i)$ follows $(\cC_{i-1}, L_{i-1})$;
	for each $i=1, \ldots, K_+-1$, $(\cC_i,L_i)$ follows $(\cC_{i+1}, L_{i+1})$;
	$(\cC_0,L_1)$ follows $(\cC_{-1},L_{-1})$, and
	$(\cC_0,L_{-1})$ follows $(\cC_1, L_1)$.
	In each case, $\Gamma_i$ is a $q$-connection with constant $C_0$.
\end{enumerate}
This condition follows by using Lemmas \ref{doubling-q}--\ref{doubling4-q} and exactly the same argument as before.

We now complete the proof of Proposition \ref{yesPIl2}. First, concatenate the path families $\Gamma_{K_-}, \ldots, \Gamma_{K_+}$ by gluing paths using the natural ordered bijection on each block.  This gives a path family $\Gamma$
consisting of pairwise disjoint, rectifiable curves joining $x$ to $y$, carrying a
probability measure $\sigma = \sigma_\Gamma$ on $\Gamma$ which agrees with $\sigma_{\Gamma_i}$ for each $i$ on $\cC_i$.
The measure $\nu = \nu_\Gamma$ on the support of $\Gamma$ defined as in \eqref{nu-Gamma} restricts to $\nu_{\Gamma_i}$ on each $\cC_i$, and is absolutely continuous with respect to $\mu_{xy}$. For $i = {K_-}, \ldots, {K_+}$, we have the following bound on the
Radon--Nikodym derivative $\frac{d\nu}{d\mu_{xy}}$:
\begin{align*}
	\left\| \frac{d\nu}{d\mu_{xy}} \right\|_{L^q(\cC_i; \mu_{xy})}
		& \asymp \left\| \frac{d\nu}{d\cH^2} \right\|_{L^q(\cC_i; \cH^2)} \diam(\cC_i)^{1-1/q} \\
		& \lesssim \left( \cH^1(\pi_M(L_i)) \right)^{-1+2/q} \diam(\cC_i)^{1-1/q} \\
		& \lesssim \diam(\cC_i)^{-1+2/q} \diam(\cC_i)^{1-1/q} = \diam(\cC_i)^{1/q}.
\end{align*}
This bound also holds on $\cC_{K_-}$ and $\cC_{K_+}$: note that on $\spt(\Gamma_{K_-}) \subset \cC_{K_-}$, both
$\nu$ and $\mu_{xy}$ are comparable to the measure $A \mapsto \int_A  1/|x-z| \ d\cH^2(z)$.
An elementary calculation gives
\[
	\left\| \frac{d\nu}{d\mu_{xy}} \right\|_{L^q(\cC_{K_-}; \mu_{xy})} \lesssim s_M^{1/q} \asymp \diam(\cC_{K_-})^{1/q}.
\]
An analogous argument proves the bound for $\cC_{K_+}$. Summing over $i$ gives
$$
\left\| \frac{d\nu}{d\mu_{xy}} \right\|_{L^q(\mu_{xy})}^q = \sum_{i={K_-}}^{K_+} \left\| \frac{d\nu}{d\mu_{xy}} \right\|_{L^q(\cC_i, \mu_{xy})}^q \lesssim \sum_{i={K_-}}^{K_+} \diam(\cC_i) \lesssim |x-y|.		
$$
If $\rho$ is admissible for $\Gamma$, then
\begin{align*}
1 &\le \int_{\Gamma} \int_\gamma \rho \, ds \, d\sigma(\gamma)
 = \int_{\spt\Gamma} \rho \ d\nu \\
 &= \int_{S_{\ba,M}} \rho \, \frac{d\nu}{d\mu_{xy}} \, d\mu_{xy}
 \le \| \rho \|_{L^p(\mu_{xy})}    \    \left\| \frac{d\nu}{d\mu_{xy}} \right\|_{L^q(\mu_{xy})}
 \lesssim \| \rho \|_{L^p(\mu_{xy})}  \  |x-y|^{1/q}.
\end{align*}
Consequently
\[
	\int_{S_{\ba,M}} \rho^p \, d\mu_{xy} \gtrsim |x-y|^{-p/q} = |x-y|^{1-p}
\]
and so \eqref{keith-equation-1} holds. This completes the proof of Proposition \ref{yesPIl2}, 
modulo the building block lemmas~\ref{doubling-q}--\ref{doubling4-q}.
\end{proof}


\subsection{Bending curve families}\label{subsec:Bending}

In this section, we introduce the bending machinery needed to prove Lemmas \ref{doubling-q}--\ref{doubling4-q}.
The methods used allow us to build curve families of positive modulus in a wide class of compact planar sets,
as illustrated by the following theorem, which gives a general sufficient condition for such curve families.

\begin{theorem}\label{modulus-extension-theorem}
Let $D\subset\R^2$ be the closure of a domain and let $\ba = (a_1,a_2,\ldots) \in \ell^2$ with $a_m \in (0,1)$ for all $m$. For
each $m$, let $s_m = \prod_{j=1}^m a_j$, and let $\cU_m$ be a family of disjoint open subsets of $D$. Assume that
the following two conditions are satisfied:
\begin{itemize}
\item for all $U \in \cU_m$, $\diam U \le 2 s_m$, and
\item for all $U\in \cU_m$ and all $V\in \{\partial D \} \cup \cU_1\cup\cU_2\cup\cdots\cup\cU_m$ with $V \neq U$, we have  $\dist(U,V) \ge \tfrac25s_{m-1}$.
\end{itemize}
Let $S_M := D \setminus \cup\{U : U \in \cU_1 \cup \cdots \cup \cU_M\}$
and
$$
S = \bigcap_{M\ge 0} S_M.
$$
Then for all $p>1$ and all relatively open balls $B\subset S$, there exists a curve family contained in $B$ with positive $p$-modulus
with respect to the measure $\cH^2$ restricted to $S$.
\end{theorem}

The coefficients $2$ and $\tfrac25$ in Theorem \ref{modulus-extension-theorem} have been fixed for the sake of definiteness and can be varied without changing the result.

In the setting of the carpets $S_\ba$, our arguments give the following corollary, independent of Theorem~\ref{thmB}.

\begin{corollary}\label{thm-l2-mod}
For any $p>1$ and $\ba \in \ell^2$, there exists a positive constant $C = C(p,\ba)$ so that the $p$-modulus of the curve family joining the left hand edge to the right hand edge of $S_\ba$ is at least $C$. If $\ba$ is monotone decreasing, then $C = C(p,\|\ba\|_{2})$.
\end{corollary}


The basic idea of the construction in this section is as follows. We
present an algorithm which accepts as input a family of curves in the
plane and which yields as output a new family of curves which avoids a
prespecified obstacle at a small quantitative multiplicative cost to
the $p$-modulus. We apply this algorithm recursively to avoid all of
the omitted sets. The algorithm in question works by splitting the
family of input curves in two pieces which are deformed to pass on
either side of the obstacle. (Similar ideas appear in a paper
of Chris Bishop \cite{bis:a1} on $A_1$ deformations of the plane.)

The curve families that we consider are axiomatized in the following
definition.

\begin{definition}\label{def-c2-family-curves}
An \emph{open measured family of $C^2$ curves} is a collection $\Gamma$ of
disjoint, oriented $C^2$ curves in a set $X \subset \R^2$, together with a
probability measure $\sigma$ on $\Gamma$, such that the union of all the curves in
$\Gamma$, denoted $\spt\Gamma$, is an open subset of $X$.
We will denote such a pair by $(\Gamma,\sigma)$, or just by $\Gamma$ if the measure $\sigma$ is understood.
\end{definition}

There is a natural measure $\nu_\Gamma=\nu_{(\Gamma,\sigma)}$ defined on $\spt\Gamma$ by
\begin{equation}\label{eq-def-nu}
\nu_\Gamma(V) = \int_\Gamma \mathcal{H}^1(V \cap \gamma)
d\sigma(\gamma).
\end{equation}
At this point, the integral in \eqref{eq-def-nu} should be interpreted
as an upper integral with value in $[0,\infty]$. However, under the
conditions of Definition~\ref{def-goodcurves}, $\nu_\Gamma$ will
be a finite Borel measure.

We assume that each curve in $\Gamma$ is parameterized with nonzero speed, consistent with the specified orientation.
Since each curve in $\Gamma$ is $C^2$, there is a vector field $\dot{\Gamma}$ defined on $\spt\Gamma$ such that $\dot\Gamma(x)$ coincides with the unit tangent vector to the unique curve $\gamma_x \in \Gamma$ passing through $x$ at time $t_x$. In fact,
\begin{equation}\label{dotGamma}
\dot\Gamma(x) = \frac{\gamma_x'(t_x)}{|\gamma_x'(t_x)|}.
\end{equation}

\begin{definition}\label{def-goodcurves}
Fix $\delta_0 \geq 0$ and $r_0 > 0$.  Suppose $(\Gamma,\sigma)$ is an open measured family
of $C^2$ curves, with $\nu_\Gamma$ defined as in \eqref{eq-def-nu}.
We say that $(\Gamma,\sigma)$ is a \emph{$\delta_0$-good family
of curves on scales less than $r_0$} if
for any ball $B(z,r)$, $z \in X$, $0 < r \leq r_0$ we have the following properties:
\begin{itemize}
\newcounter{custref}
\renewcommand{\thecustref}{\Alph{custref}}

\refstepcounter{custref}\label{eq-good-curves-fin-many}
\item[\eqref{eq-good-curves-fin-many}] 
If $B(z,r)$ does not contain any endpoint of any $\gamma \in \Gamma$, then
the complement of the closure of $\spt\Gamma$ in $B(z,r)$ is a
connected open set.

\refstepcounter{custref}\label{eq-good-curves-gamma-dot}
\item[\eqref{eq-good-curves-gamma-dot}] 
For any $x,y \in B(z,r)$,
\begin{equation*}
\angle(\dot{\Gamma}(x), \dot{\Gamma}(y)) \leq \delta_0 \left(
\frac{|x-y|}{2r_0} \right)^{2/3},
\end{equation*}
where $\angle(\mathbf{v}, \mathbf{w})$ denotes the angle between vectors $\mathbf{v}$ and $\mathbf{w}$.

\refstepcounter{custref}\label{eq-good-curves-weight}
\item[\eqref{eq-good-curves-weight}] 
There is a constant $A_{z,r} \in (0,\infty)$ so that on $B(z,r) \cap \spt \Gamma$
the Radon-Nikodym derivative $w_\Gamma=d\nu_\Gamma/d\mathcal{H}^2$
exists, is $\frac{2}{3}$-H\"older continuous with constant
$(2r_0)^{-2/3}A_{z,r}\delta_0$, and satisfies
\begin{equation*}
w_\Gamma(B(z,r) \cap \spt \Gamma) \subset [(1+\delta_0)^{-1}A_{z,r},(1+\delta_0)A_{z,r}].
\end{equation*}
\end{itemize}
\end{definition}
The first condition ensures that $\overline{\spt\Gamma} \cap B(z,r)$
is either the empty set, one half of the ball, all of the ball except
for one open gap, or all of the ball.
Figure \ref{fig:Gamma-in-B} illustrates typical instances of this.

\begin{figure}[t]
\centering
\includegraphics[width=0.75\textwidth]{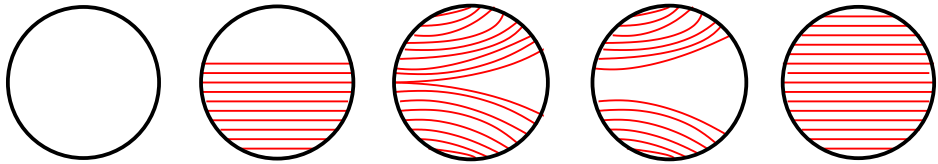}
\caption{Structure of $\overline{\spt\Gamma} \cap
  B(z,r)$}\label{fig:Gamma-in-B}
\end{figure}

The last two conditions guarantee that the vector field $\dot\Gamma$
is $\frac23$-H\"older continuous (with suitable constant) and that the
Radon-Nikodym derivative $w_\Gamma=\tfrac{d\nu_\Gamma}{d\cH^2}$ exists
and is locally close to constant on $\spt\Gamma$.

Why are the vector fields only H\"older continuous?  If they were
Lipschitz continuous, then by the uniqueness of solutions to ODE with
Lipschitz coefficients, the curves could not split to bend round an
obstacle. The choice of $\frac{2}{3}$ is fixed in view of the cubic
spline which we construct in Lemma~\ref{lem-bump-fn}(4). This choice
is merely a convenience; any H\"older exponent strictly less than one
would serve our purposes equally well.\footnote{Note that we do not explicitly
solve the ODE corresponding to the vector field $\dot\Gamma$. In the actual
proof of Proposition \ref{prop-weight-bound},
the H\"older continuity assumption in Definition \ref{def-goodcurves}\eqref{eq-good-curves-weight} arises from Lemma \ref{lem-bump-fn}.}

The following proposition provides the key inductive step in bending curve families
with control on their modulus.  We postpone its proof until subsection \ref{modulus-subsection}.

\begin{proposition}\label{prop-weight-bound}
For any sufficiently small $\delta_0 > 0$, there exist positive constants $a'$ and $C$ with the
following property.

Suppose $(\Gamma_i,\sigma_i)$ is a $\delta_0$-good family of curves on
scales smaller than $s_i$ in $\R^2$, for some $s_i>0$, and we are given $a_{i+1} \in (0,a']$.
Let $\nu_i = \nu_{\Gamma_i}$ be the natural
measure, let $w_i=d\nu_i/d\cH^2$ be the corresponding weight, and set $s_{i+1} = a_{i+1}s_i$.
Then given any $z \in \spt \Gamma_i$, with no curves in $\Gamma_i$ stopping
inside $B(z,s_i)$, we can deform $(\Gamma_i,\sigma_i)$ inside
$B(z,s_i/5)$ into a new open measured curve family
$(\Gamma_{i+1},\sigma_{i+1})$ that is $\delta_0$-good on scales
smaller than $s_{i+1}$, so that $\spt \Gamma_{i+1}$ does not meet
$B(z,2 s_{i+1})$ and
\begin{equation}\label{eq-prop-weight-gain}
\int_{B(z,s_i/5)} |w_{i+1}|^q d\mathcal{H}^2 \le (1+ C a_{i+1}^2)
\int_{B(z,s_i/5)} |w_i|^q d\mathcal{H}^2.
\end{equation}
\end{proposition}

By ``deform'', 
we mean that there exists a $C^2$ homeomorphism of the plane which restricts to the identity outside $B(z,s_i/5)$,
and, up to discarding finitely many curves, induces a well defined, measure preserving bijection between 
$(\Gamma_i, \sigma_i)$ and  $(\Gamma_{i+1}, \sigma_{i+1})$.
We emphasize that the numbers $s_i>0$ and $a_{i+1} \in (0,a']$ in the statement of Proposition \ref{prop-weight-bound}
are arbitrary and are not assumed to be arising from a specific sequence $\mathbf{a}$ under consideration.

Intuitively, Proposition \ref{prop-weight-bound} asserts that we can deform $\Gamma_i$
inside a ball $B$ on a given scale $s_i$ so as to avoid a prespecified
obstacle of size $s_{i+1}$ (in this case, a ball of radius $2s_{i+1}$
concentric with $B$) and so that the $\ell^q$ norm of the associated
weight increases multiplicatively by at most a factor of $1+Ca_{i+1}^2$,
where $C$ is independent of $\ba$. The point is that we can repeatedly apply the proposition (on smaller and smaller scales)
without losing control of $\delta_0$.

\subsection{Using the bending machinery}
We now prove our building block lemmas, Theorem~\ref{modulus-extension-theorem} and Corollary~\ref{thm-l2-mod}.
\begin{proof}[Proof of Lemmas \ref{doubling-q}--\ref{doubling4-q}]
Recall that $(\cC_2,L_2)$ is a directed block following $(\cC_1,L_1)$ in a directed square $(T',L') \in \cT_{m-1}$.
Let $\Gamma_m$ be the open, measured curve family consisting of straight line segments 
in $\R^2$ connecting each point of the interior of $L_1$ to the
corresponding point of $L_2$ under the natural ordered bijection,
equipped with the measure $\sigma_m$ induced from normalized linear
measure on $L_1$.  See Figure~\ref{fig:modular-lemma-redux-figure}.

\begin{figure}
\centering
\psfrag{C1}{$\mathcal{C}_1$}
\psfrag{C2}{$\mathcal{C}_2$}
\psfrag{L1}{$L_1$}
\psfrag{L2}{$L_2$}
\includegraphics[width=0.2\textwidth]{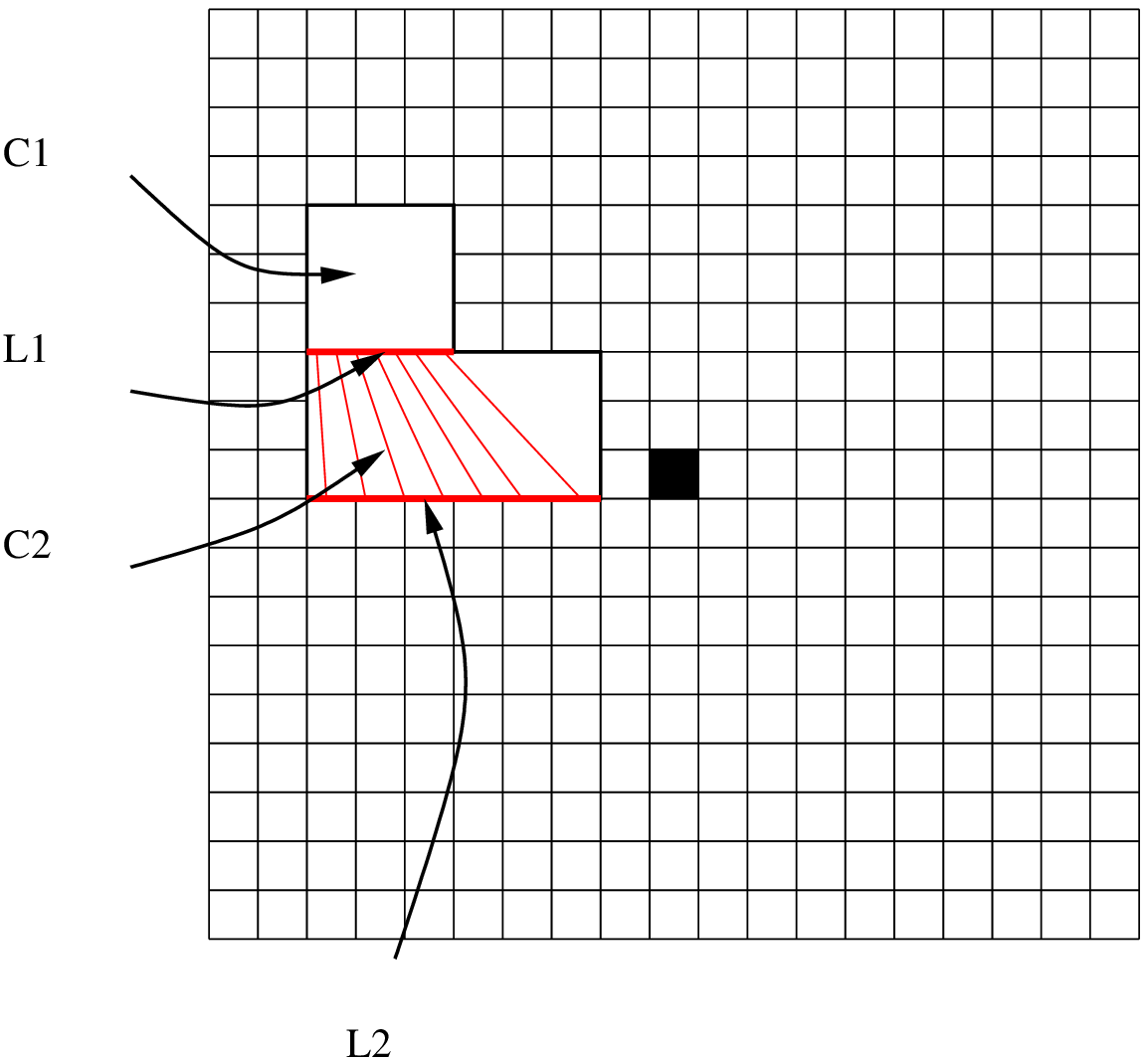}
\hspace{10pt}
\includegraphics[width=0.2\textwidth]{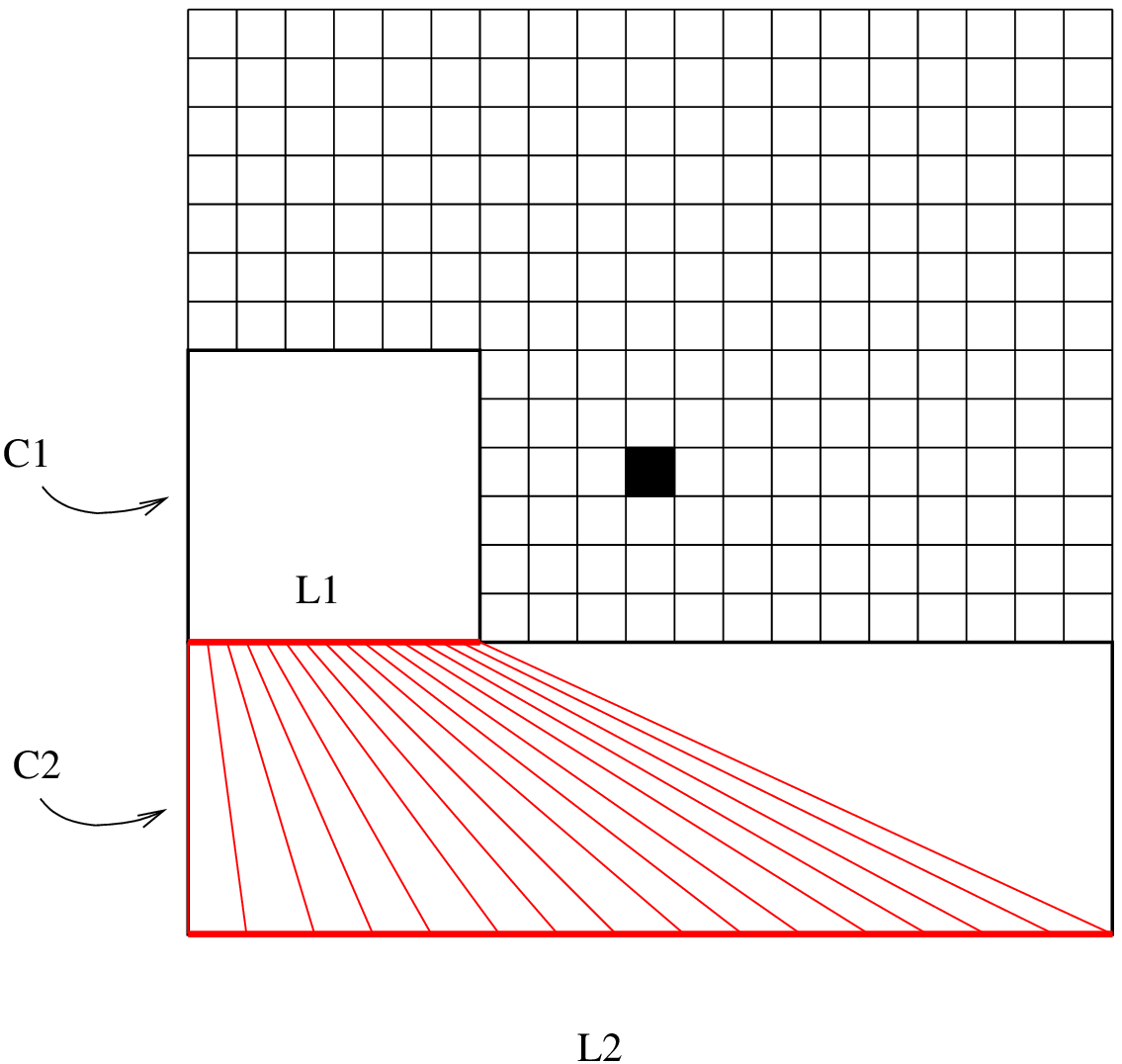}
\hspace{10pt}
\includegraphics[width=0.2\textwidth]{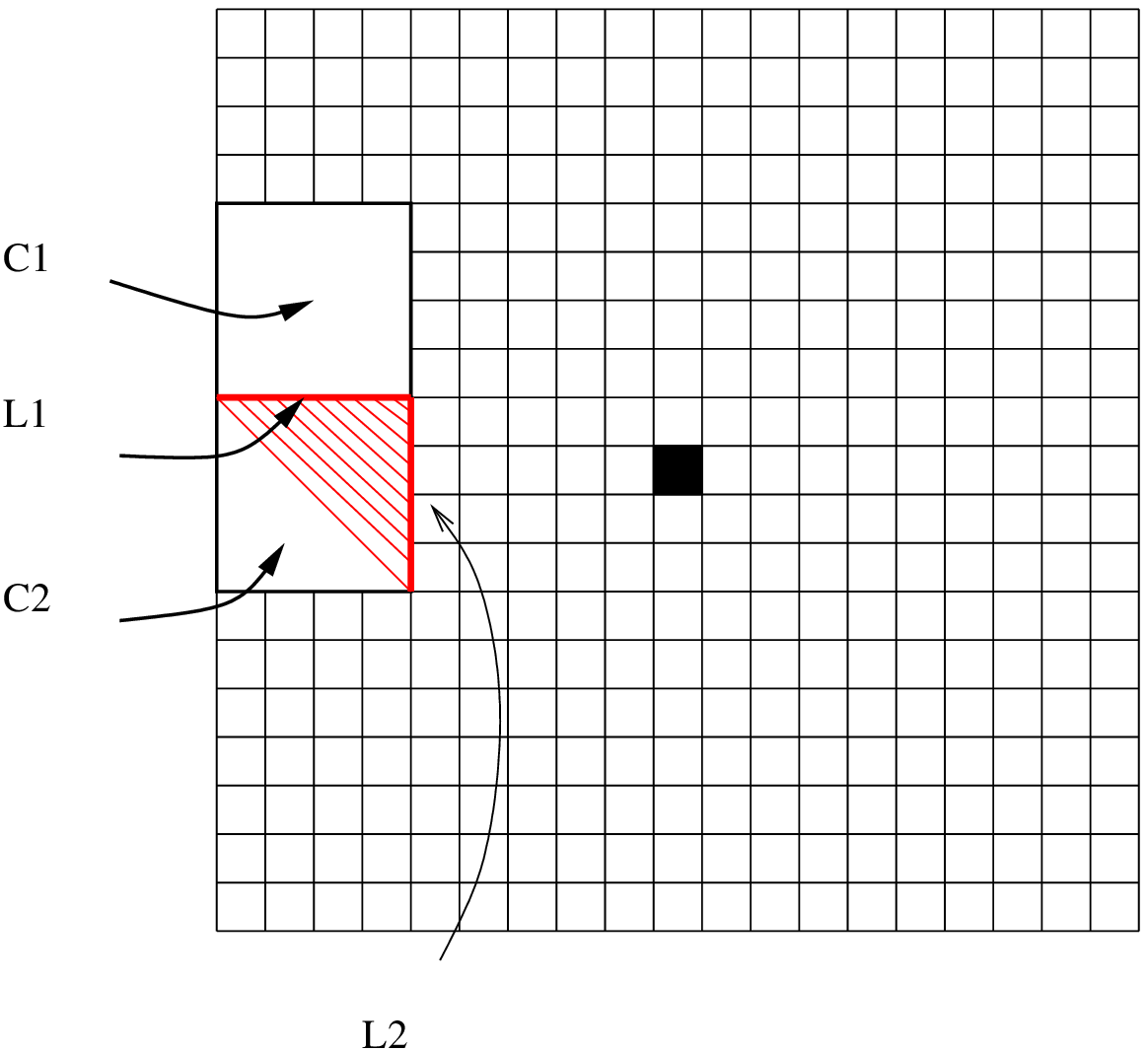}
\hspace{10pt}
\includegraphics[width=0.2\textwidth]{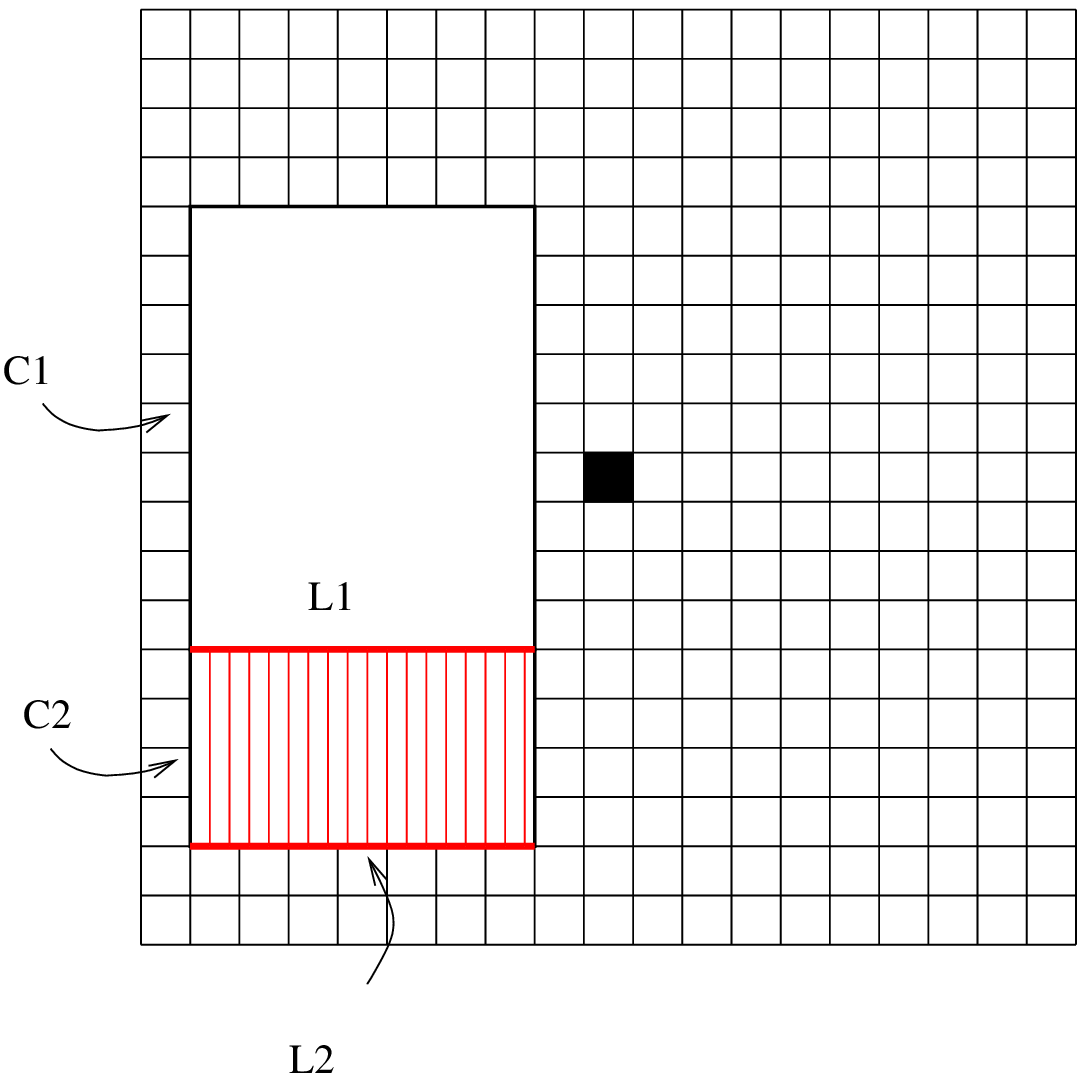}
\caption{Initial curve families for the $p>1$ case of Lemmas \ref{doubling} - \ref{doubling4}}
\label{fig:modular-lemma-redux-figure}
\end{figure}

Let $w_m = d\nu_{\Gamma_m} / d\cH^2$ be the natural weight function associated to $\Gamma_m$.
Observe that on $\spt \Gamma_m$, we have $\nu_{\Gamma_m} \asymp \cH^1(\pi_M(L_1))^{-1} \cH^2$,
and that $\cH^2(\spt \Gamma_m) \asymp (\cH^1(\pi_M(L_1)))^2$, so
\begin{equation}\label{eq-initial-weight-bound}
	\| w_m \|_q \asymp (\cH^1(\pi_M(L_1)))^{-1+2/q}.
\end{equation}

Since a $q$-connection must lie in $\mathcal{C}_2
\cap S_{\ba, M}$, we `bend' this initial family around the subsquares
of $\mathcal{C}_2$ that were removed in the construction of $S_{\ba,M}$.
This construction is inductive, building open measured curve families
$(\Gamma_i, \sigma_i)$ for $m \leq i \leq M$.

Observe that there exists a universal constant $\delta_1>1$ so that any such initial curve family 
$\Gamma_m$ is a $\delta_1$-good curve family on scales below $s_m$.
We use the following sublemma.
\begin{sublemma}
	For any $\eps \in (0,1)$, a $\delta_1$-good curve family on scales below $r_1$
	is also a $(\delta_1 \eps^{1/3})$-good curve family on scales below $\eps r_1$.
\end{sublemma}
\begin{proof}
	The only non-trivial estimate is the last part of Definition \ref{def-goodcurves}\eqref{eq-good-curves-weight}.
	On a ball of radius $\eps r_1$, with associated constant $A$, the ratio of maximum to minimum values of $w$ is at most
	\[
		\frac{(1+\delta_1)^{-1}A + (2r_1)^{-2/3}A \delta_1 (2 \eps r_1)^{2/3}}{(1+\delta_1)^{-1}A}
			\leq 1+ (1+\delta_1)\delta_1 \eps^{2/3} \leq (1+\delta_1\eps^{1/3})^2.\qedhere
	\]
\end{proof}
Fix the constant $a_{**}$ in \eqref{aia2} so that \[ a_{**}  = a' \del_1^{-3} \del_0^3 \leq a', \]
where $\delta_0$ and $a'$ are chosen by Proposition~\ref{prop-weight-bound}.
The sublemma above shows that $\Gamma_m$ is a $\delta_0$-good curve family on scales below $s_{m}' := \del_1^{-3} \del_0^3 s_{m}$.

The removed squares of side $s_{m+1}$ in $\mathcal{C}_2$ are all at least $s_m/2$ apart,
so we apply Proposition~\ref{prop-weight-bound} to $(\Gamma_m, \sigma_m)$ independently for some $z$ in each such removed square,
with values $s_i = s_{m}'$ and $s_{i+1}=s_{m+1}$; this is valid since $s_{m+1} / s_{m}' < a'$.
Denote the resulting open, measured curve family by $(\Gamma_{m+1}, \sigma_{m+1})$,
which is $\delta_0$-good on scales below $s_{m+1}$.

By applying \eqref{eq-prop-weight-gain} around each removed square,
we have that the natural weight function $w_{m+1}$ associated to $\Gamma_{m+1}$
satisfies 
\begin{equation}\label{eq-weight-blocks}
\| w_{m+1} \|_q \leq (1+ C a_{m+1}^2)^{1/q} \| w_m \|_q.
\end{equation}
Similarly, for $i = m+1, \ldots, M-1$, we build $(\Gamma_{i+1}, \sigma_{i+1})$ from $(\Gamma_i, \sigma_i)$ by
applying Proposition~\ref{prop-weight-bound} indepenently for each removed square of side $s_{i+1}$;
this is valid since $s_{i+1}/s_i = a_{i+1} \leq a_{**} < a'$.
We obtain a open, measured
curve family $(\Gamma, \sigma) = (\Gamma_M, \sigma_M)$, with $\spt \Gamma$ in $\mathcal{C}_2 \cap S_{\ba, M}$.
Iterating the weight bound of \eqref{eq-weight-blocks}, we see that
\begin{equation*}\label{eq-absoluteweightbound-blocks}
\left\| \frac{d\nu_{\Gamma}}{d\cH^2} \right\|_{L^q(\mathcal{C}_2,\cH^2)} 
 = \| w_M \|_q \leq \left( \prod_{i=m+1}^M (1+Ca_i^2) \right)^{1/q} \|
w_{m} \|_q 
\le \exp \left( \frac{C}{q} ||\ba||_2^2 \right) \cH^1(\pi_M(L_1))^{-1+2/q},
\end{equation*}
where this last inequality uses \eqref{eq-initial-weight-bound}.
This implies that the output family is a
$q$-connection.
\end{proof}

A similar argument extends to give modulus bounds.

\begin{proof}[Proof of Theorem \ref{modulus-extension-theorem}]

We fix $a' < \frac{1}{100}$, $\del_0$, and $C$ as in Proposition \ref{prop-weight-bound} and choose $M$
so that $a_i < a'$ when $i\ge M$.
Note that if $\ba$ is monotone decreasing, then $M$ depends only on
$\| \ba \|_2$.

If $S$ has nonempty interior we are
done, since open sets in $\R^2$ certainly contain curve families with
positive $p$-modulus.
Otherwise there exists $\cU_{M'} \neq \emptyset$, for some $M' \geq M$.
Choose $U \in \cU_{M'}$ so that $\dist(U, \partial D) \geq \frac{2}{5} s_{M'-1} \geq 10s_{M'}$.
We choose a square $W \subset D$ of side $s_{M'}$ so that $\dist(W, U)$ is
between $2s_{M'}$ and $3 s_{M'}$.
By the assumptions of the theorem, and choice of $M$, $W$ only meets sets from
$\cup_{m \geq (M' +1)} \cU_m$, which have diameter much smaller than $W$.

We choose coordinates so that $W = [0, s_{M'}]^2$.
Let $\Gamma_{M'}$ be the family of curves $\{ \gamma_u : u \in (0, s_{M'}) \}$, where $\gamma_u:(0,s_{M'}) \ra W$ is
defined as $\gamma_u(t)=(t, u)$.
We equip $\Gamma_{M'}$ with the probability measure $\sigma_{M'}$ given by scaling Lebesgue measure on $(0, s_{M'})$ by $s_{M'}^{-1}$.
Observe that $\Gamma_{M'}$ is a $0$-good family of curves on scales below $s_{M'}$,
with $w_{M'} \equiv s_{M'}^{-1}$.

We now build measured curve families $(\Gamma_m,\sigma_m)$ for $m \geq M'$ that are
$\delta_0$-good on scales below $s_m$, with the additional property that
\begin{equation}\label{eq-avoids-required-sets}
	\spt \Gamma_m \cap \big([s_{M'}/4, 3s_{M'}/4] \times \R\big) \subset S_m.
\end{equation}
Let $\nu_m = \nu_{\Gamma_m}$ be the natural measure on
$\spt\Gamma_m$ and let $w_m=d\nu_m/d\cH^2$ be the corresponding
weight.

The construction is inductive. Assume that we have constructed a
measured curve family $(\Gamma_m,\sigma_m)$ that is
$\delta_0$-good on scales below $s_m$.
We apply Proposition~\ref{prop-weight-bound}
to bend $\Gamma_m$ around each set $V \in \cU_{m+1}$ which meets
$[s_{M'}/8, 7s_{M'}/8] \times [-s_{M'}, 2s_{M'}]$.
As the sets in $\cU_{m+1}$ are all at least $s_m$ apart, we can apply
Proposition~\ref{prop-weight-bound} at each location independently
to create a new measured curve family $(\Gamma_{m+1}, \sigma_{m+1})$,
which is $\delta_0$-good on scales below $s_{m+1}$.
(Proposition \ref{prop-weight-bound} requires that no curve ends near where we bend;
this is why we restrict the obstacles that we bend around.)
These curve families satisfy \eqref{eq-avoids-required-sets} for $m \geq M'$.

As in the proof of Lemmas \ref{doubling-q}--\ref{doubling4-q}, we conclude that for all
$m \geq M'$ the natural weight function $w_m$ associated to $\Gamma_m$ satisfies
\begin{equation}\label{eq-absoluteweightbound}
\| w_m \|_q \leq \left( \prod_{i=M'+1}^\infty (1+Ca_i^2) \right)^{1/q} \|
w_{M'} \|_q \le \exp \left( \frac{C}{q} ||\ba||_2^2 \right) \|w_{M'}\|_q
\end{equation}

If $\rho: B \rightarrow [0,\infty]$ is admissible for
$\mod_p\,\Gamma_m$, that is, $1 \le \int_\gamma \rho \,
d\mathcal{H}^1$ for each $\gamma \in \Gamma_m$, then by averaging over
$\Gamma_m$ with respect to $\sigma_m$, we see that
\begin{equation}\begin{split}\label{here-is-q}
1 \le \int_{\Gamma_m} \int_\gamma \rho \ d\mathcal{H}^1 d\sigma_m(\gamma)
= \int_{\spt\Gamma_m} \rho \ d\nu_m = \int_{\spt\Gamma_m} \rho \, w_m \
d\mathcal{H}^2 \le \| \rho \|_p \ \| w_m \|_q,
\end{split}\end{equation}
where $\frac{1}{p}+\frac{1}{q} = 1$.
Combining \eqref{eq-absoluteweightbound}, \eqref{here-is-q} and $\ba \in \ell^2$,
we see that $\mod_p\,\Gamma_m$ is uniformly bounded from below independent of $m$.

Each curve $\gam \in \Gamma_m$ has a subcurve $\gam'$ which joins
the set $\{s_{M'}/4\} \times \R$ to $\{3s_{M'}/4\} \times \R$, and $\gam' \subset S_m$.
Let $\Gamma_m'$ be the collection of all such curves.
By basic properties of the modulus, $\mod_p \, \Gamma_m' \geq \mod_p \, \Gamma_m$.

Thus for every $m \geq M'$, there is a curve family $\Gamma_m'$ in $S_{m}$
with $p$-modulus bounded from below independent of $m$.
By the upper semicontinuity of modulus (see, for instance, Heinonen--Koskela \cite[\S3]{hk:quasi} or Keith~\cite[Theorem 1]{kei:modulus}),
this bound will continue to the limit.
\end{proof}

\begin{proof}[Proof of Corollary \ref{thm-l2-mod}]
Again, choose $a'$ as in Proposition \ref{prop-weight-bound} and
choose $M$ so that $a_i<a'$ when $i\ge M$. 
The $p$-modulus of the family of curves joining the left and right edges of
the carpet $S_\ba$ is bounded from below by the $p$-modulus of the family
of curves $\Gamma$ joining the left and right edges of the strip 
$([0,1] \times [0, s_M]) \cap S_\ba$.

Let $\Gamma_M$ be the family of horizontal lines in $T = [0,1]\times [0,s_M]$,
with induced natural weight $w_M$, equal to $s_M^{-1}$ on $T$, and so having
$\|w_M \|_q = s_M^{-1/p}$.
The argument in the proof of Theorem~\ref{modulus-extension-theorem}, in particular
\eqref{eq-absoluteweightbound} and \eqref{here-is-q}, give that any admissible function $\rho$ for $\Gamma$
satisfies
\[
	1 \leq \| \rho \|_p \exp\left( \frac{C}{q} \| \ba \|_2^2 \right) \| w_M \|_q.
\]
Therefore, as $\|w_M \|_q \leq C(\ba) < \infty$, we have $\mod_p(\Gamma)  > 0$.
Finally, we note that in the case when $\ba$ is monotone decreasing,
then $M$ can be chosen only depending on $||\ba||_2$ (and not on the
actual sequence $\ba$).  This implies that $\|w_M \|_q \leq C(\|\ba\|_2)$, and this
establishes the final claim of the corollary.
\end{proof}

It remains to establish Proposition \ref{prop-weight-bound}. This is
the goal of the following subsection. The argument is rather technical
although essentially elementary. The reader is invited to skip the remainder of this section on a first reading of the paper.

\subsection{Compressing curve families: the proof of Proposition
  \ref{prop-weight-bound}}\label{modulus-subsection}

The following construction is standard. For the convenience of the reader we provide a short proof.

\begin{lemma}\label{lem-bump-fn}
There is a $C^2$ function $\varphi:[-1,1] \ra[0,1]$ which satisfies
\begin{enumerate}
\item $\varphi|_{[-0.1,0.1]} \equiv 1$,
\item the support of $\varphi$ lies in $[-0.9,0.9]$,
\item $|\varphi'| \le 5/2$, $|\varphi''| \leq 25/2$, and
\item $|\varphi'| \le 14 \varphi^{2/3}$.
\end{enumerate}
\end{lemma}

\begin{proof}
Choose $\varphi'':[-1,1]\ra\R$ to be the simplest piecewise linear
function whose graph passes through the points $(\pm 1,0)$, $( \pm
0.9,0)$, $(\pm 0.7,b)$, $(\pm0.3,-b)$, $(\pm0.1,0)$, and $(0,0)$,
where $b > 0$ is a constant to be determined.
	
Assuming that $\varphi'(-1)=\varphi(-1)=0$, we integrate to find
$\varphi$. Note that $\varphi|_{[-0.1,0.1]} \equiv 0.08b$, so we
choose $b = 1/0.08 = 25/2$. With this choice, $|\varphi'| \leq 0.2b =
5/2$ and $|\varphi''| \le b = 25/2$. Hence conditions (1), (2) and (3)
are satisfied.
	
Finally, note that for $0 \le h \le 0.2$, $\varphi''(-0.9+h) = 5bh =
\frac{125}2h$, $\varphi'(-0.9+h) = \frac{5}{2}bh^2 = \frac{125}4h^2$
and $\varphi(-0.9+h) = \frac{5}{6}bh^3 = \frac{125}{12}h^3$, so
\[
|\varphi'(-0.9+h)| = 5 \left( \frac32 \right)^{2/3}
|\varphi(-0.9+h)|^{2/3} \le 10 |\varphi(-0.9+h)|^{2/3}.
\]
This bound also applies for $x \in [0.7,0.9]$. On the other hand, for
$x \in [-0.7,0.7]$ we have $\varphi(x) \geq \varphi(-0.7) =
\frac{1}{12}$ and $|\varphi'(x)| \le \frac52$, so
$|\varphi'(x)| \le \frac52 = \frac52 \cdot 12^{2/3} \left( \frac1{12} \right)^{2/3} \le
14 |\varphi(x)|^{2/3}$.
\end{proof}

We now begin the proof of Proposition \ref{prop-weight-bound}.

\begin{proof}[Proof of Proposition \ref{prop-weight-bound}]
We fix a positive constant $\delta_0\le\tfrac1{200}$. We will choose a
large positive integer $N=N(\delta_0)\geq 5$; the precise choice will be
made later in the proof. Finally, we assume that $a' \le 10^{-2-N}$; we
only consider $a_i < a'$.

Let $(\Gamma_i,\sigma_i)$ be a $\delta_0$-good family of curves on
scales smaller than $s_i$, let $\nu_i$ be the corresponding measure as defined in \eqref{eq-def-nu}, and suppose that
$z \in \spt \Gamma_i$.
We can apply an isometry of $\R^2$ to reduce to the case when $z=0$,
\begin{align*}
B(z,2s_{i+1}) & \subset P := [-2s_{i+1},2s_{i+1}]^2 \\
& \subset Q := [-10^Ns_{i+1},10^Ns_{i+1}]^2 \\
& \subset R := [-s_i/10,s_i/10]^2 \subset B(z,s_i/5),
\end{align*}
and $\dot{\Gamma_i}(0)$ has horizontal slope.
This last assertion, in conjunction with
Definition \ref{def-goodcurves}\eqref{eq-good-curves-gamma-dot},
implies that all curves in $R$ have slopes
within $(1/5)\delta_0$ of zero. In particular,
each curve $\gamma \in \Gamma_i$ is a graph over the $x$-axis inside
$R$. Henceforth we will assume that each curve is given in graph form:
$y=\gamma(x)$.  Nevertheless, we continue to denote by $\gamma =
\{(x,\gamma(x))\}$ the graph itself.

We choose a curve $\gamma_0$ which passes near $P$ and either bounds an existing gap in $Q$, or is far from an existing
gap in $Q$.
To be precise, let $U$ be the complement of $\overline{\spt \Gamma_i}$ in $B(z,s_i/5)$ which, by condition \eqref{eq-good-curves-fin-many}, is a connected open set.
Moreover, $\bdry U \cap B(z,s_i/5)$ lies in one or two $C^2$ curves whose slopes satisfy, along with $\Gamma_i$,
condition \eqref{eq-good-curves-gamma-dot}.
If $U$ meets $L = \{0\} \times [-3s_{i+1},3s_{i+1}]$,
choose $\gamma_0$ which bounds an edge of $U$ meeting $L$ (as $(0,0) \notin U$, such a $\gamma_0$ exists).
If $U$ does not meet $L$, choose $\gamma_0 \in \Gamma_i$ which passes
through $(0,-3s_{i+1})$ or $(0,3s_{i+1})$, chosen so that
\begin{equation}\label{eq-gamma0-condition}
	\dist(\gamma_0(0), U) \geq 5s_{i+1}.
\end{equation}
(Recall that $\gamma_0$ is a graph over the $x$-axis and we have normalized so that $z=0$.)

We will compress the curves inside $Q$ into the complement of $P$, leaving everything unchanged in $R \setminus Q$.
To build $\Gamma_{i+1}$ we will delete $\gamma_0$ from $\Gamma_i$ if necessary,
and apply a diffeomorphism on $X \setminus
\gamma_0$ to compress the remaining curves around~$P$.
The two options in the choice of $\gamma_0$ correspond to either
enlarging an existing gap in $\Gamma_i$, or creating a new gap at least
$4s_{i+1}$ from any previous gap.

We rescale the function $\varphi$ from Lemma~\ref{lem-bump-fn} to the
scale of $Q$ by defining
$$
\tilde{\varphi}(x) = 6 s_{i+1} \varphi\left(\frac{x}{10^Ns_{i+1}}\right).
$$
Note that $|\tilde{\varphi}'| \leq 10^{2-N}$, $|\tilde{\varphi}''|\leq
10^{2-2N} s_{i+1}^{-1}$, and
\begin{equation}\label{eq-bumpderiv-control}
|\tilde{\varphi}'| \leq 10^{2-N} \left(
  \frac{\tilde{\varphi}}{2s_{i+1}} \right)^{2/3}.
\end{equation}
We now define the {\it local compression map} $H:Q\setminus\{\gamma_0\} \ra Q$. Let $g:Q\setminus \{\gamma_0\} \ra [-1,1]$ be given by
\[
g(x,y) = \begin{cases}
\varphi\left(\frac{y-\gamma_0(x)}{10^Ns_{i+1}-\gamma_0(x)}\right)
	& \text{if } y \in (\gamma_0(x),10^Ns_{i+1}) \\
-\varphi\left(\frac{\gamma_0(x)-y}{10^Ns_{i+1}+\gamma_0(x)}\right)
	& \text{if } y \in (-10^Ns_{i+1},\gamma_0(x)).
\end{cases}
\]
Since the functions $\gamma_0$ and $\varphi$ are $C^2$ and
$10^Ns_{i+1}-\gamma_0$ and $10^Ns_{i+1}+\gamma_0$ take values in $[0.99\cdot
10^Ns_{i+1},1.01 \cdot 10^N s_{i+1}]$, $g$ is $C^2$. The function $g$
varies from $0$ near the top of $Q$ to $1$ just above $\gamma_0$, and from $-1$
just below $\gamma_0$ to $0$ near the bottom of $Q$. Next let
\[
h(x,y) = y + \tilde\varphi(x) \cdot g(x,y),
\]
and define
\[
H(x,y) = \big( x, h(x,y) \big).
\]
Both $h$ and $H$ are $C^2$, moreover, $H$ is a diffeomorphism onto its image.
	
Extend $g$ to be zero outside $Q$ and extend $H$ to be the identity
outside $Q$. The new collection of curves is defined by pushing
forward by the local compression map $H$:
\begin{equation}\label{Gamma-recursion}
\Gamma_{i+1} = \{H(\gamma) : \gamma\in \Gamma_i, \gamma \neq \gamma_0\}.
\end{equation}
The probability measure $\sigma_i$ on $\Gamma_i$ pushes forward in the obvious way
to a probability measure on $\Gamma_{i+1}$ that we denote by $\sigma_{i+1}$. We
define $\dot{\Gamma}_{i+1}$ and $\nu_{i+1}$ as in \eqref{dotGamma} and \eqref{eq-def-nu}.
Since $H(Q\setminus\{\gamma_0\})$ and $P$ are disjoint, so are $\spt\Gamma_{i+1}$ and $P$.
	
\begin{proposition}\label{good-prop}
$(\Gamma_{i+1},\sigma_{i+1})$ is a $\delta_0$-good family of curves on
scales below $s_{i+1}$, and $\| w_{i+1} \|_q \leq 2 \| w_i \|_q$ on $Q$.
\end{proposition}

Assuming for the moment the validity of Proposition \ref{good-prop} we
quickly complete the proof of Proposition \ref{prop-weight-bound}.
By
condition \eqref{eq-good-curves-weight}, we have, for $A = A_{z, s_i}$,
\begin{equation}\label{eq-weight-bound-Q}
\int_Q w_i^q \,d\mathcal{H}^2 \leq (1+\delta_0)^qA^q \cH^2(Q)
	\leq (1+\delta_0)^q A^q \cdot 10^{2N}s_{i+1}^2.
\end{equation}
As $z \in \spt \Gamma_i$, by condition \eqref{eq-good-curves-fin-many}, we know that
$\cH^2(\spt\Gamma_i\cap R) \geq \frac13 \cH^2(R) = 75^{-1} s_i^2$.  So we bound
\begin{equation}\label{eq-weight-bound-R}
\int_R w_i^q \,d\mathcal{H}^2 \geq (1+\delta_0)^{-q}A^q \cH^2(\spt\Gamma_i\cap R) \geq 75^{-1} (1+\delta_0)^{-q}A^q s_i^2.
\end{equation}
Note $w_{i+1}=w_i$ on $R\setminus Q$, and Proposition~\ref{good-prop} controls $w_{i+1}$ on $Q$.
Therefore, by \eqref{eq-weight-bound-Q}
and \eqref{eq-weight-bound-R} we have
\begin{align*}\label{eq-weight-rq}
	\int_R w_{i+1}^q \,d\mathcal{H}^2
	& \leq \int_{R\setminus Q} w_{i}^q \,d\mathcal{H}^2 + 2^q \int_Q w_{i}^q \,d\mathcal{H}^2 \\
	& \leq \int_{R} w_{i}^q \,d\mathcal{H}^2 + C a_{i+1}^2 \int_R w_{i}^q \,d\mathcal{H}^2 \\
	& = (1+C a_{i+1}^2) \int_R w_{i}^q \,d\mathcal{H}^2,
\end{align*}
where $C = 75 \cdot 2^q (1+\del_0)^{2q} 10^{2N}$.
This completes the proof of Proposition \ref{prop-weight-bound}.
\end{proof}

The proof of Proposition \ref{good-prop} is divided into three lemmas. In all three of these lemmas, the context is the modified measured curve family $(\Gamma_{i+1},\sigma_{i+1})$ defined in \eqref{Gamma-recursion}.

\begin{lemma}
Condition \eqref{eq-good-curves-fin-many} of Definition \ref{def-goodcurves} is satisfied.
\end{lemma}

\begin{proof}
Recall that the curve $\gamma_0$ was chosen in one of two ways.
First, suppose that $\gamma_0$ was chosen to contain part of the boundary of $U = B(z,s_i/5) \setminus \overline{\spt \Gamma_i}$.  It is clear that the deformation $H$ has only enlarged this set, and it is easy to see that $\Gamma_{i+1}$ will satisfy condition \eqref{eq-good-curves-fin-many}.

Now in the remainder of this proof, we suppose that $\gamma_0 \in \Gamma_i$ was chosen so that
$\dist(\gamma_0(0),U) \geq 5s_{i+1}$. We must check that the new
open set opened up along $\gamma_0$ will not result in two open gaps in $\spt\Gamma_{i+1}$ in a common $s_{i+1}$-ball.

Denote the curve which bounds the edge of $U$ closest to $\gamma_0$
by $\gamma_1$.  Without loss of generality, we may assume that
$\gamma_0(x) \leq \gamma_1(x)$ for all $x \in I:=[-10^Ns_{i+1},10^Ns_{i+1}]$.
To complete the proof of this lemma, it suffices to show that
$\gamma_1(x)-\gamma_0(x) \geq 4s_{i+1}$ for all $x \in I$,
since then in the image they will remain sufficiently far apart.

Let $\sigma_0$ be the $\sigma_i$ measure of those curves of $\Gamma_i$
which lie between $\gamma_0$ and $\gamma_1$.
For $x_1 \leq x_2$, let
\[
	T[x_1,x_2] = \{ (x,y) \in Q \cap \spt\Gamma_i : x_1 \leq x \leq x_2,\
		\gamma_0(x) < y < \gamma_1(x) \}.
\]
By \eqref{eq-def-nu}, for any $x \in I$, $h>0$ we have
$\sigma_0 h \leq \nu_i(T[x,x+h])$.
On the other hand, by condition \eqref{eq-good-curves-weight} we have
$\nu_i(T[x,x+h]) \leq (\gamma_1(x)-\gamma_0(x)+h/100)h(1+\delta_0)A$,
for the appropriate value of the constant $A=A_{z,r}$.
Combining these and letting $h \to 0$, we see that
\begin{equation}\label{eq-Acheck-1}
	\sigma_0 \leq (1+\delta_0)A(\gamma_1(x)-\gamma_0(x)).
\end{equation}
Likewise, by considering $T[0,h]$, we see that
\[
	(\gamma_1(0)-\gamma_0(0)-h/100)h(1+\delta_0)^{-1}A \leq \nu_i(T[0,h]) \leq 1.01\sigma_0h,
\]
and therefore
\begin{equation}\label{eq-Acheck-2}
	(\gamma_1(0)-\gamma_0(0))(1+\delta_0)^{-1}A \leq 1.01\sigma_0.
\end{equation}
Combining \eqref{eq-gamma0-condition}, \eqref{eq-Acheck-1} and \eqref{eq-Acheck-2}, we see that
\[
	5s_{i+1} \leq \gamma_1(0)-\gamma_0(0)
		\leq (1+\delta_0)A^{-1}1.01\sigma_0
		\leq (1+\delta_0)^2 1.01 (\gamma_1(x)-\gamma_0(x)).
\]
Therefore $\gamma_0$ and $\gamma_1$ are always at least $4s_{i+1}$ apart in $Q$.
\end{proof}

\begin{lemma}\label{Blemma}
Condition \eqref{eq-good-curves-gamma-dot} of Definition \ref{def-goodcurves} is satisfied.
\end{lemma}

\begin{proof}
Let us consider $u_1,u_2 \in R\cap (\spt\Gamma_{i+1})$ so that
$|u_1-u_2|\le 2s_{i+1}$. Note that $u_k = H(v_k)$ for some $v_k \in
\spt\Gamma_i \cap R$, $k\in\{1,2\}$.
If $v_1$ and $v_2$ lie on different sides of $\gamma_0$ then one can calculate that $|v_1-v_2| \leq 1.02 |u_1-u_2|$,
as $v_k$ and $u_k$ lie on the same vertical line and the slope of $\gamma_0$ is close to zero.
When $v_1$ and $v_2$ are on the same side of $\gamma_0$, the same estimate follows from
\eqref{eq-distancecontrol} below.
Write
$$
v_k = (x_k,\gamma_k(x_k)) = (\id \otimes \gamma_k)(x_k)
$$
for some $\gamma_k \in \Gamma_i$. We calculate the differential of the function $h\circ(\id\otimes\gamma_k)$ as follows:
\begin{equation}\label{c123}
(h \circ (\id \otimes \gamma_k))'(x_k) = C_{1k} + C_{2k} + C_{3k},
\end{equation}
where $C_{1k} = \gamma'_k(x_k)$, $C_{2k} = \tilde\varphi'(x_k)
g(v_k)$, and
$$
C_{3k} = \tilde\varphi(x_k) (g\circ(\id\otimes\gamma_k))'(x_k).
$$
Eventually, we want to estimate the difference between
$(h\circ(\id\otimes\gamma_1))'(x_1)$ and
$(h\circ(\id\otimes\gamma_2))'(x_2)$. In view of \eqref{c123}, we
write $\Delta C_1$, $\Delta C_2$, and $\Delta C_3$ for the differences
of the summands. We will estimate everything in terms of the scale-invariant quantity
\begin{equation}\label{ALPHA}
\alpha(u_1,u_2) := \left( \frac{|u_1-u_2|}{2s_{i+1}} \right)^{2/3}.
\end{equation}

To estimate $\Delta C_1$ we use the following elementary fact: for any
$v,w \in [-1/4,1/4]$, the vectors $\mathbf{v}=(1,v)$ and
$\mathbf{w}=(1,w)$ satisfy
\begin{equation}\label{eq-geom-lemma}
\tfrac12|v-w| \leq \Bigl| |\mathbf{v}|^{-1}\mathbf{v} -
  |\mathbf{w}|^{-1}\mathbf{w} \Bigr| \leq 2 |v-w|,
\end{equation}
as one quickly sees from the identity
\[
	\frac{\big| |\mathbf{v}|^{-1}\mathbf{v} - |\mathbf{w}|^{-1}\mathbf{w} \big|^2}{|v-w|^2}
	=\frac{2}{(1+v^2)(1+w^2) + (1+vw)\sqrt{1+v^2}\sqrt{1+w^2}}\ .
\]

Recall that $\dot\Gamma_i(v_k) =
\frac{(1,\gamma_k'(x_k))}{\sqrt{1+\gamma_k'(x_k)^2}}$ for
$k\in\{1,2\}$. Since $\dot\Gamma_i(v_k)$ is a unit vector,
\begin{equation}\label{arclength}
|\dot\Gamma_i(v_1)-\dot\Gamma_i(v_2)| = 2 \sin \left( \tfrac12
  \angle(\dot\Gamma_i(v_1),\dot\Gamma_i(v_2)) \right) \le
\angle(\dot\Gamma_i(v_1),\dot\Gamma_i(v_2)).
\end{equation}
An application of \eqref{eq-geom-lemma} gives
\begin{equation}\label{eq-deform-c1}\begin{split}
|\Delta C_1| &= |\gamma_1'(x_1) - \gamma_2'(x_2)|
\le 2 |\dot\Gamma_i(v_1)-\dot\Gamma_i(v_2)|
\le  2\angle(\dot{\Gamma}_i(v_1), \dot{\Gamma}_i(v_2)) \\
& \le 2 \delta_0 \left( \frac{|v_1-v_2|}{2s_i} \right)^{2/3} \le 3 \delta_0 a_{i+1}^{2/3} \alpha(u_1,u_2),
\end{split}\end{equation}
where $\alpha(u_1,u_2)$ is defined as in \eqref{ALPHA}.

If one of the points $v_1$ and $v_2$ lies outside $Q$, then both are
close to the edge of $Q$, where $H$ is the identity, thus
\begin{equation}\label{eq-deform-c23}
|\Delta C_2|=|\Delta C_3 |=0.
\end{equation}
We therefore assume that both $v_1$ and $v_2$ are in~$Q$.

Suppose $v_1$ and $v_2$ lie on opposite sides of $\gamma_0$.
Since $|v_1-v_2| \le 3 s_{i+1}$, we have that
$| \tilde\varphi(x_1) | +| \tilde\varphi(x_2) |  \leq 2 |u_1-u_2|$.
Therefore, using \eqref{eq-bumpderiv-control} we see that
\begin{equation}\label{eq-deform-c2-5}\begin{split}
|\Delta C_2|
& = | \tilde\varphi'(x_1) g(v_1) - \tilde\varphi'(x_2) g(v_2) |
\le | \tilde\varphi'(x_1) | + | \tilde\varphi'(x_2) | \\
& \leq 10^{2-N} \left( \frac{\tilde{\varphi}(x_1)}{2s_{i+1}} \right)^{2/3} +
10^{2-N} \left( \frac{\tilde{\varphi}(x_2)}{2s_{i+1}} \right)^{2/3} \\
&\le 4 \cdot 10^{2-N} \alpha(u_1,u_2).
\end{split}\end{equation}
Since $v_1$ and $v_2$ are both close to $\gamma_0$, they are both in
the region where $|g|=1$, whence
\begin{equation}\label{eq-deform-c3}
\Delta C_3 = 0.
\end{equation}

It remains to bound $\Delta C_2$ and $\Delta C_3$ when $v_1$ and $v_2$
lie on the same side of $\gamma_0$. Without loss of generality, we may
assume that both $v_1$ and $v_2$ are above $\gamma_0$. Then $g(v_k) =
\varphi(A_k/B_k)$ where $A_k = \gamma_k(x_k)-\gamma_0(x_k)$ and $B_k =
10^Ns_{i+1}-\gamma_0(x_k)$. Note that
\begin{gather*}
|A_k| \leq 1.01 \cdot 10^N s_{i+1}, \quad |B_k| \in [0.99\cdot
10^Ns_{i+1},1.01 \cdot 10^N s_{i+1}], \\
|A_1-A_2| \leq 2|v_1-v_2|, \quad |B_1-B_2| \leq \frac{1}{100} |v_1-v_2|.
\end{gather*}
To estimate $|g(v_1)-g(v_2)|$ we use the simple estimate
\begin{equation}\begin{split}\label{eq-deform-c2-2b}
\left| \frac{A_1}{B_1} - \frac{A_2}{B_2} \right |
\le \frac{ |A_1-A_2| \cdot |B_2| + |A_2| \cdot |B_1-B_2|}{|B_1 B_2|}
\le 3 \cdot 10^{-N} \left( \frac{|v_1-v_2|}{s_{i+1}} \right). 
\end{split}\end{equation}
Thus
\begin{equation}\label{eq-deform-c2-2}
|g(v_1)-g(v_2)| = \left| \varphi\left(\frac{A_1}{B_1}\right) -
  \varphi\left(\frac{A_2}{B_2}\right) \right|
\le \|\varphi'\|_\infty \left| \frac{A_1}{B_1} - \frac{A_2}{B_2} \right |
\le 10^{1-N} \left( \frac{|v_1-v_2|}{s_{i+1}} \right). 
\end{equation}
Since $u_1-u_2 = v_1-v_2 + (0,\tilde\varphi(x_1) g(v_1) - \tilde\varphi(x_2) g(v_2))$,
the estimate $|u_1-u_2| \geq 0.99 |v_1-v_2|$ follows from the following bound.
\begin{equation}\begin{split}\label{eq-distancecontrol}
	|\tilde\varphi(x_1) g(v_1) - \tilde\varphi(x_2) g(v_2)|
	& \leq |\tilde\varphi(x_1)-\tilde\varphi(x_2)| \, \| g \|_\infty + \| \tilde\varphi \|_\infty |g(v_1)-g(v_2)| \\
	& \leq 10^{2-N} |x_1-x_2| + 6s_{i+1} \cdot 10^{1-N} \left( \frac{|v_1-v_2|}{s_{i+1}} \right)
		\leq 10^{3-N} |v_1 - v_2|.
\end{split}\end{equation}

Observe too that
\begin{equation}\label{eq-deform-c2-3}
|\tilde\varphi'(x_1)-\tilde\varphi'(x_2)| \leq \|\tilde{\varphi}''\|_\infty \,|x_1-x_2|
\le 10^{2-2N} s_{i+1}^{-1} |v_1-v_2| \leq 10^{3-2N} \alpha(u_1,u_2).
\end{equation}

From \eqref{eq-deform-c2-2} we have $|g(v_1)-g(v_2)| \leq 10^{2-N} \alpha(u_1,u_2)$, which combines with \eqref{eq-deform-c2-3} to get
\begin{equation}\label{eq-deform-c2-4}\begin{split}
|\Delta C_2|
	& \leq |\tilde\varphi'(x_1)| \, |g(v_1)-g(v_2)| + |\tilde\varphi'(x_1)-\tilde\varphi'(x_2)|\,|g(v_2)|\\
	& \leq 10^{2-N} \cdot 10^{2-N} \alpha(u_1,u_2) + 10^{3-2N} \alpha(u_1,u_2) \leq 10^{5-2N} \alpha(u_1,u_2).
\end{split}\end{equation}

Finally, we must bound $|\Delta C_3|$. As $v_1$, $v_2$ both lie above $\gamma_0$, we have
$$
(g\circ(\id\otimes\gamma_k))'(x_k) = E_k \varphi'(A_k/B_k),
$$
where $A_k$ and $B_k$ are as defined above and
\[
E_k = \frac{A_k' B_k - A_k B_k'}{B_k^2},\quad  A_k' =
\gamma_k'(x_k)-\gamma_0'(x_k), \quad B_k' = \gamma_0'(x_k).
\]
Now $\max\{|A_k'|,|B_k'|\} \le \tfrac{1}{100}$, so $|E_k| \leq 3 \cdot
10^{-2-N} s_{i+1}^{-1}$ and $(g\circ(\id\otimes\gamma_k))'(x_k)
\leq 10^{-1-N}s_{i+1}^{-1}$. Since $v_1$ and $v_2$ lie on the same
side of $\gamma_0$, we have
\[
\max\{|A_1-A_2|,|B_1-B_2|\} \le 10 s_{i+1} \alpha(u_1,u_2) \quad \mbox{and} \quad \max\{|A_1'-A_2'|,|B_1'-B_2'|\} \le \alpha(u_1,u_2),
\]
so
\begin{align*}
|E_1-E_2|
& = \left| \frac{A_1' B_1 - A_1 B_1'}{B_1^2} - \frac{A_2' B_2 - A_2
    B_2'}{B_2^2} \right| \\
& \leq 10^{1-4N} s_{i+1}^{-4} \left| (A_1' B_1 - A_1 B_1')B_2^2 -
  (A_2' B_2 - A_2 B_2')B_1^2 \right| \\
& = 10^{1-4N} s_{i+1}^{-4} \left| B_1 B_2^2 (A_1'-A_2') + A_2' B_1
  B_2(B_2-B_1) + A_2 B_1^2(B_2' - B_1') \right. \\
& \qquad \left. +B_1' B_1^2(A_2-A_1) + A_1 B_1'(B_1+B_2)(B_1-B_2) \right| \\
& \le 10^{1-4N} s_{i+1}^{-4} \left( 2 \cdot 10^{1+3N}s_{i+1}^3\alpha(u_1,u_2) +
  3 \cdot 10^{2N}s_{i+1}^3\alpha(u_1,u_2) \right) \le 10^{3-N} s_{i+1}^{-1} \alpha(u_1,u_2).
\end{align*}
Thus $|(g\circ(\id\otimes\gamma_1))'(x_1)-(g\circ(\id\otimes\gamma_2))'(x_2)|$
is equal to
\begin{equation*}\begin{split}
\left| E_1 \,\varphi'\left(\frac{A_1}{B_1}\right) - E_2
  \,\varphi'\left(\frac{A_2}{B_2}\right) \right|
& \le \left| E_1-E_2 \right| \cdot \left|
  \varphi'\left(\frac{A_1}{B_1}\right) \right|  +  |E_2| \cdot \left|
  \varphi'\left(\frac{A_1}{B_1}\right) -
  \varphi'\left(\frac{A_2}{B_2}\right) \right| \\
& \le \left| E_1-E_2 \right| \cdot \| \varphi' \|_\infty + |E_2|
\cdot \| \varphi''\|_\infty \cdot \left|
  \frac{A_1}{B_1}-\frac{A_2}{B_2} \right| \\
& \le 10^{4-N} s_{i+1}^{-1} \alpha(u_1,u_2) + 10^{1-2N} s_{i+1}^{-1} \alpha(u_1,u_2)
  \le 10^{5-N} s_{i+1}^{-1} \alpha(u_1,u_2),
\end{split}\end{equation*}
while
$$
|\tilde\varphi(x_1)-\tilde\varphi(x_2)| \le ||\tilde\varphi'||_\infty |x_1-x_2| \le 10^{2-N} |v_1-v_2|.
$$
Putting all this together,
\begin{equation}\begin{split}\label{eq-deform-c3-3}
|\Delta C_3| &\le |\tilde\varphi(x_1)| \,
\left|(g\circ(\id\otimes\gamma_1))'(x_1)-(g\circ(\id\otimes\gamma_2))'(x_2)
\right| \\
& \qquad \qquad + |\tilde\varphi(x_1) - \tilde\varphi(x_2)| \, \left|
  (g\circ(\id\otimes\gamma_2))'(x_2) \right| \\
& \le 6s_{i+1} \cdot 10^{5-N} s_{i+1}^{-1} \alpha(u_1,u_2) + 10^{2-N} |v_1-v_2| \cdot 10^{-1-N} s_{i+1}^{-1} \\
& \le 10^{6-N} \alpha(u_1,u_2) + 10^{2-2N} \alpha(u_1,u_2) \le 10^{7-N} \alpha(u_1,u_2).
\end{split}\end{equation}
We can now tie all these estimates together. Using \eqref{arclength} again, we estimate
\begin{equation*}\begin{split}
\angle(\dot{\Gamma}_{i+1}(u_1), \dot{\Gamma}_{i+1}(u_2))
&\le \frac\pi2|\dot{\Gamma}_{i+1}(u_1)-\dot{\Gamma}_{i+1}(u_2)| \\
&\le \pi | (h\circ(\id\otimes\gamma_1))'(x_1) - (h\circ(\id\otimes\gamma_2))'(x_2) |.
\end{split}\end{equation*}
This follows from \eqref{eq-geom-lemma} upon noting that
$$
\dot\Gamma_{i+1}(u_k) =
(H\circ(\id\otimes\gamma_k))'(x_k)/|(H\circ(\id\otimes\gamma_k))'(x_k)|
$$
and
$(H\circ(\id\otimes\gamma_k))'(x_k)=(1,(h\circ(\id\otimes\gamma_k))'(x_k))$.

We combine \eqref{eq-deform-c1}, \eqref{eq-deform-c23}, \eqref{eq-deform-c2-5},
\eqref{eq-deform-c3}, \eqref{eq-deform-c2-4}, \eqref{eq-deform-c3-3} and $N \geq 4$ to conclude that
\begin{equation*}\begin{split}
\angle(\dot{\Gamma}_{i+1}(u_1), \dot{\Gamma}_{i+1}(u_2))
 & \le \pi \left( |\Delta C_1| + |\Delta C_2| + |\Delta C_3| \right) \\
 & \le \pi \left( 3\delta_0 a_{i+1}^{2/3} + 4 \cdot 10^{2-N} + 10^{7-N} \right) \alpha(u_1,u_2)
\le \delta_0 \left( \frac{|u_1-u_2|}{2s_{i+1}} \right)^{2/3},
\end{split}\end{equation*}
where the last inequality holds provided that $N=N(\delta_0)$ is chosen large enough. This completes the proof of Lemma \ref{Blemma}.
\end{proof}

\begin{lemma}\label{Clemma}
Condition~\eqref{eq-good-curves-weight} of Definition \ref{def-goodcurves} is satisfied, and $\| w_{i+1} \|_q \leq 2 \| w_i \|_q$ on $Q$.
\end{lemma}

\begin{proof}
Let us write $\| D_{\dot{\Gamma_i}} H(v)\|$ for the magnitude of the
directional derivative of $H$ in the direction of $\dot{\Gamma}_i$ at
the point $v$. Since $H$ is $C^2$ on an open set,
\[
w_{i+1}(H(v)) JH(v) = \| D_{\dot{\Gamma}_i} H(v)\| w_i(v)
\]
for every $v$ in the domain of $H$; here $JH$ denotes the Jacobian of $H$. Thus, for $u \in \spt\Gamma_{i+1}$,
\begin{equation}\label{eq-deform-wt-1}
w_{i+1}(u) = JH^{-1}(u) \ \| D_{\dot{\Gamma}_i} H( H^{-1}(u))\| \ w_i(H^{-1}(u)).
\end{equation}
We want to show that, on any given ball of radius $s_{i+1}$, there is
a constant $A'$ so that $w_{i+1}$ takes values in
$[(1+\delta_0)^{-1}A',(1+\delta_0)A']$ and is $\frac{2}{3}$-H\"older
continuous with constant $\frac{A'\delta_0}{(2s_{i+1})^{2/3}}$.
	
\begin{sublemma}\label{lem-deform-wt-1}
On any ball of radius $s_{i+1}$, $JH^{-1}$ is
$10^{5-2N}s_{i+1}^{-1}$-Lipschitz with values in
$[(1+10^{3-N})^{-1},1+10^{3-N}]$. In particular, $JH^{-1}$ is
$\frac{2}{3}$-H\"older continuous with constant
$\frac{10^{6-2N}}{(2s_{i+1})^{2/3}}$.
\end{sublemma}

\begin{proof}
First, we compute the differential of $H$:
\[
DH(x,y) = \begin{pmatrix}
1 & 0 \\ \tilde{\varphi}'(x) g(x,y) + \tilde{\varphi}(x) g_x(x,y) &
	1 + \tilde{\varphi}(x) g_y(x,y) \end{pmatrix}.
\]
Thus
\begin{equation}\label{eq:JH}
JH = 1+\tilde{\varphi}(x) g_y(x,y).
\end{equation}
Outside $Q$, near the edge of $Q$, or near $\gamma_0$, we have $JH
\equiv 1$, so $JH^{-1} \equiv 1$. It remains to consider the case when
$v_1, v_2$ are in $Q$ and above $\gamma_0$. We see that
\[
g_y(x,y) = \frac{1}{10^Ns_{i+1}-\gamma_0(x)}
\varphi'\left(\frac{y-\gamma_0(x)}{10^Ns_{i+1}-\gamma_0(x)}\right).
\]
Now $10^Ns_{i+1}-\gamma_0$ is $10^{-2}$-Lipschitz and takes values in $[0.99
\cdot 10^N s_{i+1}, 1.01 \cdot 10^N s_{i+1}]$, thus
$(10^Ns_{i+1}-\gamma_0)^{-1}$ is $10^{-1-2N}s_{i+1}^{-2}$-Lipschitz and takes
values in $[0.98\cdot 10^{-N}s_{i+1}^{-1},1.02\cdot 10^{-N}s_{i+1}^{-1}]$. On the other
hand, $\varphi'(\frac{y-\gamma_0(x)}{10^Ns_{i+1}-\gamma_0(x)})$ has size at most
$\tfrac52$ and is Lipschitz with constant
\[
\| \varphi'' \|_\infty \left( 1.01\cdot 10^N s_{i+1} \cdot
  10^{-1-2N}s_{i+1}^{-2} + 1.03 \cdot 10^{-N}s_{i+1}^{-1} \right)
\le 10^{2-N}s_{i+1}^{-1}.
\]
Thus $g_y$ has size at most $10^{1-N}s_{i+1}^{-1}$ and is Lipschitz
with constant
\[
1.02\cdot 10^{-N}s_{i+1}^{-1} \cdot 10^{2-N}s_{i+1}^{-1} + \frac{5}{2} \cdot
10^{-1-2N}s_{i+1}^{-2} \le 10^{3-2N}s_{i+1}^{-2}.
\]
Therefore, $JH$ takes values in $[(1+10^{3-N})^{-1},1+10^{3-N}]$ and
is Lipschitz with constant
\[
6s_{i+1} \cdot 10^{3-2N}s_{i+1}^{-2} + 10^{1-N}s_{i+1}^{-1} \cdot
10^{2-N} \le 10^{4-2N}s_{i+1}^{-1}.
\]
Since $H^{-1}$ is $1.01$-Lipschitz, $JH^{-1}=(JH \circ H^{-1})^{-1}$
takes values in $[(1+10^{3-N})^{-1},1+10^{3-N}]$ and is Lipschitz with
constant $10^{5-2N}s_{i+1}^{-1}$.
\end{proof}

\begin{sublemma}\label{lem-deform-wt-2}
On any ball of radius $s_{i+1}$, $\| D_{\dot{\Gamma}_i} H \circ H^{-1}\|$
is $\frac{2}{3}$-H\"older continuous with constant
$\frac{1}{2}\delta_0(2s_{i+1})^{-2/3}$ and takes values in
$[1/1.01,1.01]$. 
\end{sublemma}

\begin{proof}
Writing $v_k = (x_k, \gamma_k(x_k))$ for $\gamma_k \in \Gamma_i$, we
calculate
\[
\| D_{\dot{\Gamma_i}} H(v_k)\| = \frac{\| DH_{v_k}
  (1,\gamma_k'(x_k))\|}{\| (1,\gamma_k'(x_k))\|}
= \sqrt{\frac{{1+((h \circ (\id \otimes
      \gamma_k))'(x_k))^2}}{{1+(\gamma_k'(x_k))^2}}}.
\]
Thus $\| D_{\dot{\Gamma_i}} H(v_k)\| \in [1/1.01,1.01]$.
Note that $H^{-1}$ is $1.01$-Lipschitz, and that near $z=1$, $\sqrt{z}$ is $1$-Lipschitz. Thus it suffices to show that
\begin{equation}\label{C-Holder}
\frac{{1+((h \circ (\id \otimes \gamma_k))')^2}}{{1+(\gamma_k')^2}} \mbox{ is $\frac{2}{3}$-H\"older continuous}
\end{equation}
with H\"older constant $\tfrac{1}{3}\delta_0(2s_{i+1})^{-2/3}$. To this end, consider the equality
\begin{equation}\label{eq-deform-wt-2}
\frac{1+L_1^2}{1+M_1^2} -\frac{1+L_2^2}{1+M_2^2} =
\frac{(1+M_1^2)(L_1^2-L_2^2)+(1+L_1^2)(M_2^2-M_1^2)}{(1+M_1^2)(1+M_2^2)}.
\end{equation}
In our case, $M_k = \gamma_k'$ is at most $\tfrac1{200}$ and is
$\frac{2}{3}$-H\"older continuous with constant
$3\delta_0a_{i+1}^{2/3}(2s_{i+1})^{-2/3}$, while $L_k = (h \circ (\id
\otimes \gamma_k))'$ is at most $\tfrac1{100}$ and is
$\frac{2}{3}$-H\"older continuous with constant
$\frac{1}{10}\delta_0(2s_{i+1})^{-2/3}$. (This follows from condition
\eqref{eq-good-curves-gamma-dot}, for sufficiently large $N$.)
In view of \eqref{eq-deform-wt-2}, \eqref{C-Holder} is satisfied with H\"older constant
\begin{equation*}
\frac{1}{10}\delta_0(2s_{i+1})^{-2/3} + 3\delta_0a_{i+1}^{2/3}(2s_{i+1})^{-2/3} \le
\frac{1}{3}\delta_0(2s_{i+1})^{-2/3}. \qedhere
\end{equation*}
\end{proof}

Sublemmas \ref{lem-deform-wt-1} and \ref{lem-deform-wt-2} combine with \eqref{eq-deform-wt-1} to show
that $\| w_{i+1} \|_q \leq 2 \|w_i \|_q$ on $Q$.

\begin{sublemma}\label{lem-deform-wt-3}
On a ball $B$ of radius $s_{i+1}$, $A^{-1} w_i\circ H^{-1}$ is
$\frac{2}{3}$-H\"older continuous with constant
$\frac{10^{-N/2}\delta_0}{(2s_{i+1})^{2/3}}$ and takes values in
$[(1+\delta_0)^{-1},1+\delta_0]$, where $A$ is the constant from condition \eqref{eq-good-curves-weight} for $B$.
\end{sublemma}

\begin{proof}
We already know that $A^{-1}w_i$ is $\frac{2}{3}$-H\"older continuous
with constant $\frac{\delta_0}{(2s_i)^{2/3}}$ and takes values in
$[(1+\delta_0)^{-1},1+\delta_0]$. Since $H^{-1}$ is $1.01$-Lipschitz,
we conclude that $A^{-1} w_i\circ H^{-1}$ is $\frac{2}{3}$-H\"older
with constant
\begin{equation*}
1.01^{2/3} \frac{\delta_0}{(2s_i)^{2/3}} \leq \frac{1.01 a_{i+1}^{2/3}
  \delta_0}{(2s_{i+1})^{2/3}} \leq
\frac{10^{-N/2}\delta_0}{(2s_{i+1})^{2/3}}.\qedhere
\end{equation*}
\end{proof}

The following lemma is trivial.

\begin{sublemma}\label{lem-deform-wt-4}
If $f_i: X \ra [M_i^{-1},M_i]$ is $\alpha$-H\"older continuous with
constant $L_i$ for $i\in\{1,2,3\}$, then $g: X \ra
[(M_1M_2M_3)^{-1},M_1 M_2 M_3]$ given by $g=f_1f_2f_3$ is
$\alpha$-H\"older continuous with constant $L_1 M_2 M_3 + M_1 L_2 M_3
+ M_1 M_2 L_3$.
\end{sublemma}


Sublemmas \ref{lem-deform-wt-1} to \ref{lem-deform-wt-4} together with
\eqref{eq-deform-wt-1} combine to show that $A^{-1}w_{i+1}$ takes
values in $[1/1.02,1.02]$ and is $\frac{2}{3}$-H\"older continuous
with constant at most
$$
(2s_{i+1})^{-2/3} \left( 10^{6-2N} (1.01)(1+\delta_0) +
  (1+10^{3-N})\frac{\delta_0}{2}(1+\delta_0) +
  (1+10^{3-N})(1.01)10^{-N/2}\delta_0 \right),
$$
which is bounded above by $\tfrac{3}{4}\delta_0(2s_{i+1})^{-2/3}$
provided $N$ is large enough.

Thus on balls of radius $s_{i+1}$, the ratio of maximum to minimum
values of $w_{i+1}$ (which is the same as the ratio for
$A^{-1}w_{i+1}$) is at most
\[
\left( \frac1{1.02} \right)^{-1} \left(
  \frac{1}{1.02}+\frac{3}{4}\delta_0 \right) = 1+
1.02\frac{3}{4}\delta_0 \le (1+\delta_0)^2.
\]
We can choose $A' \in [A/1.02,1.02A]$ as appropriate for the given
ball to conclude that the H\"older constant for $w_{i+1}$ on the
ball is at most $A\frac{3}{4}\delta_0(2s_{i+1})^{-2/3} \leq
\frac{A'\delta_0}{(2s_{i+1})^{2/3}}$.  This completes the proof of Lemma \ref{Blemma}.

\

Conditions (A), (B) and (C) having been verified for the new measured curve family $(\Gamma_{i+1},\sigma_{i+1})$, the proof of Proposition \ref{good-prop} is now complete.
\end{proof}


\section{Uniformization by round carpets and slit carpets}\label{sec:uniformization}

In this section we prove Corollaries \ref{corC} and \ref{corD} from the introduction, on the existence of round and slit carpets supporting Poincar\'e inequalities.

The proof of Corollary \ref{corC} relies on the following uniformization theorem of Bonk
\cite[Theorem 1.1 and Corollary 1.2]{bonk:uniformization}.

\begin{theorem}[Bonk]\label{bonk-theorem}
Let $\{D_i:i\in I\}$ be a family of pairwise disjoint domains in
$\R^2$ with Jordan curve boundaries $S_i = \partial D_i$. Assume
that the curves $\{S_i\}$ are uniformly relatively separated uniform
quasicircles. Then there exists a quasiconformal map
$f:\R^2\to\R^2$ so that $f(D_i)$ is a disc for all $i\in I$.
In particular, if $T$ is a planar carpet whose peripheral circles are
uniformly relatively separated uniform quasicircles, then $T$ can be
mapped to a round carpet $T'$ by a quasisymmetric homeomorphism $f$.
Furthermore, if $T$ has measure zero, then $f$ is unique up to
post-composition with a M\"obius transformation.
\end{theorem}

A Jordan curve $S\subset\R^2$ is a {\it quasicircle} if it is the
image of $\Sph^1$ under a quasiconformal map of $\R^2$. A family of
Jordan curves $\{S_i\}_{i\in I}$ consists of {\it uniform
quasicircles} if there exists $K\ge 1$ so that each $S_i$ is the
image of $\Sph^1$ under a $K$-quasiconformal map of $\R^2$. Finally, a
family of Jordan curves $\{S_i\}_{i\in I}$ is {\it uniformly relatively separated} if there exists $c>0$ so that
\begin{equation}\label{urs}
\frac{\dist(S_i,S_j)}{\min\{\diam S_i,\diam S_j\}} \ge c \qquad
\forall i,j \in I, i \ne j.
\end{equation}
Theorem \ref{bonk-theorem} is stated in \cite{bonk:uniformization} for
the extended complex plane $\widehat\C$ endowed with the spherical
metric. However, an application of the (conformal) stereographic
projection mapping converts the statement in
\cite{bonk:uniformization} to our formulation.

\subsection{Quasisymmetric uniformizability of $S_\ba$ by round carpets}\label{subsec:QS-uniform-round}

In this subsection we prove Corollary \ref{corC}.

For any $\ba$, the peripheral circles of the carpet $S_\ba$
are uniformly separated. Indeed, when $\ba \in c_0$, the uniform
relative separation condition \eqref{urs} holds in the following
stronger form:
\begin{equation}\label{urs2}
\lim_{\max\{\diam S_i,\diam S_j\} \to 0} \frac{\dist(S_i,S_j)}{\min\{\diam S_i,\diam S_j\}} \to \infty.
\end{equation}
Since the peripheral circles are rigid squares, they are uniform
quasicircles. By Theorem \ref{bonk-theorem}, $S_\ba$ is quasisymmetrically
uniformized by a round carpet $T'$ whenever $\ba\in c_0$.

When $\ba \not\in\ell^2$ (and so $S_\ba$ has zero Lebesgue measure),
$T'$ is rigid, i.e., unique up to the application of a M\"obius
transformation.

When $\ba\in\ell^2$ (and so $S_\ba$ has positive Lebesgue measure),
we observe that $T'=f(S_\ba)$ is an Ahlfors $2$-regular subset of
$\R^2$
(the volume upper bound is trivial; the lower bound follows from the
quasisymmetry of $f$ as in Corollary 3.10 and Remark 3.6(1) of
\cite{tys:analytic}). It now follows from Theorem \ref{thmB} and a
result of Koskela and MacManus \cite[Theorem 2.3]{km:sobolev} that the
round carpet $T'$ supports a $p$-Poincar\'e inequality for some $p<2$.

\subsection{Quasisymmetric uniformizability of $S_\ba$ by slit carpets}

We now turn to the proof of Corollary \ref{corD}.

Let $S_\ba$ be a carpet with $\ba\in\ell^2$. For each $m$, the
interior $S_{\ba,m}^o$ of the precarpet $S_{\ba,m}$ is a finitely
connected domain in the plane. By Koebe's uniformization theorem (see
for instance \cite[V\S2]{gol:functiontheory}), $S_{\ba,m}^o$ can be
conformally uniformized to a parallel slit domain $D_m$. Now the
identity map between the Euclidean metric and the internal metric
$\delta$ on $D_m$ is conformal, hence $S_{\ba,m}^o$ is conformally
equivalent to $(D_m,\delta)$.

By Proposition \ref{properties-of-mu} and the proof of Theorem
\ref{thmB}, the precarpets $S_{\ba,m}$ are Ahlfors $2$-regular and
support a $p$-Poincar\'e inequality for any $p>1$ with constants
independent of $m$. In particular, such precarpets are $2$-Loewner in
the Euclidean metric; since they are quasiconvex (with constant
independent of $m$) they are also $2$-Loewner in the internal metric.

It is straightforward to check that the domains $(D_m,\delta)$ are
LLC. We now appeal to a theorem of Heinonen \cite[Theorem
6.1]{hei:john} which asserts that any quasiconformal map from a
bounded domain which is Loewner in the internal metric to a bounded
LLC domain is quasisymmetric. The preceding result is quantitative in
the usual sense; in our situation this implies that the relevant
quasisymmetric distortion function is independent of $m$. Moreover,
the target domains $(D_m,\delta)$ are uniformly Ahlfors $2$-regular
by an argument similar to that used in
subsection \ref{subsec:QS-uniform-round}.

Passing to the limit as $m\to\infty$, we obtain a quasisymmetric map
from $S_\ba$ onto an Ahlfors $2$-regular parallel slit carpet. To
complete the proof, we again use \cite[Theorem 2.3]{km:sobolev} to
conclude that the target carpet supports a $p$-Poincar\'e inequality
for some $p<2$.

\bibliographystyle{acm}
\bibliography{biblio}
\end{document}